\documentclass[a4paper]{amsart}
\newif\ifcomments
\usepackage[margin=3.5cm]{geometry}
\usepackage{amsmath, amsthm, amsfonts,stmaryrd,amssymb}
\usepackage{accents}
\usepackage{tikz-cd}
\tikzcdset{row sep/normal = 1cm}
\tikzcdset{column sep/normal = 1cm}
\usepackage{mathrsfs}                      % Nice \mathscr fonts
\usepackage{mathtools}                     % \mathrlap
\usepackage{booktabs}
\usepackage{todonotes}
\usepackage[all]{xy}
\usepackage[utf8]{inputenc}
\usepackage{mathtools}                     % Extendable arrows with text
\usepackage{todonotes} 
\usepackage{bbm}                         % For \mathbbm{1}
\usepackage{framed}
\usepackage{graphicx}   
\usepackage{hyperref} % crossreferences
\usepackage{xcolor}
\hypersetup{
    colorlinks=true,
    linkcolor=blue,  
    urlcolor=cyan,%,
    pdftitle={},
    pdfauthor={},
} 
\usepackage{theoremref}
\usepackage{subfig}

\numberwithin{equation}{section}
\theoremstyle{plain}
\newtheorem{theorem}{Theorem}[section]

\newtheorem{remark}[theorem]{Remark}
\newtheorem{lemma}[theorem]{Lemma}

\newtheorem{proposition}[theorem]{Proposition}

\theoremstyle{definition}

\def\hatt{}
\def\ed{\mathrm{d}}

\def\exp{\operatorname{exp}}

\def\modif#1{#1}

\let\mc=\mathcal

\let\mr=\mathrm

\begin{document}
\date{\today}
\title[A mixed finite element discretization of optimal transport]{A mixed finite element discretization of dynamical optimal transport}

\author[A. Natale]{Andrea Natale}
\address{Andrea Natale (\href{mailto:andrea.natale@u-psud.fr}{\tt andrea.natale@inria.fr}) Inria, Project team Rapsodi, Univ. Lille, CNRS, UMR 8524 - Laboratoire Paul Painlevé, F-59000 Lille, France}
\author[G. Todeschi]{Gabriele Todeschi}
\address{Gabriele Todeschi (\href{mailto:gabriele.todeschi@inria.fr}{\tt gabriele.todeschi@inria.fr}):  Inria, Project team Mokaplan, Universit\'e Paris-Dauphine, PSL Research University, UMR CNRS 7534-Ceremade
}

\maketitle

\begin{abstract}
In this paper we introduce a new class of finite element discretizations of the quadratic optimal transport problem based on its dynamical formulation. These generalize to the finite element setting the finite difference scheme proposed by Papadakis et al. [SIAM J Imaging Sci, 7(1):212–238,2014].
We solve the discrete problem using a proximal splitting approach and we show how to modify this in the presence of regularization terms which are relevant for physical data interpolation.
\end{abstract}

\section{Introduction}

Optimal transport provides a convenient framework for density interpolation as a convex optimization problem. Its most remarkable feature is its sensitivity to horizontal displacement, which generally allows one to retrieve translations when interpolating between two densities. This property has motivated the application of optimal transport to many imaging problems, especially in the context of physical sciences and fluid dynamics. A typical example comes from satellite image interpolation in oceanography. In this case, one is interested in reconstructing the evolution of a quantity of interest such as Sea Surface Temperature (SST) or Sea Surface Height (SSH) between two given observations.
%However, the optimal transport interpolation does not involve any regularizing effect. For example, given two smooth densities on a smooth but non-convex domain, it may only be possible to define the interpolation between the two as a measure. 
As highlighted in \cite{hug2015multi}, for this type of applications one needs to include appropriate regularization terms to avoid the appearance of unphysical phenomena such as mass concentration in the reconstructed density evolution.

%This property has motivated the application of optimal transport to many imaging problems including satellite image interpolation. In this case, one is interested in reconstructing the evolution of a quantity of interest such as Sea Surface Temperature (SST) or Sea Surface Height (SSH) between two given observations. As highlighted in \cite{hug2015multi}, for this type of applications one needs to include appropriate regularization terms to avoid the appearance of unphysical phenomena such as mass concentration in the reconstructed density evolution. % With the same motivation in mind, we propose a finite element approach to solve numerically the optimal transport problem. This leads to numerical methods that are easy to implement and modify with different regularization terms, that are numerically stable on non-structured grids, and whose discrete solutions can be proved to converge to the exact ones.      

In this paper we propose a finite element approach to solve the dynamical formulation of optimal transport with quadratic cost on unstructured meshes (and therefore can be easily implemented on complex domains) and that can be easily modified to include different type of regularizations which are relevant for the dynamic reconstruction and interpolation of physical quantities. For some choices of finite element spaces, using the framework introduced in \cite{lavenant2019unconditional}, we can prove convergence of our discrete solutions to the ones of the continuous problem.

% Static formulation ->

% the problem
The dynamical formulation of optimal transport inspired some of the first numerical methods for this problem. This reads as follows: given two probability measures $\rho_0, \rho_1 \in \mc{P}(D)$ on a compact domain $D\subset \mathbb{R}^d$, find the curve $t\in[0,1] \mapsto \rho(t,\cdot) \in \mc{P}(D)$ which solves
\begin{equation}\label{eq:dynamic}
\inf_{\rho,v} \left \{ \int_0^1\int_D  \frac{|v(t,\cdot)|^2}{2}\ed \rho(t,\cdot) \ed t \,;\, \partial_t{\rho}+ \mathrm{div}_x(\rho v) =0 ,\rho(0) = \rho_0,\,\rho(1) = \rho_1 \right\}
\end{equation} 
where $v:[0,1]\times D \rightarrow \mathbb{R}^d$ is a time-dependent velocity field on $D$ tangent to the boundary $\partial D$, and $|\cdot|$ denotes the Euclidean norm. In other words, problem \eqref{eq:dynamic} selects the curve of minimal kinetic energy with fixed endpoints $\rho_0$ and $\rho_1$.

Benamou and Brenier \cite{benamou2000computational} realized that introducing the momentum $m \coloneqq \rho  v$, problem \eqref{eq:dynamic} can be recast into a convex optimization problem in the variables $(\rho, m)$, with a linear constraint, since the continuity equation becomes
\begin{equation}\label{eq:continuity}
\partial_t{\rho} + \mathrm{div}_x m = 0 \,.
\end{equation}
If we define $\sigma \coloneqq (\rho,m)$, regarded as a measure on $[0,1]\times D$, this constraint is equivalent to $\mathrm{div}\,\sigma = 0$, where now $\mathrm{div}$ denotes the divergence operator on the space-time domain $[0,1] \times D$. Introducing the dual variable $q = (a,b)$ where $a \in C([0,1]\times D)$ and $b \in C([0,1]\times D;\mathbb{R}^{d})$, the kinetic energy minimized in \eqref{eq:dynamic} can be written in the form 
\[
\sup_q \left \{   \int_0^1 \int_D q \cdot \ed \sigma \,;\, a+\frac{|b|^2}{2} \leq 0  \right\}.
\]
Combining this expression with \eqref{eq:dynamic} we obtain a saddle point problem in the variables $(q,\sigma)$ with a nonlinear constraint on $q$ and a linear one on $\sigma$.
 
The numerical method proposed in \cite{benamou2000computational} involves discretizing $q$ and $\sigma$ by their values on a regular grid, and expressing the constraint on $\sigma$ via a Lagrange multiplier; then the dual problem can be solved by an Augmented Lagrangian ADMM approach, optimizing separately in $q$ and the Lagrange multiplier and then performing a gradient descent step on $\sigma$. Disregarding the discretization in space-time, the convergence of the method has been studied in \cite{guittet2003time,hug2017convergence}. The same approach was used to discretize different problems related to optimal transport (e.g., gradient flows \cite{benamou2016augmented}, mean field games \cite{benamou2015augmented}, unbalanced optimal transport \cite{gallouet2019unbalanced}) using a finite element discretization in space-time. Importantly, in these cases the numerical method is obtained by discretizing the several steps of the augmented Lagrangian approach rather than as a discrete optimization algorithm. This implies that in general it is difficult to establish the convergence of the discrete algorithms. Moreover,  for these type of methods, convergence results towards the continuous solutions with mesh refinement are only available for specific settings (e.g., the $L^1$-type optimal transport problems studied in \cite{igbida2017augmented}), but they are not available for the optimal transport problem \eqref{eq:dynamic}. 

%Similar augmented Lagrangian approaches have been applied to a closely related class of optimal transport problems associated with the $L^1$ cost or even Finsler distances. In these cases convergence results have

Papadakis, Peyré, and Oudet proposed in \cite{papadakis2014optimal} a staggered finite difference discretization on regular grids of the optimal transport problem \eqref{eq:dynamic}, and they considered different proximal splitting algorithms to solve it. The computational bottleneck for these methods as well as for the original augmented Lagrangian approach is the \modif{projection onto the space of divergence-free vector fields $\sigma$, which amounts to solving a Poisson equation at each iteration.} This however can be avoided by exploiting the Helmholtz decomposition of vector fields, as recently showed in \cite{henry2019primal}, or adding regularization terms as in \cite{li2018computations}. Recently, Carrillo and collaborators \cite{carrillo2019primal} proposed a finite difference scheme similar to that in \cite{papadakis2014optimal} (in the context of the discretization of Wasserstein gradient flows), for which they could also prove its convergence with mesh refinement, but only upon strong regularity assumptions on the solutions of the continuous problem.

%To the best of our knowledge, also for these methods no results are available in terms of convergence of the discrete solutions towards those of problem \eqref{eq:dynamic}, with mesh refinement.

In \cite{lavenant2018dynamical} a numerical scheme was proposed using tools from finite element and finite volume methods, where one explicitly constructs a duality structure for the discrete variables. Later Lavenant \cite{lavenant2019unconditional} proved convergence  of this scheme, unconditionally with respect to the time/space step size, to the solutions of the optimal transport problem, proposing a general framework for convergence of discretizations of problem \eqref{eq:dynamic} between two arbitrary probability measures. This filled a critical gap for the analysis of discrete dynamical transport models, since previously convergence results were only known in case of sufficiently smooth solutions (as in \cite{carrillo2019primal}) or conditional to the relative time/space step sizes (e.g., in the context of finite volume methods, combining the results in \cite{erbar2020computation} and \cite{gladbach2018scaling}).  

%Finally, we mention that mixed finite element techniques have already been employed to solve $L^1$ optimal transport problems using a PDE approach in \cite{barrett2007mixed}.

\subsection{Contributions and structure of the paper}

In this paper we propose a mixed finite element discretization of \eqref{eq:dynamic} which generalizes to the finite element setting the finite difference scheme proposed by Papadakis et al.\ \cite{papadakis2014optimal}. We derive our method by discretizing a saddle point formulation of the dynamic optimal transport problem on Hilbert spaces, where one looks for a solution $(q, \sigma) \in L^2([0,1]\times D;\mathbb{R}^{d+1})^2$. Nonetheless, we stress that the method we obtain is still well-defined when the initial and final data are arbitrary probability measures. By using $H(\mathrm{div})$-conforming spaces for the variable $\sigma$, we are able to construct discrete solutions that satisfy exactly the weak form of the continuity equation \eqref{eq:continuity}.

\modif{Using the framework of \cite{lavenant2019unconditional}, we also show that our discrete solutions, for specific choices of finite element spaces, converge towards the solutions of the optimal transport problem between two arbitrary measures, and therefore even when the solution $\sigma$ is only a measure (see Theorem \ref{th:convergence}). Such a result carries over also to a slight modification of the finite difference scheme proposed in  \cite{papadakis2014optimal}, which can be viewed as a particular instance of our discretization on a uniform quadrilateral grid (see Remark \ref{rm:convpeyre}).}

\modif{
Finally, as in \cite{papadakis2014optimal}, we solve the discrete problem using a proximal splitting algorithm \cite{pock2009algorithm}. Importantly, this is not only a discretization of the same algorithm applied to the continuous saddle point formulation as in previous works, but also a genuine optimization scheme applied to the finite dimensional problem. Furthermore, we observe numerically that the proposed modification of the finite difference scheme in \cite{papadakis2014optimal} (which we derived to prove convergence with mesh refinement) also yields a remarkable speedup for the convergence of the proximal splitting algorithm itself, keeping approximately the same computational cost per iteration.}

The paper is structured as follows. We establish the notation in Section \ref{sec:notation}. In Section \ref{sec:dynamical} we give the precise formulation of problem \eqref{eq:dynamic} and describe the proximal splitting algorithm applied to the continuous problem in the Hilbert space setting. In Section \ref{sec:fe} we introduce and discuss the main finite element tools we use for our method. In Section \ref{sec:discrete} we define our finite element discretization of problem \eqref{eq:dynamic} and state the convergence result. In Section \ref{sec:proximal} we detail the steps required for solving our discrete optimal transport problem with a proximal splitting algorithm. In Section \ref{sec:regularization} we describe how to introduce regularization terms in the formulation. Finally in Section \ref{sec:num} we present some numerical results.

% what has been done 

% what we do

% plan of the paper

\section{Notation}\label{sec:notation}
Throughout the paper we will denote by $D\subset \mathbb{R}^d$ a convex polytope, with $d\in\{2,3\}$, and by $\Omega \coloneqq [0,1] \times D$ the space-time domain. For differential operators such as $\nabla$ or $\mathrm{div}$, we use the subscript $x$ to emphasize that these are defined on $D$ rather than $\Omega$, but we will drop this subscript when this is clear from the context.  

We use the standard notation for Sobolev spaces on $D$ or $\Omega$. In particular, $L^p(D;\mathbb{R}^d)$ denotes the space of functions $f:D\rightarrow \mathbb{R}^d$ whose Euclidean norm $|f|$ is in $L^p(D)$. We use a similar notation for functions taking values on a subset $K \subset \mathbb{R}^d$, or defined on $\Omega$. We denote by $H(\mathrm{div};D)$ the space of vector fields $f:D \rightarrow \mathbb{R}^d$ in $L^2(D;\mathbb{R}^d)$ whose divergence is in $L^2(D)$. Similarly, $H(\mathrm{div}; \Omega)$ the space of vector fields $f: \Omega \rightarrow \mathbb{R}^{d+1}$ in $L^2(\Omega;\mathbb{R}^{d+1})$ whose divergence is in $L^2(\Omega)$.

Finally, we denote by $\mc{M}(D)$ the set of finite signed measures on $D$,  by $\mc{M}_+(D) \subset \mc{M}(D)$ the convex subset of positive measures; by $\mc{P}(D) \subset \mc{M}_+(D)$ the set of positive measures of total mass equal to one; and by $C(D)$ the space of continuous functions on $D$. We use a similar notation for the spaces of measures and continuous functions on $\Omega$.  We use $\langle \cdot,\cdot \rangle$ to denote either the duality pairing between measures and continuous functions or the $L^2$ inner product, on either $D$ or $\Omega$, according to the context.    

\section{Dynamical formulation of optimal transport}\label{sec:dynamical}

The dynamical optimal transport problem \eqref{eq:dynamic} can be formulated as a saddle point problem on the space of measures $\sigma \coloneqq (\rho,m)\in \mc{M}(\Omega) \times \mc{M}(\Omega)^{d}$. This can be written as follows
\begin{equation}\label{eq:BB}
\inf_{\sigma \in \mc{C}} \mc{A}(\sigma) ,\quad \mc{A}(\sigma) \coloneqq \sup_{q \in C(\Omega; K)} \langle q, \sigma\rangle,
\end{equation}
where $\mc{C}$ is the set of measures $\sigma \in \mc{M}(\Omega)^{d+1}$ satisfying $\mr{div}\, \sigma = 0$ in distributional sense with boundary conditions 
\begin{equation}\label{eq:boundary}
\sigma \cdot n_{\partial \Omega} = \mc{X} \,, \quad  \mc{X} \coloneqq 
\left \{ 
\begin{array}{ll} 
\rho_0 & \text{ on } \{0\} \times D,\\ 
\rho_1 & \text{ on } \{1\} \times D,\\
0 & \text{ otherwise, } 
\end{array}
\right.
\end{equation}   
with $\rho_0,\rho_1 \in \mc{P}(D)$, and where $C(\Omega; K)$ is the space of continuous functions on $\Omega$ taking value in the convex set
\begin{equation}\label{eq:K}
K \coloneqq \left \{ (a,b) \in  \mathbb{R}\times \mathbb{R}^{d} \,; \, a +\frac{|b|^2}{2} \leq 0 \right \}.
\end{equation}
It will be convenient to treat time and space as separate variables. In particular we will also use the action defined by
\[
A(\rho,m) \coloneqq \sup_{(a,b) \in C(D; K)} \langle \rho,a\rangle + \langle m,b\rangle\,,
\]
for any $(\rho,m) \in \mathcal{M}(D)^{d+1}$. Then, $A(\rho,m)$ is finite if and only if $m$ has a density with respect to $\rho$ and in that case $A(\rho,m) = \int_D B(\rho,m) $, where $B:\mathbb{R} \times \mathbb{R}^d \rightarrow [0, +\infty]$ is the function given by
\[
B(a,b) \coloneqq \left\{
\begin{array}{ll}
\frac{|b|^2}{2a} & \text{if } a>0, \\
0 & \text{if } a=0, b=0, \\
+\infty & \text{if } a =0,  b\neq 0 \text{ or } a<0\,.
\end{array}
\right.
\]
Due to the definition of the function $B$, any saddle point of problem (\ref{eq:BB}) must satisfy $\rho\ge0$.

The value of the infimum of problem \eqref{eq:BB} coincides with $W^2_2(\rho_0,\rho_1)/2$, where $W_2(\cdot,\cdot)$ denotes the Wasserstein distance associated with the $L^2$ cost (see Theorem 5.28 in \cite{santambrogio2015optimal}).  Moreover the infimum itself is attained by a measure $\sigma=(\rho,m)$, where $\rho$ is known as the Wasserstein geodesic between $\rho_0$ and $\rho_1$ (see proposition 5.32 in \cite{santambrogio2015optimal}).
We refer the reader to \cite{santambrogio2015optimal} for more details on the links between the dynamical formulation and the Wasserstein distance.

\subsection{Hilbert space setting and proximal splitting}\label{sec:ham}
Before discussing the discretization of problem \eqref{eq:BB}, we review its reformulation on Hilbert spaces, and discuss the convergence of the proximal splitting algorithm.  

\begin{proposition}[Guittet \cite{guittet2003time}; Hug et al. \cite{hug2017convergence}] \label{prop:hilbert}
Suppose $\rho_0,\rho_1 \in L^2(D)$. Then problem \eqref{eq:BB} is equivalent to
\begin{equation}\label{eq:BB1}
\inf_{\sigma \in \mc{C}} \sup_{q \in L^2(\Omega; K)} \langle q, \sigma\rangle\,,
\end{equation}
where $\mc{C}$ is the set of functions $\sigma \in H(\mr{div};\Omega)$ satisfying $\mr{div}\, \sigma = 0$ in weak sense with boundary conditions given by \eqref{eq:boundary}. Moreover, assuming that $\mathrm{supp}(\rho_0) \cup \mathrm{supp}(\rho_1) \subset \accentset{\circ}{D}$, there exists a saddle point $(\sigma^*,q^*) \in \mc{C} \times L^2(\Omega;K)$ solving problem \eqref{eq:BB1}.
\end{proposition}

The equivalence of problem \eqref{eq:BB1} to \eqref{eq:BB} can be easily deduced by a regularization argument on $\sigma$ and then applying Lusin's theorem as in Proposition 5.18 in \cite{santambrogio2015optimal}. The proof for the existence of a saddle point problem is more delicate and can be found in \cite{hug2017convergence}.

In order to apply a proximal splitting algorithm to solve problem \eqref{eq:BB1}, we first write it in the form
\begin{equation}\label{eq:BB2}
\inf_{\sigma \in L^2(\Omega;\mathbb{R}^{d+1}) } \sup_{q \in L^2(\Omega; \mathbb{R}^{d+1})} 
\langle q, \sigma\rangle + \iota_{\mc{C }} (\sigma) - \iota_{\mc{K}} (q)\,,
\end{equation}
where $\iota$ denotes the convex indicator function and
\[
\mc{K} \coloneqq L^2(\Omega;K) = \{ q \in L^2(\Omega;\mathbb{R}^{d+1})\,; \, q \in K ~ \text{a.e.} \}.
\]
Note in particular that $\mc{C}$ and $\mc{K}$ are closed convex sets of $L^2$. 

We apply to \eqref{eq:BB2} the primal-dual projection algorithm proposed in \cite{pock2009algorithm}. In particular,  given $\tau_1, \tau_2>0$ and an admissible $(\sigma^0, q^0)\in \mc{C}\times \mc{K}$, we define the sequence $\{(\sigma^k,q^k)\}_k$ by the two-step algorithm:
\begin{subequations}\label{eq:PS}
\begin{align}
%\begin{array}{l}
\text{\textbf{Step 1}}: \qquad & \sigma^{k+1} = P_{\mc{C}} (\sigma^k- \tau_1 q^k)\,.\\
\text{\textbf{Step 2}}: \qquad &q^{k+1} = P_{\mc{K}}(q^k +\tau_2(2 \sigma^{k+1}-\sigma^k))\,.
\end{align}
\end{subequations}
where $P_{\mc{C}}$ and $P_{\mc{K}}$ are the $L^2$ projections on the closed convex sets $\mc{C}$ and $\mc{K}$, respectively. The projection onto $\mc{C}$ amounts to computing the Helmholtz decomposition of $\sigma^k- \tau_1 q^k$ and selecting the divergence-free part, whereas the projection onto $\mc{K}$ is a pointwise projection applied to a representative of $q^k +\tau_2(2 \sigma^{k+1}-\sigma^k)$.

The proof of convergence in  \cite{pock2009algorithm} holds also in our setting. More precisely, the following convergence theorem holds.
\begin{theorem}[Pock et al.\ \cite{pock2009algorithm}]\label{th:cp}
If $\tau_1 \tau_2 <1$ then $(\sigma^k, q^k)\rightarrow (\sigma^*, q^*)\in \mc{C}\times \mc{K}$ which solves \eqref{eq:BB1}.
\end{theorem}

Discretizing problem \eqref{eq:BB2}, and consequently the proximal splitting algorithm \eqref{eq:PS}, with finite elements requires choosing finite-dimensional spaces for $\sigma$ and $q$ so that the steps in \eqref{eq:PS} are well-posed and \modif{computationally feasible}. However, satisfying these requirements is not enough to guarantee convergence of the discrete solutions to the ones of the infinite dimensional problem. Hereafter we will identify a class of finite element spaces for which the theory developed in \cite{lavenant2019unconditional} applies, which allows us to deduce convergence to the solutions of problem \eqref{eq:BB}, i.e.\ even when $\rho_0$ and $\rho_1$ are arbitrary probability measures and the Hilbert space setting presented in this section is not well-defined.

\section{Mixed finite element setting}\label{sec:fe}

\subsection{Finite element spaces on $D$}\label{sec:femd}
We recall that $D$ is a convex polytope in $\mathbb{R}^d$, with $d\in\{2,3\}$. We consider a triangulation of $D$ which we denote $\mc{T}_h$, i.e. a decomposition of $D$ in either simplicial or quadrilateral \modif{(disjoint)} elements, where $h$ is the maximum diameter of the elements in $\mc{T}_h$. We assume that there exists a constant $C_{mesh}$ such that
\begin{equation}\label{eq:Cmesh}
|h|^d \leq C_{mesh} |T|\,, \quad \forall\, T\in \mc{T}_h\,.
\end{equation}
This implies that the mesh is quasiuniform, meaning that the ratio of any two element diameters is uniformly bounded by a constant depending only on $C_{mesh}$, and shape-regular,
that is, for each element $T\in \mc{T}_h$, the ratio of its diameter and the diameter of the largest inscribed ball is uniformly bounded by a constant depending only on $C_{mesh}$ (see, e.g., \cite{arnold2006finite}). %meaning that the ratio of the diameter of each element $T\in \mc{T}_h$ and the diameter of the largest inscribed ball is uniformly bounded by a constant depending only on $C_{mesh}$ (see, e.g., \cite{arnold2006finite}).
 
For any $T\in \mc{T}_h$, we denote by $\mc{P}_k(T)$ the space of polynomials of degree up to $k$ on $T$. If $T$ is a quadrilateral element, i.e., in general, if $T$ is obtained by an affine transformation $\phi:I^d \rightarrow T $ where $I$ is the unit interval, then we define $\mc{P}_{k_1,\ldots k_d}(I^d)\coloneqq \mc{P}_{k_1}(I) \otimes \ldots \otimes \mc{P}_{k_d}(I)$ and $\mc{P}_{k_1,\ldots k_d}(T) \coloneqq \mc{P}_{k_1,\ldots k_d}(I^d) \circ \phi^{-1}$.

We now define the finite element spaces $Q_h$ and $V_h$ which will serve to construct approximations of the density $\rho$ and the momentum $m$, respectively. We set
\[ 
Q_{h} \coloneqq \{ \varphi \in L^{2}(D)\,;\, \varphi|_T \in \mc{P}_0(T), ~ \forall\, T\in \mc{T}_h\},
\]
\[
V_{h} \coloneqq \{ v \in H(\mr{div};D) \,;\, v|_T \in V_h(T), ~ \forall\,  T\in \mc{T}_h\}.
\]
where $V_h(T)$ is the so-called shape function space. We distinguish two cases:
\begin{enumerate}
\item for simplicial elements (triangles or tetrahedrons), we take $V_h(T)$ to be either  
\[
\mc{RT}_0(T) \coloneqq \{ v = v_0 + v_1 \hat{x} \,;\, v_0 \in (\mc{P}_0(T))^d\,, v_1\in\mc{P}_0(T) \} \subset (\mc{P}_1(T))^d,
\] 
where $\hat{x} = (x_1,\ldots,x_d)\in (\mc{P}_1(T))^d$,
which generates the lowest order Raviart-Thomas space; or $\mc{BDM}_1(T) = (\mc{P}_1(T))^d$, which generates the lowest order Brezzi-Douglas-Marini $H(\mr{div})$-conforming space;
\item for quadrilateral elements, we set $T = \phi( I^d)$, where $I$ is an interval and $\phi$ an affine transformation, and we take $V_h(T)$ to be the tensor product space which generates the lowest order Raviart-Thomas space on quadrilateral elements. This is defined as follows: 
\[
{\mc{RT}}_{[0]}(T) \coloneqq \left \{ 
\begin{array}{ll} 
\mc{P}_{1,0}(T) e_1 + \mc{P}_{0,1}(T) e_2 & \text{if } d =2\,,\\
\mc{P}_{1,0,0}(T) e_1 + \mc{P}_{0,1,0}(T) e_2 + \mc{P}_{0,0,1}(T) e_3& \text{if } d =3\,,
\end{array}
\right.
\]
where $\{e_i\}_i$ is the basis for $\mathbb{R}^d$ aligned with the edges of $T$.
\end{enumerate}

In other words, the space $V_h$ is chosen as one of the standard lowest order $H(\mr{div})$-conforming spaces. In fact, the property of being piece-wise linear will be crucial in the following, namely to prove the convergence result in Theorem \ref{th:convergence} (see, in particular, Proposition \ref{prop:properties} in the appendix). A graphical representation of the degrees of freedom associated with these spaces is shown in figure \ref{fig:dofs}.

Importantly, with the choices mentioned above, one can define projection operators $\Pi_{Q_h}: L^2(D) \rightarrow Q_h$ and \modif{ $\Pi_{V_h}: \mathcal{V}_D \subset H(\mr{div};D)\rightarrow V_h$ that commute with the divergence operator \cite{arnold2006finite,boffi2013mixed}, where $\mathcal{V}_D$ is a dense subset of sufficiently smooth vector fields. 
By an appropriate regularization procedure of such operators (see, e.g., Section 5.4 in \cite{arnold2006finite}), one can construct bounded projections $\tilde{\Pi}_{Q_h}: L^2(D) \rightarrow Q_h$ and $\tilde{\Pi}_{V_h}: H(\mr{div};D)\rightarrow V_h$ satisfying a similar property.}
In other words, the following diagram commutes
\[
\centering
\begin{tikzcd}
H(\mr{div};D) \arrow[d,"\tilde{\Pi}_{V_h}"]\arrow[r, "\mr{div}"] & L^2(D)\arrow[d,"\tilde{\Pi}_{Q_h}"] \\V_h\arrow[r, "\mr{div}"] & Q_h
\end{tikzcd}
\]
As a consequence, the divergence operator is surjective onto $Q_h$ when restricted on $V_h$, i.e.\ $\mr{div}\, V_h = Q_h$.
Finally, we let $Q_h^+\subset Q_h$ the convex subset of non-negative piecewise constant functions.

\begin{remark}\label{rem:interpolators}
\modif{For the proof of Theorem \ref{th:convergence} in the appendix, we will consider as commuting projections $\Pi_{V_h}$ and $\Pi_{Q_h}$ the canonical projections defined in Section 5.2 of \cite{arnold2006finite}. Here, we will only need the explicit definition of $\Pi_{Q_h}$,} which is given by 
\begin{equation}\label{eq:PiQh}
\Pi_{Q_h} \rho |_T = \frac{1}{|T|} \int_T \rho \,,
\end{equation}
for any $T\in \mc{T}_h$. Note, in particular, that  $\Pi_{Q_h}$ is well-defined on $\mc{M}(D)$ and its restriction on $\mc{M}_+(D)$ is surjective onto $Q_h^+$. 
%
%
% The projection operators $\Pi_{V_h}$ and $\Pi_{Q_h}$ that we will use in the following (namely, for the proof of Theorem \ref{th:convergence} in the appendix) are the canonical projections described in Section 5.2 of \cite{arnold2006finite}.  Here, we will only need the explicit definition of $\Pi_{Q_h}$,} which is given by 
% \begin{equation}\label{eq:PiQh}
% \Pi_{Q_h} \rho |_T = \frac{1}{|T|} \int_T \rho \,,
% \end{equation}
% for any $T\in \mc{T}_h$. Note, in particular, that  $\Pi_{Q_h}$ is well-defined on $\mc{M}(D)$ and its restriction on $\mc{M}_+(D)$ is surjective onto $Q_h^+$. 
\end{remark}

\begin{figure}
\centering
\subfloat[{$\mc{RT}_0$}]{\includegraphics[width=.22\textwidth]{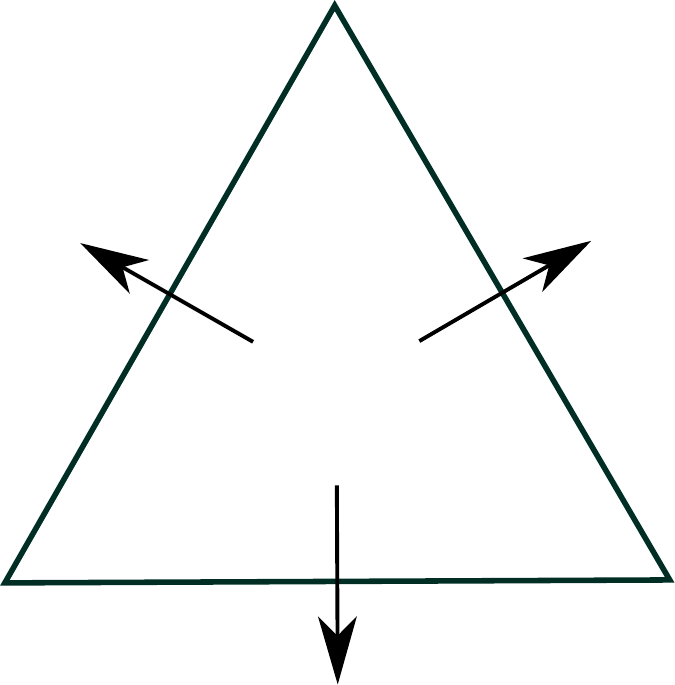}} \qquad
\subfloat[{$\mc{BDM}_1$}]{\includegraphics[width=.22\textwidth]{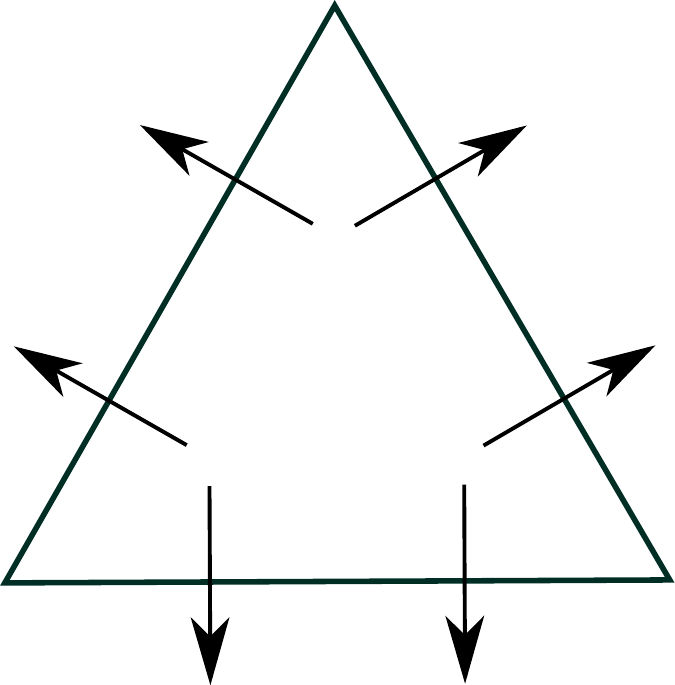}} \qquad
\subfloat[{${\mc{RT}}_{[0]}$}]{\includegraphics[width=.22\textwidth]{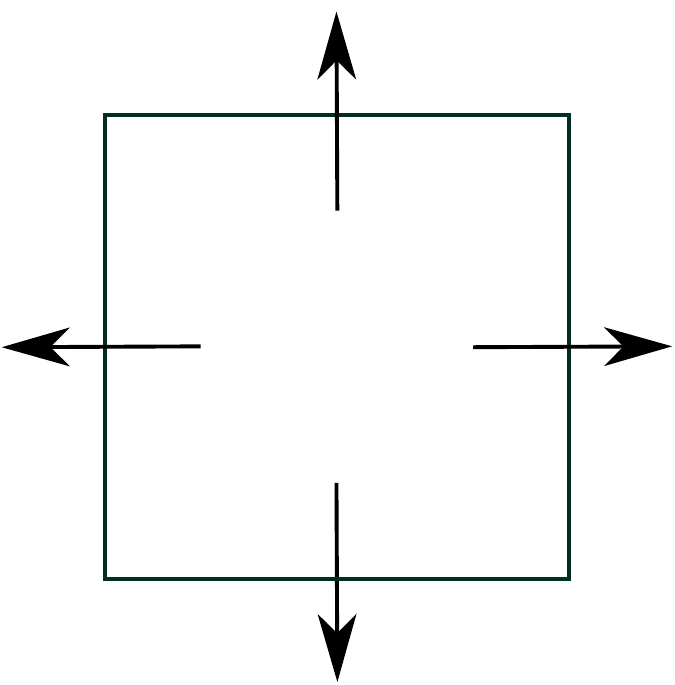}}
\caption{Degrees of freedom for different choices of shape function space $V_h(T)$}
\label{fig:dofs}
\end{figure}

\subsection{Finite element spaces on $\Omega$}

We now introduce finite element spaces on the space-time domain $[0,1]\times D$. We first define a decomposition $\hatt{\mc{T}}_{h,\tau}$, obtained by a tensor product construction. In other words, we assume that $\hatt{\mc{T}}_{h,\tau}$ is obtained by tensor product of a triangulation $\mc{T}_h$ of $D$ and a decocomposition of $[0,1]$ of maximum size $\tau$, so that any element $S\in \hatt{\mc{T}}_{h,\tau}$ is of the form $S = [t_0,t_1] \otimes T$ where $T\in \mc{T}_h$.

We now define the finite element spaces $F_{h,\tau}$ and $Z_{h,\tau}$ on the space-time domain. The space $Z_{h,\tau}$ will be constructed using the standard tensor product construction based on the spaces $Q_h$ and $V_h$ defined on $D$, and continuous $\mc{P}_1$ and discontinuous $\mc{P}_0$ spaces on $[0,1]$. In our discretization, the space-time vector field $(\rho,m)$ will be an element of $Z_{h,\tau}$ whereas $F_{h,\tau}$ will be the space of discrete Lagrange multipliers associated with the continuity equation, which is equivalent to the constraint that the space-time divergence of $(\rho,m)$ is zero. 

More precisely, we define 
\[
F_{h,\tau} \coloneqq \{ \phi\in L^{2}(\Omega)\,;\, \phi|_S \in \mc{P}_0(\hatt{S}), ~ \forall\, \hatt{S}\in \hatt{\mc{T}}_{h,\tau}\},
\]
\[
Z_{h,\tau} \coloneqq \{ v \in H(\mr{div};\Omega) \,;\, v|_S \in Z_{h,\tau}(\hatt{S}), ~ \forall\,  \hatt{S}\in \hatt{\mc{T}}_{h,\tau}\}.
\]
For $\hatt{S} = [t_0,t_1] \otimes T$, the shape function space $Z_{h,\tau}(\hatt{S})$ is built by defining a shape function space for the density, in the space-time domain, which is given by
\[
Q_{h,\tau}(S) \coloneqq \mc{P}_1([t_0,t_1])\otimes Q_h(T) 
\]
(i.e.\ the density is piecewise linear in time), and a shape function space for the momentum, in the space-time domain, which is given by
\[
V_{h,\tau}(S) \coloneqq \mc{P}_0([t_0,t_1])\otimes V_{h}(T) \,
\]
(i.e.\ the momentum is piecewise constant in time). Then, we set
\[
Z_{h,\tau}(\hatt{S}) \coloneqq   ( Q_{h,\tau}(S) \,\hat{t} )\oplus V_{h,\tau}(S)\,,
\]
where $\hat{t}$ is the unit vector oriented in the time direction. The spaces $F_{h,\tau}$ and $Z_{h,\tau}$ inherit from $Q_h$ and $V_h$ the commuting diagram property mentioned above. In particular, there exist bounded projections $\tilde{\Pi}_{F_{h,\tau}}: L^2(\Omega) \rightarrow F_{h,\tau}$ and $\tilde{\Pi}_{Z_{h,\tau}}: H(\mr{div};\Omega)\rightarrow Z_{h,\tau}$ for which the following diagram commutes
\begin{equation}\label{eq:commdiag}
\centering
\begin{tikzcd}
H(\mr{div};\Omega) \arrow[d,"\tilde{\Pi}_{Z_{h,\tau}}"]\arrow[r, "\mr{div}"] & L^2(\Omega)\arrow[d,"\tilde{\Pi}_{F_{h,\tau}}"] \\ Z_{h,\tau}\arrow[r, "\mr{div}"] & F_{h,\tau}
\end{tikzcd}
\end{equation}
where the divergence is the one associated with the space-time domain $\Omega$. Then, as before, the divergence operator is surjective onto $F_{h,\tau}$ when restricted on $Z_{h,\tau}$, i.e.\ $\mr{div}\, Z_{h,\tau} = F_{h,\tau}$. Note that the precise definition for the projection operators on tensor product meshes can be found in \cite{Arnold14}.

\subsection{Discrete projection on the divergence-free subspace}\label{sec:proj}
Denote by $\mc{B}$ the kernel of the divergence operator on $H(\mr{div};\Omega)$. Given $\xi \in L^2(\Omega)$ we define the projection $P_{\mc{B}}( \xi)$ to be the divergence-free vector field $\sigma$ minimizing the $L^2$ distance from $\xi$. This can be obtained solving the following problem for $(\sigma,\phi) \in H(\mr{div};\Omega) \times L^2(\Omega)$
\begin{equation}\label{eq:mixed1}
\left \{
\begin{array}{ll}
\langle \sigma,v \rangle + \langle \phi,\mr{div}\, v \rangle = \langle \xi,v\rangle  & \forall v \in H(\mr{div};\Omega)\,,\\
\langle \mr{div}\, \sigma, \psi \rangle = 0  & \forall \psi \in L^2(\Omega)\,.
\end{array}\right.
\end{equation} 
%The system is well-posed thanks to the surjectivity of the divergence operator onto $L^2(\Omega)$. 
Let $\mc{B}_{h,\tau}$ be the kernel of the divergence operator restricted on $Z_{h,\tau}$. We define the projection $P_{\mc{B}_{h,\tau}}( \xi)$ to be the divergence-free vector field $\sigma_{h,\tau} \in Z_{h,\tau}$ minimizing the $L^2$ distance from $\xi$. This can be obtained solving the following problem for $(\sigma_{h,\tau} ,\hatt{\phi}_{h,\tau} ) \in Z_{h,\tau}  \times F_{h,\tau}$
\begin{equation}\label{eq:mixed2}
\left \{
\begin{array}{ll}
\langle \hatt{\sigma}_{h,\tau} ,v \rangle + \langle \hatt{\phi}_{h,\tau} ,\mr{div}\, v \rangle = \langle \xi,v\rangle  & \forall v \in Z_{h,\tau}\,,\\
\langle \mr{div}\, \hatt{\sigma}_{h,\tau} , \psi \rangle = 0  & \forall \psi \in F_{h,\tau}\,.
\end{array}\right.
\end{equation} 
The commuting diagram \eqref{eq:commdiag} implies well-posedness of the discrete system. In particular, it implies the following inf-sup condition: there exists a constant $\beta>0$ independent of $h$ and $\tau$ such that
\[
\inf_{\phi \in F_{h,\tau}} \sup_{\sigma \in Z_{h,\tau}} \frac{\langle \phi,\mathrm{div}\, \sigma \rangle}{\| \sigma \|_{H(\mathrm{div})} \| \phi \|_{L^2} } \geq \beta\,,
\]
see for example proposition 5.4.2 in \cite{boffi2013mixed}.  Then, problem \eqref{eq:mixed2} is well-posed, i.e.\ it has a unique solution $(\sigma_{h,\tau},\phi_{h,\tau})$ which verifies $\sigma_{h,\tau}\in \mathcal{B}$ and
\begin{subequations}\label{eq:stabest}
\[ \| \sigma_{h,\tau} \|_{L^2} \leq C_1 \| \xi\|_{L^2} \,,  \]%
\[ \| \phi_{h,\tau} \|_{L^2} \leq C_2 \| \xi\|_{L^2} \,,   \]%
\[ \| \sigma_{h,\tau} -\sigma \|_{L^2} +\| \phi_{h,\tau} -\phi \|_{L^2} \leq C_3 \| \xi_{h,\tau} - \xi \|_{L^2} \,,  \]
\end{subequations}
where $C_1, C_2, C_3>0$ are constants independent of $h$ and $\tau$, $\xi_{h,\tau}$ is the $L^2$ projection of $\xi$ onto $Z_{h,\tau}$ and $(\sigma,\phi)$ is the unique solution of problem \eqref{eq:mixed1} (e.g., these results can be derived as particular cases of Theorems 4.3.2, 5.2.1 and 5.2.5 in \cite{boffi2013mixed}).
% It is easy to see this for the variable $\sigma_{h,\tau} $. In fact, one has $\mr{div} \Pi_{Z_{h,\tau}} \sigma  = \Pi_{F_{h,\tau}} \mr{div} \sigma  =0$. In turn, this implies that $(\Pi_{Z_{h,\tau}} \sigma  - \hatt{\sigma}_{h,\tau}  )\in \mc{B}_{h,\tau}$ and combining the systems \eqref{eq:mixed1} and \eqref{eq:mixed2} we obtain
% \begin{equation}
% \langle \hatt{\sigma}_{h,\tau}  - \sigma, \hatt{\sigma}_{h,\tau}  - \Pi_{Z_{h,\tau}} \sigma \rangle =0 .
% \end{equation}
% In particular, this implies that
% \begin{equation}
% \| \sigma - \hatt{\sigma}_{h,\tau}  \| \leq \| \sigma - \Pi_{Z_{h,\tau}} \sigma \| 
% \end{equation}
% so that the error on the discrete solution $\hatt{\sigma}_{h,\tau} $ is fully determined by the approximation properties of the operator $\Pi_{Z_{h,\tau}}$. 
%

In the following we will need to compute the discrete version of the $L^2$ projection onto $\mc{C}$. In particular we define
\begin{equation}\label{eq:Cht}
\mc{C}_{h,\tau} \coloneqq \{ \sigma \in \mc{B}_{h,\tau}\,, \,\sigma\cdot n_{\partial \Omega} = \hatt{\mc{X}}_{h,\tau} \}\,,
\end{equation}
where, since $\Pi_{Q_h}$ can be defined on $\mc{M}(D)$ (see equation \eqref{eq:PiQh}), we set
\begin{equation}\label{eq:boundaryd}
\hatt{\mc{X}}_{h,\tau} \coloneqq 
\left \{ 
\begin{array}{ll} 
\Pi_{Q_h} \rho_0 & \text{ on } \{0\} \times D,\\ 
\Pi_{Q_h} \rho_1 & \text{ on } \{1\} \times D,\\
0 & \text{ otherwise. } 
\end{array}
\right.
\end{equation} 
The well-posedness results described above for the $L^2$ projections onto $\mc{B}$ and $\mc{B}_{h,\tau}$ hold also for the $L^2$ projections onto $\mc{C}$ and $\mc{C}_{h,\tau}$ up to adding \modif{Neumann} boundary conditions to the spaces $H(\mr{div};\Omega)$ and $Z_{h,\tau}$, and replacing $L^2(\Omega)$ and $F_{h,\tau}$ by $L^2(\Omega)/ \mathbb{R}$ and $F_{h,\tau}/\mathbb{R}$, respectively.

\section{Discrete dynamical formulation and convergence}\label{sec:discrete}

In this section we formulate the discrete problem and state a convergence result obtained by applying the theory developed in \cite{lavenant2019unconditional}. For this, we need to introduce a space for the discrete dual variable $q$. We adopt the same notation as for the spaces defined in Section \ref{sec:fe}. In particular, we set for $r\in\{0,1\}$,
\[
X_{h}^r \coloneqq \{\phi \in L^2(D)\,;\, \phi|_T \in X_h^r(T), ~ \forall\, T\in \mc{T}_h\}.
\]
The superscript $r$ denotes the polynomial order of the shape function space $X_h^r(T)$.  We distinguish two cases:
\begin{enumerate}
\item for simplicial elements (triangles or tetrahedrons), we take $X_h^r(T)\coloneqq \mc{P}_r(T)$.
\item for quadrilateral elements, we set $T = \phi( I^d)$, where $I$ is an interval and $\phi$ an affine transformation, and we take $X^r_h(T)\coloneqq \mc{P}_r(I)^d\circ \phi^{-1}$.
\end{enumerate}
The associated space-time space is defined by
\[
\hatt{X}_{h,\tau}^r \coloneqq \{\phi \in L^2(\Omega)\,;\, \phi|_S \in \hatt{X}^r_{h,\tau}(\hatt{S}), ~ \forall\, \hatt{S}\in \hatt{\mc{T}}_{h,\tau}\},
\]
with $\hatt{X}_{h,\tau}^r(\hatt{S}) =\mc{P}_0([t_0,t_1]) \otimes X_h^r(T)$. In order to simplify the notation, we will omit the superscript $r$ when not relevant to the discussion.

\begin{remark} The choice $r\in\{0,1\}$ is dictated by computational feasibility of the algorithm. In fact, for these cases, we can compute explicitly the projection on $\mc{K}\cap \hatt{X}^r_{h,\tau}$ \modif{(with respect to appropriate inner products)} as it will be explained in the next section. On the other hand, we restrict ourselves to piecewise constant functions in time since this is crucial for the convergence of the algorithm, as shown in \cite{lavenant2019unconditional}.
\end{remark}

The discrete action (at fixed time) is defined as follows:
\[
A_h(\rho,m) \coloneqq \sup_{(a,b) \in (X_{h})^{d+1}} \{ \langle \rho,a\rangle + \langle m,b\rangle\,;\, (a,b) \in K \, a.e. \}
\]
for any $(\rho,m) \in Q_h \times V_h$. By construction, $A_h:Q_h\times V_h \rightarrow [0, +\infty]$ is a proper convex function $-1$-positively homogeneous in its first variable and $2$-positively homogeneous in its second variable. Moreover, it is non-increasing in its first argument, i.e.\ $A_h(\rho_1+\rho_2,m) \leq A_h(\rho_1,m)$ for any $\rho_1,\rho_2 \in Q_h^+$ and $m\in V_h$. In fact, suppose that $A_h(\rho_1 + \rho_2,m)<+\infty$. Then there exists $(a^*,b^*) \in \modif{(X_h)^{d+1}\cap \mc{K}}$ such that  $\langle \rho_1+\rho_2,a^*\rangle + \langle m,b^*\rangle = A_h(\rho_1+\rho_2,m)$; in particular $a^*\leq 0$. Then
\[
A_h(\rho_1,m) \geq A_h(\rho_1+\rho_2,m)-\langle \rho_2,a^* \rangle \geq A_h(\rho_1+\rho_2,m)\,,
\] 
and by a similar reasoning we obtain that if $A(\rho_1+\rho_2,m) = +\infty$ then we also have $A(\rho_1,m) = +\infty$.

The space-time discretization of problem \eqref{eq:BB} is given by
\begin{equation}\label{eq:BBd}
\inf_{\substack{\sigma \in \mc{C}_{h,\tau},\\\modif{\rho\geq 0}}} \mc{A}_{h,\tau}(\sigma) ,\quad \mc{A}_{h,\tau} (\sigma) \coloneqq \sup_{q \in (X^r_{h,\tau})^{d+1}\cap \mc{K}}  \langle q, \sigma\rangle\,.
\end{equation}
Note that, by definition, $\mc{A}_{h,\tau}$ is convex and non-negative. Therefore, problem \eqref{eq:BBd} always admits minimizers.

Suppose that the time discretization is given by a decomposition of the interval $[0,1]$ in $N$ elements, i.e.\ fixing the points $0=t_0 <t_1 <\ldots < t_{N+1} =1$. Given $\sigma = (\rho,m) \in Z_{h,\tau}$, we can identify the density $\rho$ with the collection $\{ \rho_i\}_{i=0}^{N+1}$ with $\rho_i \in Q_h$, and the momentum $m$ with the collection  $\{ m_i\}_{i=1}^{N+1}$ with $m_i \in V_h$. Since $q$ is piecewise constant in time, we have the following equivalent formulation
\begin{equation}\label{eq:spacetimeaction}
\mc{A}_{h,\tau}(\sigma) = \sum_{i=1}^{N+1}A_h \left(\frac{\rho_i+\rho_{i-1}}{2},m_i \right) |t_{i}-t_{i-1}| \,.
\end{equation}
Note that in order to obtain \eqref{eq:spacetimeaction} from \eqref{eq:BBd}, we relied on the particular choice of finite element spaces for density (piecewise linear in time), momentum (piecewise constant in time) and the corresponding dual variables (piecewise constant in time).

\begin{remark}[Continuity constraint] The choice of a  $H(\mathrm{div})$-conforming finite element space for $\sigma$  implies that the weak form of the continuity equation $\partial_t{\rho} + \mathrm{div}_x\, m =0$ is satisfied exactly by any solution of the  discrete saddle point problem \eqref{eq:BBd} (this is also directly implied by the definition of the constraint set $\mathcal{C}_{h,\tau}$ in \eqref{eq:Cht}) .
\end{remark}

\begin{remark}[Positivity constraint] \label{rem:pos} \modif{Note that removing the positivity constraint in the formulation \eqref{eq:BBd}, we obtain a different scheme. In that case, since the action is evaluated on the mean density (in time), the positivity constraint $\rho\ge0$ is then only enforced on $\frac{\rho_i+\rho_{i-1}}{2}$, rather than on each $\rho_i$ separately.} %For certain choices of spaces this may lead to oscillations as shown in the numerical tests in section\ref{sec:num}.
\end{remark}

The objects introduced until now define a finite dimensional model of optimal transport in the sense of Definition 2.5 in \cite{lavenant2019unconditional}. The framework developed therein can be used to deduce a convergence result for our scheme. %Before stating it, we list some additional assumption on the mesh $\mc{T}_h$ that will be required.

% \begin{assumption}[Isotropy condition on $\mc{T}_h$]\label{ass:isotropy} For each $T\in\mc{T}_h$, let $\mc{T}_{h,T}$ be the set of neighbouring elements $L\in \mc{T}_h$ such that $f_{T,L}\coloneqq \overline{T}\cap \overline{L} \neq \emptyset$; morever, let $x_T$ be a point in the interior of $T$ such that $x_T -x_L$ is orthogonal to $f_{T,L}$ for any $L\in \mc{T}_{h,T}$, so that $n_{T,L} \coloneqq \pm (x_L-x_T)/|x_L-x_T|$ is the normal to $f_{T,L}$, and we choose the sign according to a prescribed orientation. We assume that there exists a constant $\epsilon_h$ tending to $0$ as $h\rightarrow 0$, such that for any $T\in\mc{T}_h$ and any vector $v\in \mathbb{R}^d$
% \begin{equation}
% \frac{1}{2|T|} \sum_{L\in \mc{T}_{h,T}} |f_{T,L}| |x_T - x_L| |v\cdot n_{T,L}|^2 \leq (1+\epsilon_h) |v|^2\,.
% \end{equation}
% \end{assumption}

\begin{theorem}\label{th:convergence}
Let $\rho_0,\rho_1 \in \mc{P}(D)$ be given and $\{\hatt{\mc{T}}_{h,\tau}\}_{h,\tau > 0}$ a family of tensor-product decomposition of $\Omega$ such that the time discretization is uniform, i.e.\ $t_{i} -t_{i-1} = \tau$ for all $i=1, \ldots, N+1$, and the space discretization $\mc{T}_h$ satisfies equation \eqref{eq:Cmesh}. Let $\sigma_{h,\tau}$ be a minimizer of problem \eqref{eq:BBd} associated with $\hatt{\mc{T}}_{h,\tau}$ and for $r=1$. Then, as $h,\tau \rightarrow 0$,  up to extraction of a subsequence, $\sigma_{h,\tau}$ converges weakly to $\sigma \in \mc{M}(\Omega)^{d+1}$ a minimizer of problem \eqref{eq:BB}.
\end{theorem}

The proof is essentially an extension of the one presented in \cite{lavenant2019unconditional} and is postponed to the appendix. %Note that assumtion \ref{ass:isotropy} was introduced by Gladbach, Kopfer and Mass \cite{gladbach2018scaling} to prove the convergence of finite volume space-discretizations of optimal transport. In our proof we will relate our discretization to the one in \cite{gladbach2018scaling} and use assumption \ref{ass:isotropy} to import some properties of the latter to our setting.

\begin{remark}[Stability]\label{rem:stability} The existence of bounded projections sastifying the commuting diagram \eqref{eq:commdiag} ensures stability of the projection onto $\mathcal{C}_{h,\tau}$ (see \eqref{eq:stabest}). Such commuting projections are also crucial to estabilish the convergence result in Theorem \ref{th:convergence}: in \cite{lavenant2019unconditional}, they are used to sample the continuous solution into a discrete one satisfying the continuity equation, therefore providing an admissible candidate for the discrete problem. Nonetheless, due to the nonlinear constraint  $q\in \mathcal{K}$, one cannot apply the standard linear theory in \cite{boffi2013mixed}, for example, so the commuting diagram condition does not imply directly a stability result analogous to \eqref{eq:stabest} for the saddle point problem \eqref{eq:BBd} (even if we see it as a discretization of the Hilbert space formulation in Proposition \ref{prop:hilbert}). \modif{Numerically (see Section \ref{sec:num}) the finite element pairs considered here $(Z_{h,\tau},X_{h,\tau}^r)$ appear to be stable when $r=1$, but strong oscillations may occur for $r=0$ and $V_h = \mathcal{BDM}_1$, providing empirical evidence of} the instablity of the discretization for this case.
\end{remark}

\begin{remark}\label{rm:convpeyre}
Suppose that $D = [0,1]^d$ and that $\hatt{\mc{T}}_{h,\tau}$ is a uniform quadrilateral discretization of $\Omega = [0,1]^{d+1}$. Then for $r=0$ \modif{and removing the constraint $\rho \geq 0$ (see Remark \ref{rem:pos})}, the discrete problem \eqref{eq:BBd} coincides with the discretization proposed in \cite{papadakis2014optimal}. Theorem \ref{th:convergence} shows that modifying this method with $r=1$ \modif{and adding the positivity constraint at all times}, one can prove convergence to the solution of the continuous problem \eqref{eq:BB}.
\end{remark}

% \section{The dual problem}\label{sec:dual}
% Since the optimal transport problem \eqref{eq:BB} as well as its discretization \eqref{eq:BBd} are convex optimization problems, it is useful to derive and compare the corresponding dual problems.  
%
% \begin{proposition} Problem \eqref{eq:BBd} admits the following dual formulation
% \begin{equation}\label{eq:dual}
% \end{equation}
% \end{proposition}

\section{The proximal splitting algorithm}\label{sec:proximal}
\modif{We now describe in detail the discrete version of the proximal splitting algorithm introduced in Section \ref{sec:ham}, in the simplest setting where we remove the additional positivity constraint on the density, i.e.\ we solve
\[
\inf_{\substack{\sigma \in \mc{C}_{h,\tau}}} \mc{A}_{h,\tau}(\sigma) ,\quad \mc{A}_{h,\tau} (\sigma) \coloneqq \sup_{q \in (X_{h,\tau})^{d+1}\cap \mc{K}}  \langle q, \sigma\rangle\,.
\]
As mentioned in Remark \ref{rem:pos}, this amounts to enforcing positivity only on the mean density in time between consecutive time-steps. Using this formulation rather than \eqref{eq:BBd} we can reproduce the structure of the continuous version of the scheme, described in Section \ref{sec:ham}. Note, however, that one can actually solve problem \eqref{eq:BBd} with a similar strategy, e.g., by first reformulating the problem intrudicing a Lagrange multiplier to enforce the continuity equation, and then applying the same proximal splitting algorithm considered here but with the new variables and with an appropriate choice of norms.}

We start by defining
\[
\mc{K}_{h,\tau}^r \coloneqq \mc{K} \cap (\hatt{X}_{h,\tau}^r)^{d+1} \coloneqq \{ q \in  (\hatt{X}_{h,\tau}^r)^{d+1}\,; \, q \in K \, a.e.\}\,.
\]
We write the discrete problem as follows:
\begin{equation}\label{eq:BBdprox}
\inf_{\sigma \in L^2(\Omega;\mathbb{R}^{d+1}) } \sup_{q \in L^2(\Omega; \mathbb{R}^{d+1})} 
\langle q, \sigma\rangle + \iota_{\mc{C}_{h,\tau}} (\sigma) - \iota_{\mc{K}_{h,\tau}} (q)\,,
\end{equation}
where $\mc{C}_{h,\tau}$ is defined in \eqref{eq:Cht}. Then, the proximal splitting algorithm of section  \ref{sec:ham} applied to problem \eqref{eq:BBdprox} can be formulated as follows: given $\tau_1, \tau_2>0$ and an admissible $(\sigma^0, q^0)\in \mc{C}_{h,\tau} \times \mc{K}_{h,\tau}$, we define the sequence $\{(\sigma^k,q^k)\}_k$ by performing iteratively the following two steps:
\begin{subequations}\label{eq:PSd}
\begin{align}
\text{\textbf{Step 1}}: \qquad & \sigma^{k+1} = P_{\mc{C}_{h,\tau}} (\sigma^k- \tau_1  q^k)\,.\\
\text{\textbf{Step 2}}: \qquad &q^{k+1} = P_{\mc{K}_{h,\tau}}(q^k +\tau_2 (2 \sigma^{k+1}-\sigma^k))\,.
\end{align}
\end{subequations}
%where $P_{Z_{h,\tau}} $ and $P_{\hatt{X}_{h,\tau}^r}$ denote the $L^2$ projections onto $Z_{h,\tau}$ and $\hatt{X}_{h,\tau}^r$, respectively.
The convergence result in Theorem \ref{th:cp} clearly holds also in the discrete setting and gives convergence of the algorithm to a discrete saddle point $(\sigma_{h,\tau}, q_{h,\tau})$, if the condition $\tau_1 \tau_2 <1$ is satisfied. %However, we still need Proposition \ref{prop:convergence} to pass to the limit for $h,\tau\rightarrow 0$.
The two steps in the algorithm can be computed as follows.
\\

\paragraph{\textbf{Step 1}} As discussed in Section \ref{sec:proj}, the projection $P_{\mc{C}_{h,\tau}}$ can be computed modifying the system given by \eqref{eq:mixed2} by adding the Neumann boundary conditions associated with the function \eqref{eq:boundaryd}. 
\\

\paragraph{\textbf{Step 2}} Since $P_{\mc{K}_{h,\tau}}$ is an $L^2$ projection, we have that $P_{\mc{K}_{h,\tau}} = P_{\mc{K}_{h,\tau}} \circ P_{(\hatt{X}_{h,\tau}^r)^{d+1}}$, where  $P_{(\hatt{X}_{h,\tau}^r)^{d+1}}$ denotes the $L^2$ projection onto $(\hatt{X}_{h,\tau}^r)^{d+1}$. This means that we only need to be able to compute $P_{\mc{K}_{h,\tau}}$ when applied to an element of $\hatt{X}_{h,\tau}^r$. In addition, since $\hatt{X}_{h,\tau}^r$ is discontinuous across elements, we can compute the projection element by element,  and since functions in $\hatt{X}^r_{h,\tau}(S)$ are constant along the time direction, we can also eliminate the time variable in the projection. \modif{In other words, we only need to solve for each element $[t_0,t_1]\times S$ a problem in the form
\begin{equation}\label{eq:projectionS}
\xi_{\mc{K}} \coloneqq \mr{argmin} \{ \| \xi - q \|_{L^2(T)}^2  \,;\,  q \in ({X}_{h}^r(T))^{d+1}\,, ~ q(x) \in K ~\forall x\in T\}
\end{equation}
for a given $\xi \in (X_h^r(T))^{d+1}$.} % More precisely, let $\xi \in (\hatt{X}^r_{h,\tau}(S))^{d+1}$ and $S = [t_0,t_1] \times T$, we can identify $\xi$ with an element of $({X}_{h}^r(T))^{d+1}$. Then,
%\begin{equation}\label{eq:projectionS}
%\xi_{\mc{K}} \coloneqq P_{\mc{K}_{h,\tau}}|_S \,\xi = \mr{argmin} \{ \| \xi - q \|_{L^2(T)}^2  \,; q \in ({X}_{h}^r(T))^{d+1}\cap K\}\,.
%\end{equation}
We distinguish two cases:
\begin{enumerate}
\item if $r=0$, the projection \eqref{eq:projectionS} is just the projection of a vector $\xi \in \mathbb{R}^{d+1}$ onto the convex set $K$;
\item \modif{if $r=1$, any $\xi\in (X_h^1)^{d+1}$ is fully determined by its value on the vertices $\{v_i\}_i$ of $T$, and the condition $\xi\in \mathcal{K}$, is equivalent to $\xi(v_i) \in K$, by convexity of the set $K$ (see equation \eqref{eq:K}). However the problem is coupled in these variables when computing the projection in the $L^2$ norm. Here, we use instead a different projection and we simply set
\[
\xi_{\mc{K}}(v_i) = \mr{argmin} \{ | \xi(v_i) - q |^2  \,; q \in K\}\,.
\]
Note that this is a variational crime, but it can be avoided by reformulating the algorithm using as inner product on $X_h^1$ a weighted $\ell^2$ inner product on the degrees of freedom.}
\end{enumerate}
In both cases we only need to compute for each degree of freedom the projection of a given vector $(\bar{a},\bar{b}) \in \mathbb{R} \times \mathbb{R}^{d}$ onto $K$. If $(\bar{a},\bar{b})\notin K$, such a projection is given explicitly by the vector
\[
\left (-\frac{\mu^2}{2},\mu \,\frac{\bar{b}}{|\bar{b}|}\right)
\]
where $\mu \geq 0$ is the largest real root of the third order polynomial
\[
x \mapsto \frac{x^3}{2} + x (\bar{a}+1) - |\bar{b}|\,.
\]
%\cite{papadakis2014optimal} for example.

\begin{remark}
As for the finite difference discretization studied in \cite{papadakis2014optimal}, different optimization techniques could be applied to solve problem \eqref{eq:BBd}. In particular, it should be noted that the ADMM approach orginally proposed by Benamou and Brenier \cite{benamou2000computational} could also be applied. This would lead to a very similar algorithm to \eqref{eq:PSd}, but it would require the introduction of an additional variable which avoids coupling of the degrees of freedom in the optimization step with respect to $q$. In other words, this is needed in order to be able to perform the projection on $K$ for each degree of freedom separately. More details on this issue can be found in \cite{papadakis2014optimal} for the discretization studied therein, and they hold also in the finite element setting.
\end{remark}

\section{Regularization} \label{sec:regularization}
The optimal transport problem does not involve any regularizing effect on the interpolation between two measures. 
In fact, one can even expect a loss of regularity in some cases, namely if one is interpolating between two smooth densities on a smooth but non-convex domain. Such a loss of regularity (which is often unphysical when the density represents a physical quantity) can be avoided introducing additional regularization terms in the formulation.
In this section we describe how to do so, and how these modifications translate at the algorithmic level. 

We consider the Hilbert space setting discribed in Section \ref{sec:ham} and we study problems in the form
\begin{equation}\label{eq:BBr}
\inf_{\sigma \in \mc{C}} \mc{A}(\sigma) + \alpha \mc{R}(\sigma)
\end{equation} 
where $\mc{R}:L^2(\Omega) \rightarrow \mathbb{R}$ is a convex, proper and l.s.c. functional, and $\alpha>0$. For this type of problem, we can still apply the proximal splitting algorithm \eqref{eq:PS} replacing the projection onto $\mc{C}$ by $\mr{prox}_{\tau_1 \mc F}$, the proximal operator of $\mc{F} \coloneqq \iota_{\mc C} + \alpha \mc{R}$, defined by
\[
\mr{prox}_{\tau_1 \mc{F}}(\xi) = \underset{\eta\in L^2(\Omega; \mathbb{R}^{d+1})}{\mr{argmin}} \frac{\|\xi- \eta\|^2}{2\tau_1} + \mc{F}(\eta)\,.
\]
This leads to the so-called PDGH algorithm, which for $\tau_1 \tau_2<1$ can be seen just as a proximal point method applied to a monotone operator \cite{chambolle2016introduction}, and therefore we still have convergence in the Hilbert space setting. As mentioned in \cite{lavenant2019unconditional} convergence of the discrete problem with mesh refinement is more delicate and will not be discussed here.

\subsection{Mixed $L^2$-Wasserstein distance}
% Formulas for prox
Define for any $\sigma = (\rho,m) \in L^2(\Omega) \times L^2(\Omega;\mathbb{R}^d)$
\[
\mc{R}(\sigma) \coloneqq \left \{
\begin{array}{ll}
\frac{1}{2}\| \partial_t \rho \|^2_{L^2(\Omega)} & \text{if } \partial_t\rho \in L^2(\Omega) \,,\\
+\infty & \text{otherwise}\,.
\end{array}
\right. 
\]
With this functional, problem \eqref{eq:BBr} yields an interpolation between the Wasserstein distance and the $L^2$ distance. 
It was originally considered in \cite{benamou2001mixed}, where a conjugate gradient method was proposed to compute the minimizers.
Let $V \coloneqq H^1([0,1]; L^2(D)) \times L^2([0,1];H(\mr{div};D))$ and let $\accentset{\circ}{V}$ be the same space with homogenous boundary conditions on the fluxes. For any $\xi \in L^2(\Omega)^{d+1}$, $\sigma = \mr{prox}_{\tau_1 \mc{F}}(\xi)$ is obtained by solving the following system for $(\sigma,\phi) \in V \times L^2(\Omega)/\mathbb{R}$
\[
\left \{
\begin{array}{ll}
\langle \sigma,v \rangle + \alpha \tau_1 \langle \partial_t \rho, \partial_t v_t \rangle + \langle \phi,\mr{div}\, v \rangle = \langle \xi,v\rangle \,, & \forall v \in \accentset{\circ}{V}\,,\\
\langle \mr{div}\, \sigma, \psi \rangle = 0 \, , & \forall \psi \in L^2(\Omega)/\mathbb{R}\,, \\
\sigma \cdot n_{\partial \Omega} = \mc{X}\,,
\end{array}\right.
\]
where $v_t = v\cdot \hat{t}$ is the component of $v$ in the time direction. Well-posedness can be obtained by standard methods for saddle point problems \cite{boffi2013mixed} and it translates directly into well-posedness of the discrete system obtained by replacing $V$ with $Z_{h,\tau}$, $L^2(\Omega)$ with $F_{h,\tau}$, and  $\mc{X}$ with $ \mc{X}_{h,\tau}$. 

\subsection{$H^1$ regularization}
% Formulas for prox
Define for any $\sigma = (\rho,m) \in L^2(\Omega) \times L^2(\Omega;\mathbb{R}^d)$
\begin{equation}\label{eq:regh1}
\mc{R}(\sigma) \coloneqq \left \{
\begin{array}{ll}
\frac{1}{2}\| \nabla_x \rho \|^2_{L^2(\Omega)}& \text{if } \rho \in  L^2([0,1];H^1(D))\,, \\
+\infty & \text{otherwise}\,.
\end{array}
\right. 
\end{equation}
In this case we set $V \coloneqq H(\mr{div};\Omega)$, $W \coloneqq L^2([0,1];H(\mr{div}_x;D))$ and let $\accentset{\circ}{V}$ and $\accentset{\circ}{W}$ be the same spaces with homogenous boundary conditions on the fluxes. Then, for any $\xi \in L^2(\Omega)^{d+1}$,  $\sigma = \mr{prox}_{\tau_1 \mc{F}}(\xi)$ is obtained by solving the following system for $(\sigma,\eta,\phi) \in V \times \accentset{\circ}{W} \times L^2(\Omega)/\mathbb{R}$
\[
\left \{
\begin{array}{ll}
\langle \sigma,v \rangle - \alpha \tau_1 \langle \mr{div}_x \eta, v_t \rangle + \langle \phi,\mr{div}\, v \rangle = \langle \xi,v\rangle \,, & \forall v \in \accentset{\circ}{V}\,,\\
\langle \rho, \mr{div}_x w \rangle + \langle \eta, w \rangle = 0 \,, & \forall w\in \accentset{\circ}{W}\,,\\
\langle \mr{div}\, \sigma, \psi \rangle = 0 \, , & \forall \psi \in L^2(\Omega)/\mathbb{R}\,, \\
\sigma \cdot n_{\partial \Omega} = \mc{X}\,,
\end{array}\right.
\]
where $v_t = v\cdot \hat{t}$ is the component of $v$ in the time direction. %Note that this formulation implies homogeneous boundary conditions also on $\nabla \rho \cdot n_{\partial \Omega}$ on $[0,1] \times \partial D$. The system that we obtain is similar to the usual Stokes' problem, although the Laplacian only acts on $\rho$ and not on $\sigma$. 
As before, well-posedness can be obtained by standard methods for saddle point problems \cite{boffi2013mixed}. 

%In this case, our choice of discrete spaces is not conforming for $\sigma$, since functions in $Q_{h}$ are discontinous across elements. We can still solve the problem using mixed finite element techniques although different options are available (see, e.g., \cite{eymard2008discretization} for a connection with finite volume techniques).

%Here we denote by $\mr{div}_x$ the divergence operator on $D$. 
We introduce the space $\hatt{W}_{h,\tau} \subset L^2([0,1];H(\mr{div}_x;D))$ whose shape functions on $S = [t_0,t_1] \otimes T$ are given by 
\[\hatt{W}_{h,\tau}(S) \coloneqq \mc{P}_1([t_0,t_1])\otimes V_h(T).
\] 
%Therefore $\hatt{W}_{h,\tau}$ is continuous along the time direction. 
We denote by $\accentset{\circ}{\hatt{W}}_{h,\tau}$ the same space with the boundary conditions $\eta \cdot n_{\partial \Omega} = 0$ on $[0,1] \times \partial D$. %By construction the divergence operator $\mr{div}_x$ is surjective from $\accentset{\circ}{\hatt{W}}_{h,\tau}$ to $(Z_{h,\tau}\cdot \hat{t})/\mathbb{R}$. 
Denote by $\nabla_x^h : L^2(\Omega) \rightarrow \accentset{\circ}{\hatt{W}}_{h,\tau}$ the adjoint of $-\mr{div}_x$  defined by
\[
\langle \nabla_x^h \phi, \eta \rangle = -\langle \phi, \mr{div}_x \eta \rangle \,, \quad \forall \, (\phi, \eta) \in L^2(\Omega)\times \accentset{\circ}{\hatt{W}}_{h,\tau}\,.
\]

We define a discrete version of \eqref{eq:regh1} as follows:
\[
\mc{R}_{h,\tau} (\sigma) \coloneqq  \frac{1}{2} \| \nabla_x^h \rho\|^2_{L^2(\Omega)}.
\]
Let $\mc{F}_{h,\tau} \coloneqq \iota_{\mc{C}_{h,\tau}} + \alpha \mc{R}_{h,\tau}$.
Then for any $\xi \in L^2(\Omega)^{d+1}$,  $\sigma = \mr{prox}_{\tau_1 \mc{F}_{h,\tau}}(\xi)$ is obtained by solving the following system for $(\sigma,\eta,\phi) \in \accentset{\circ}{V}_{h,\tau} \times \accentset{\circ}{W}_{h,\tau} \times F_{h,\tau}/\mathbb{R}$:
\[
\left \{
\begin{array}{ll}
\langle \sigma,v \rangle - \alpha \tau_1\langle \mr{div}_x \eta, v_t\rangle + \langle \phi,\mr{div}\, v \rangle = \langle \xi,v\rangle \,, & \forall v \in \accentset{\circ}{V}_{h,\tau}\,,\\
\langle \rho, \mr{div}_x w \rangle + \langle \eta, w \rangle =0 \,, & \forall w \in  \accentset{\circ}{W}_{h,\tau}\,,\\
\langle \mr{div}\, \sigma, \psi \rangle = 0 \, , & \forall \psi \in F_{h,\tau}/\mathbb{R}\,, \\
\sigma \cdot n_{\partial \Omega} = \mc{X}_{h,\tau}\,.
\end{array}\right.
\]

\section{Numerical results}\label{sec:num}
In this section we describe two numerical tests that demonstrate the behaviour of the proposed discretization both qualitatively and in terms of convergence of the algorithm. For both tests the time discretization is uniform, but we will use different meshes and finite element spaces for the discretization in space. \modif{For all tests,  we set $\tau_1 = \tau_2 = 1$ as parameters of the proximal splitting algorithm \eqref{eq:PSd}}. The results shown hereafter have been obtained using the finite element software Firedrake \cite{Rathgeber2016} (see \cite{McRae2016,Bercea2016}, for the tensor product constructions) and the linear solver for the mixed Poisson equation is based on PETSc \cite{petsc-user-ref,petsc-efficient}. The code to perform the tests in this section can be found in the repository 
\url{https://github.com/andnatale/dynamic-ot.git}.

\subsection{Qualitative behaviour and convergence of the proximal-splitting algorithm}
\modif{We set $D=[0,1]^2$, and consider either a structured triangular mesh, an unstructured one, or a uniform Cartesian mesh (shown in figures \ref{fig:tristrcomp}, \ref{fig:triunstrcomp} and \ref{fig:quadcomp}), and $\tau \coloneqq |t_{i+1} - t_i| = 1/20$. The initial and final densities are given by
\begin{equation}\label{eq:bcscos}
\rho_0(x) \propto \frac{3}{2} + \cos(2\pi |x-x_0|) 
, \quad \rho_1(x) \propto  \frac{3}{2} - \cos(2\pi |x-x_0|)  \,,
\end{equation}
where $x_0 = (0.5,0.5)$, and they are normalized so that the total mass is equal to one. In figures  \ref{fig:tristrcomp}, \ref{fig:triunstrcomp} and \ref{fig:quadcomp}, the interpolation at time $t=0.5$ is shown for different choices of spaces $V_h$ and $X_h$ and different meshes. The discretization corresponding to the couple $V_h = \mc{BDM}_1$ and $X_h^0$ yields a very oscillatory solution both on the structured and unstructured mesh. Oscillations appear also for $V_h = \mc{RT}_0$, although the qualitative features of the solution are well captured. For this latter case, the oscillations seem to be very sensitive to the structure of the mesh and are attenuated when choosing $X_h^1$ instead of $X_h^0$. On the Cartesian mesh the scheme does not generate any oscillations, with the choice of the space $X_h^1$ leading to slightly more diffusive results. Note that the appearance of oscillations is not related to the positivity constraint since in the case considered here the interpolation is strictly positive. On the other hand, we remark that for tests leading to pure translation of compactly supported densities (not shown) the oscillations disappear almost entirely even for the couple $V_h = \mc{BDM}_1$, $X_h^0$.  %In general, for the same space $V_h$, the discretizations corresponding to $X_h^1$ are slightly more diffusive than those corresponding to $X_h^0$, but they also lead to smaller oscillations on the positivity. On the other hand for the couple of spaces $V_h =\mc{BDM}_1$, $X_h^0$, we observe some oscillations appearing in the interpolation (see also Remark \ref{rem:pos}). 
}
\modif{
In figure \ref{fig:Gaussianconv}, the different schemes are compared in terms of convergence of the proximal splitting algorithm. For each mesh and combination of spaces, we compute a reference solution $\sigma^*$ corresponding to $10^4$ iterations of the algorithm and we estimate the error at the $n$th iteration by $\|\sigma^n-\sigma^*\|_{L^2(\Omega)}$. The cases corresponding to $X_h^1$ appear to converge significantly faster than those corresponding to $X_h^0$. Note that in the case $V_h=\mc{RT}_{[0]}$, $X_h^0$, the resulting discretization as well as the optimization algorithm coincide with the ones proposed in \cite{papadakis2014optimal}, since here we consider a uniform Cartesian grid. Also in this case, replacing $X_h^0$ with $X_h^1$ (besides providing a convergence guarantee, see Remark \ref{rm:convpeyre}) yields a considerable speedup of the algorithm.
}

\begin{figure}[htbp]
\captionsetup[subfigure]{labelformat=empty}
    
     \begin{minipage}{0.29\linewidth}
          \subfloat{\includegraphics[scale=.2, trim =  250 50 250 50,clip]{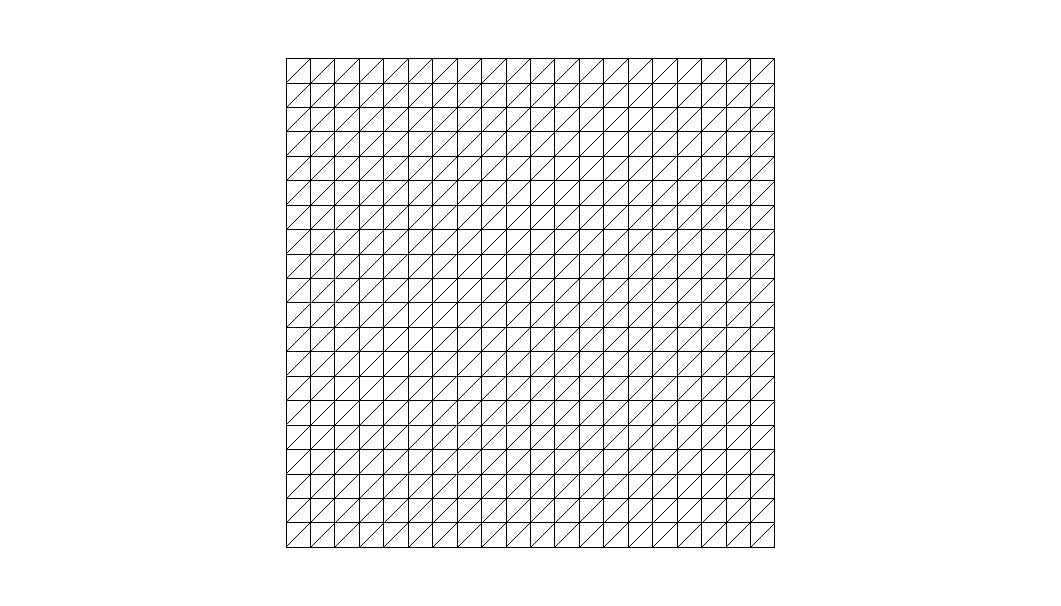}} \hfill
     \end{minipage} %
    \begin{minipage}{0.58\linewidth}
     \subfloat[{$\mc{RT}_0, X_h^0$}]{\includegraphics[scale=.2, trim =  250 50 250 50,clip]{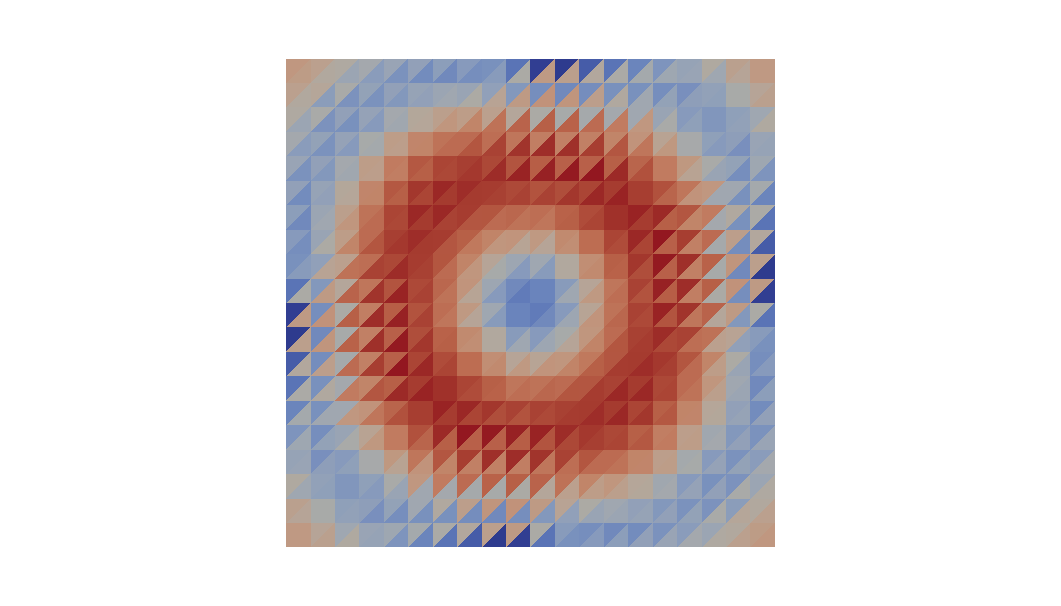}}
     \subfloat[{$\mc{RT}_{0}, X_h^1$}]{\includegraphics[scale=.2, trim = 250 50 250 50,clip]{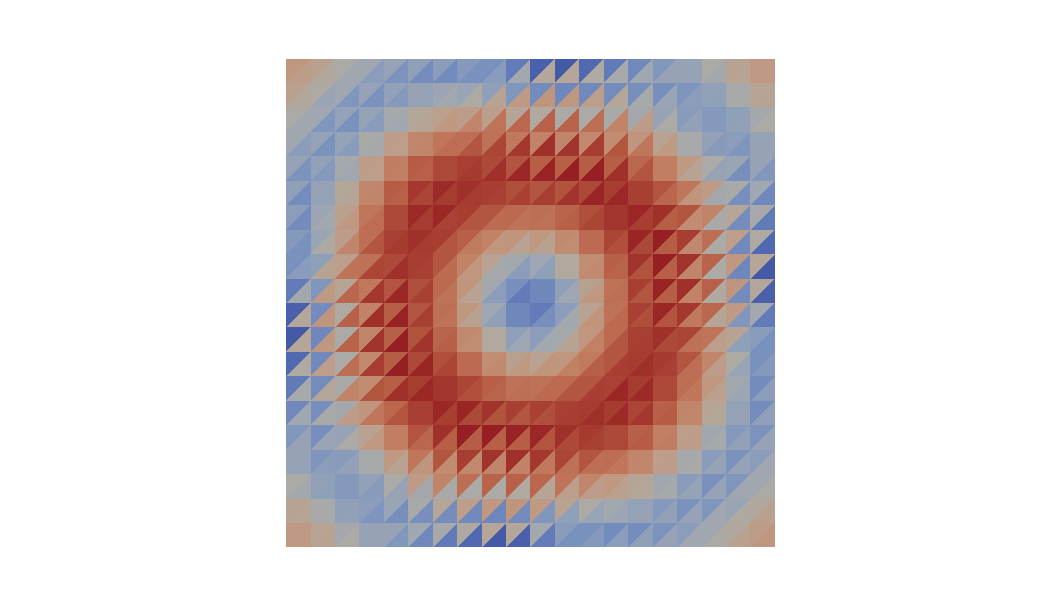}}\\
     \subfloat[{$\mc{BDM}_1, X_h^0$}]{\includegraphics[scale=.2,trim =  250 50 250 50,clip]{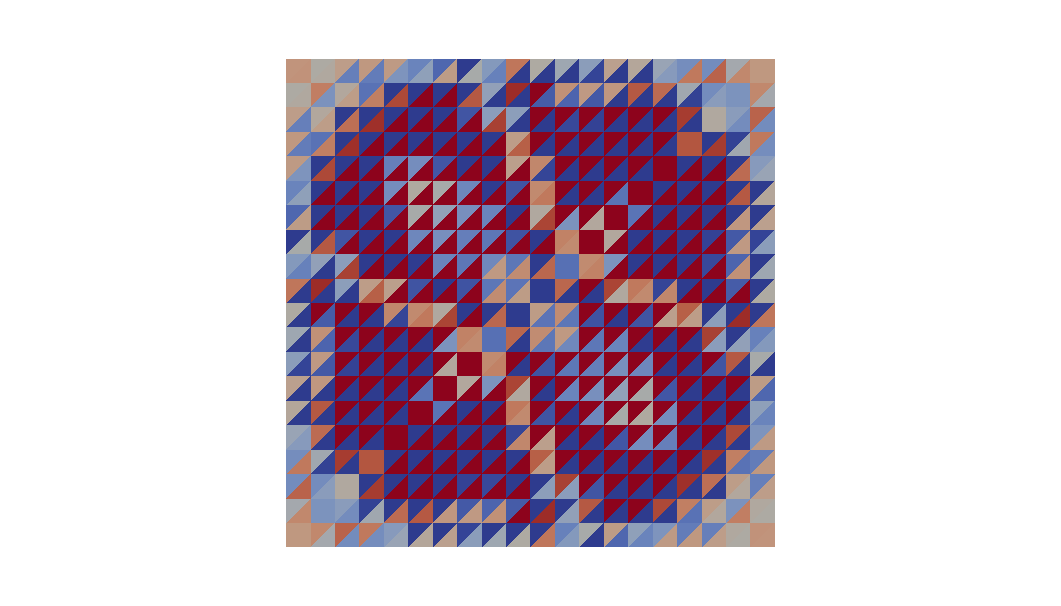}}
     \subfloat[{$\mc{BDM}_1, X_h^1$}]{\includegraphics[scale=.2, trim =  250 50 250 50,clip]{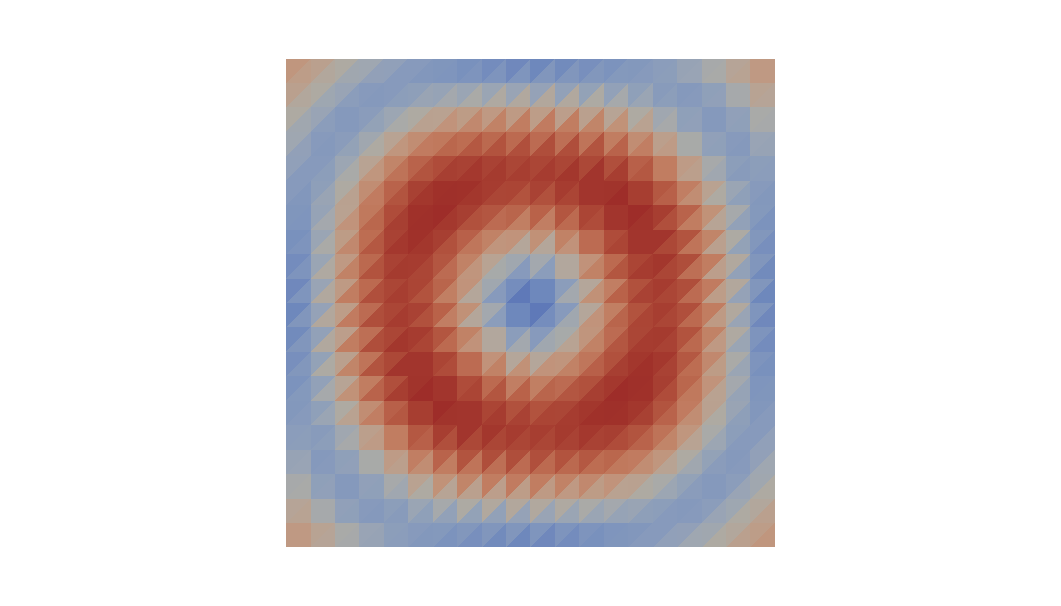}}
    \end{minipage}%
    \begin{minipage}{0.1\linewidth}\hspace{-1em}
    \subfloat{\includegraphics[scale=.27, trim = 760 90 160 120,clip]{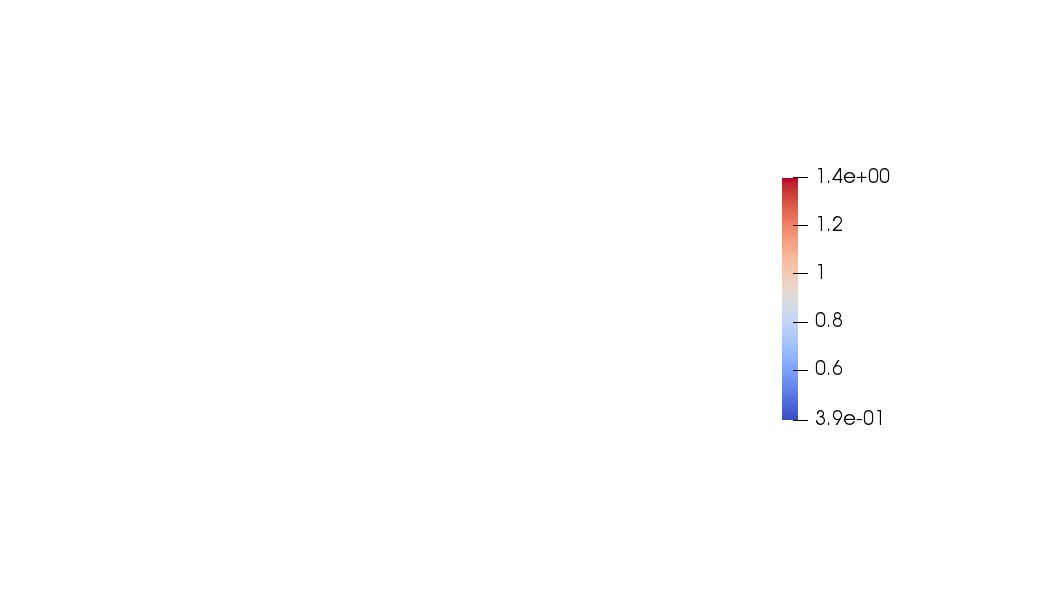}}
    \end{minipage}
     \caption{Comparison between optimal transport interpolations of the densities in \eqref{eq:bcscos} for different spaces on a structured triangular mesh. Note that in the case $V_h = \mc{BDM}_1$, $X_h^0$, the data exceeds the color map range.}
     \label{fig:tristrcomp}
\end{figure}

%
% \begin{figure}[htbp]
% \captionsetup[subfigure]{labelformat=empty}
%     
%      \begin{minipage}{0.29\linewidth}
%           \subfloat{\includegraphics[scale=.24, trim = 300 100 300 100,clip]{Final_mesh_tri}} \hfill
%      \end{minipage} %
%     \begin{minipage}{0.58\linewidth}
%      \subfloat[$V_h = \mc{RT}_0$, $X_h^0$]{\includegraphics[scale=.24, trim = 300 100 300 100,clip]{Final_RT0X0}}
%      \subfloat[$V_h = \mc{RT}_{0}$, $X_h^1$]{\includegraphics[scale=.24, trim = 300  100 300  100,clip]{Final_RT0X1}}\\
%      \subfloat[$V_h = \mc{BDM}_1$, $X_h^0$]{\includegraphics[scale=.24,trim =300 100 300 100,clip]{Final_BDM0X0}}
%      \subfloat[$V_h = \mc{BDM}_1$, $X_h^1$]{\includegraphics[scale=.24, trim = 300 100 300 100,clip]{Final_BDM0X1}}
%     \end{minipage}%
%     \begin{minipage}{0.1\linewidth}\hspace{-1em}
%     \subfloat{\includegraphics[scale=.27, trim = 760 90 160 120,clip]{Final_scale}}
%     \end{minipage}
%      \caption{Comparison between optimal transport interpolations of the densities in \eqref{eq:bcscos} for different spaces on a structured triangular mesh. Note that in the case $V_h = \mc{BDM}_1$, $X_h^0$, the data exceeds the color map range.}
%      \label{fig:tristrcomp}
% \end{figure}
%
%
%
\begin{figure}[htbp]
\captionsetup[subfigure]{labelformat=empty}
     \begin{minipage}{0.29\linewidth}
          \subfloat{\includegraphics[scale=.2, trim =  250 50 250 50,clip]{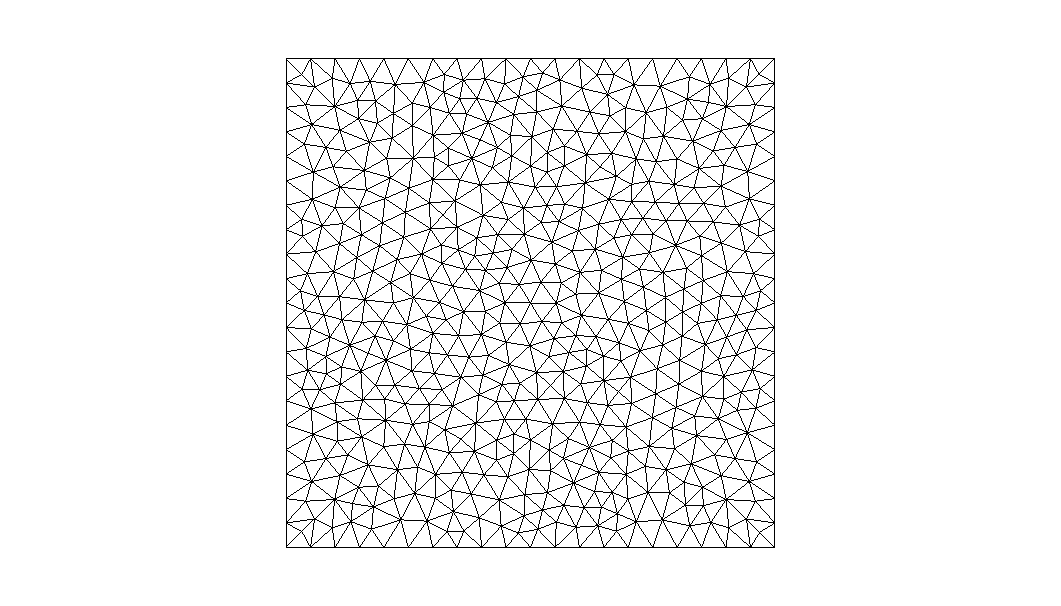}} \hfill
     \end{minipage} %
    \begin{minipage}{0.58\linewidth}
     \subfloat[{$\mc{RT}_0, X_h^0$}]{\includegraphics[scale=.2, trim =  250 50 250 50,clip]{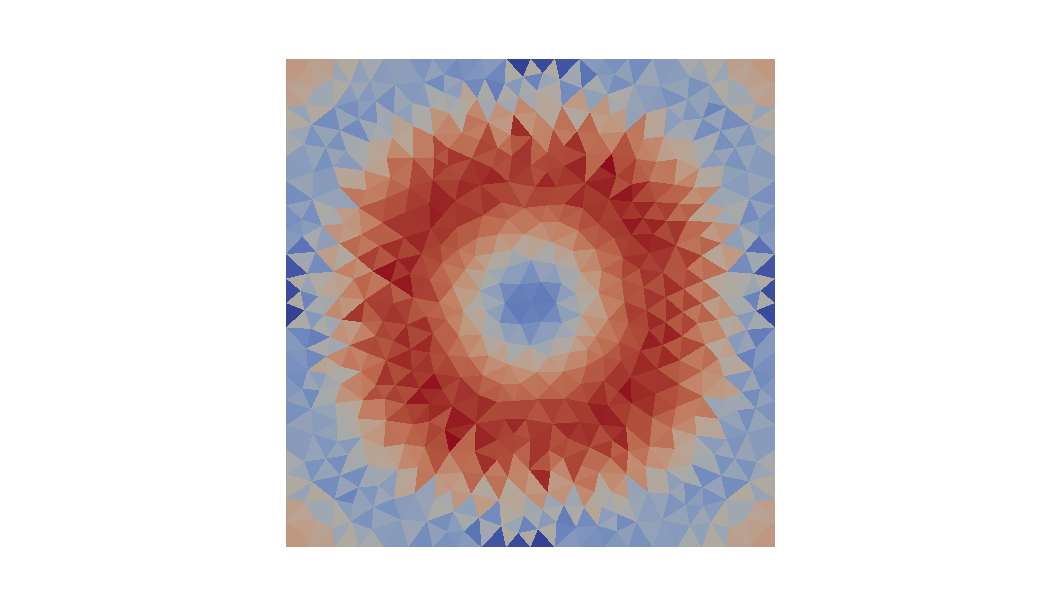}}
     \subfloat[{$\mc{RT}_{0}, X_h^1$}]{\includegraphics[scale=.2, trim =  250 50 250 50,clip]{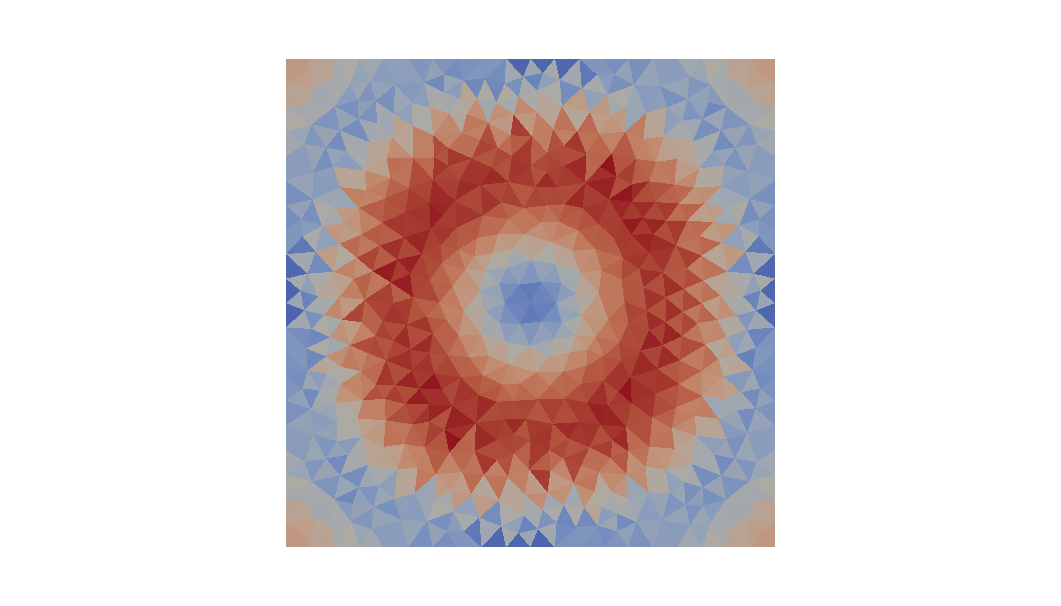}}\\
     \subfloat[{$\mc{BDM}_1, X_h^0$}]{\includegraphics[scale=.2, trim =  250 50 250 50,clip]{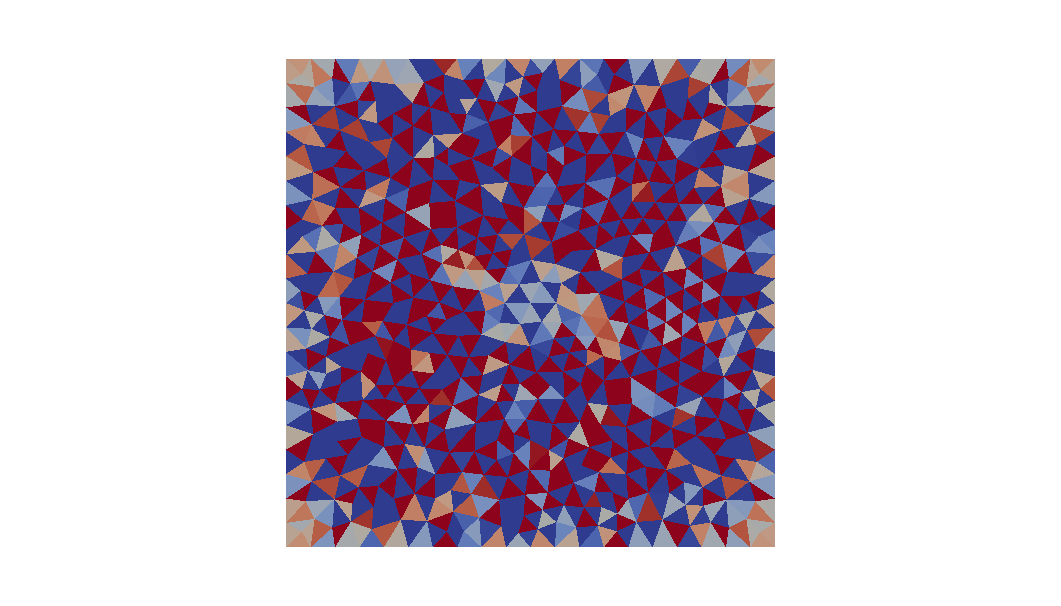}}
     \subfloat[{$\mc{BDM}_1, X_h^1$}]{\includegraphics[scale=.2, trim =  250 50 250 50,clip]{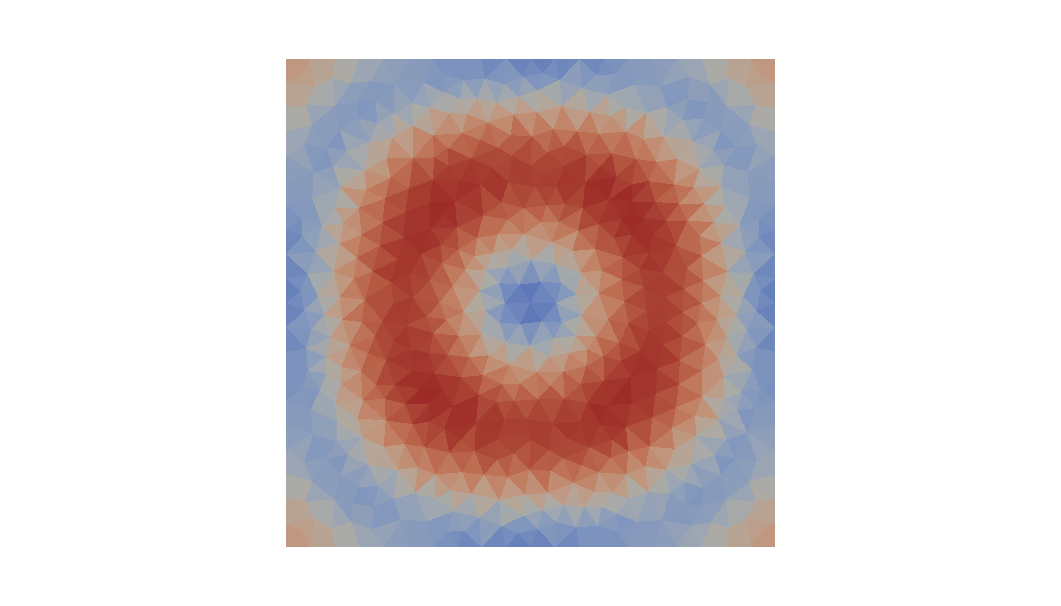}}
    \end{minipage}%
    \begin{minipage}{0.1\linewidth}\hspace{-1em}
    \subfloat{\includegraphics[scale=.27, trim = 760 90 160 120,clip]{Final_scale}}
    \end{minipage}
     \caption{Comparison between optimal transport interpolations of the densities in \eqref{eq:bcscos} for different spaces on an unstructured triangular mesh. Note that in the case $V_h = \mc{BDM}_1$, $X_h^0$, the data exceeds the color map range.}
     \label{fig:triunstrcomp}
\end{figure}

\begin{figure}[htbp]
\captionsetup[subfigure]{labelformat=empty}
     \begin{minipage}{0.28\linewidth}
          \subfloat{\includegraphics[scale=.2, trim =  250 30 250 50,clip]{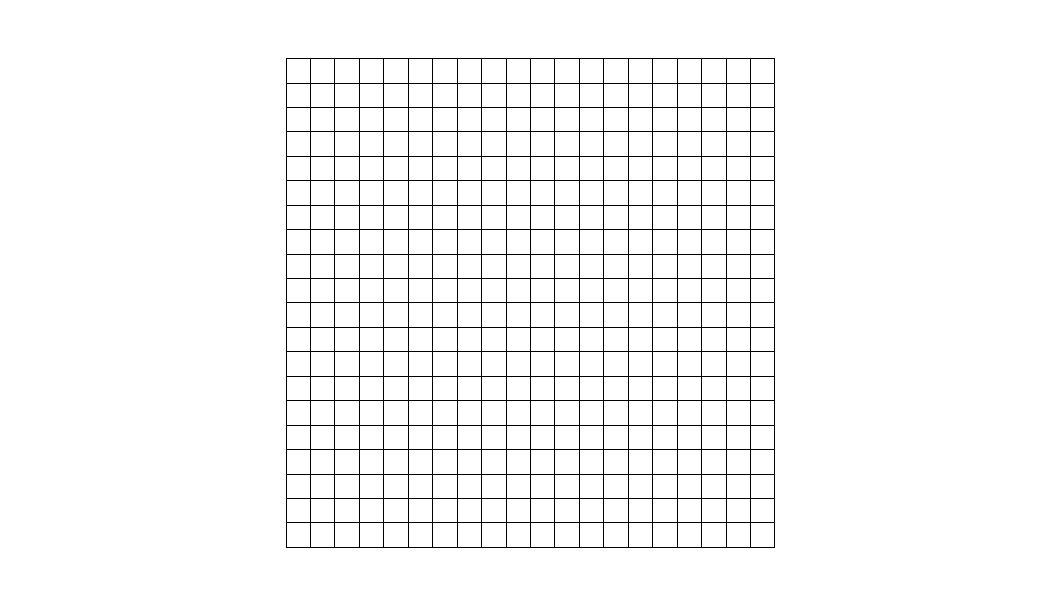}} \hfill
     \end{minipage} %
    \begin{minipage}{0.59\linewidth}
     \subfloat[\scriptsize{$\mc{RT}_{[0]}, X_h^0$}]{\includegraphics[scale=.2, trim =  250 50 250 50,clip]{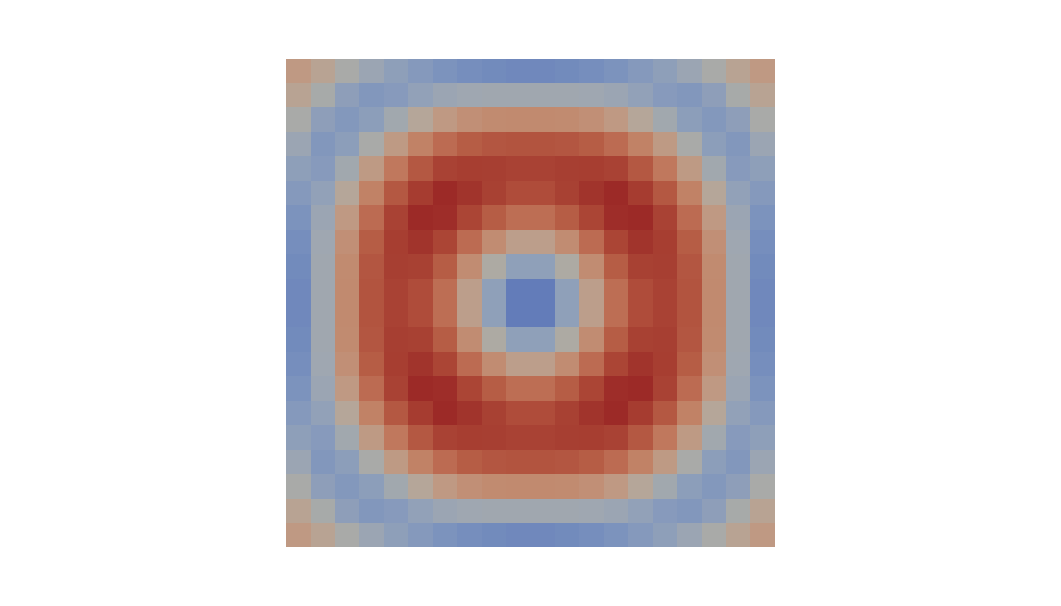}}
     \subfloat[\scriptsize{$\mc{RT}_{[0]}, X_h^1$}]{\includegraphics[scale=.2, trim =  250 50 250 50,clip]{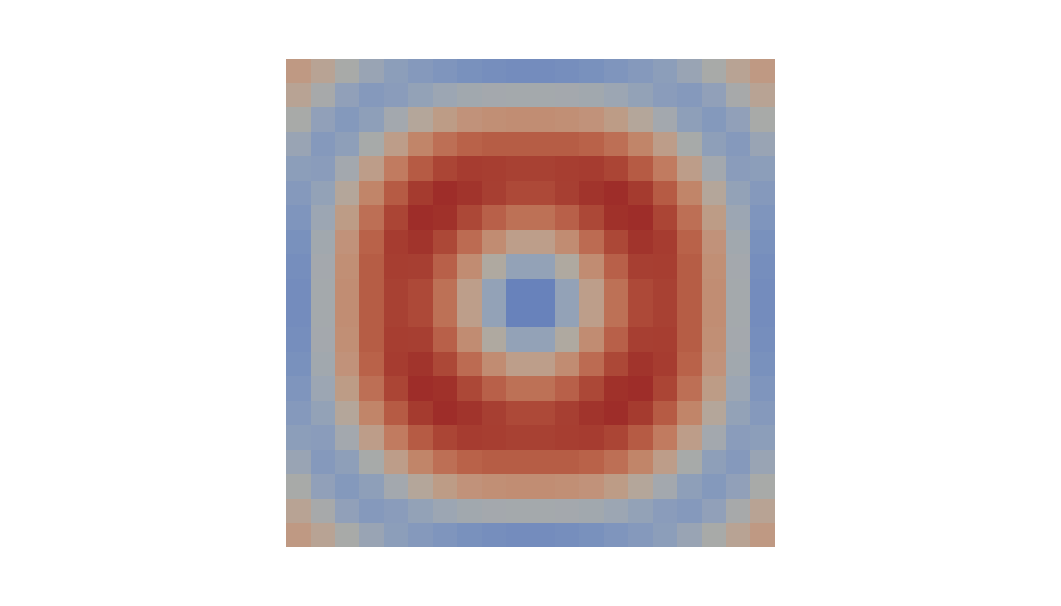}}
    \end{minipage}%
    \begin{minipage}{0.1\linewidth}\hspace{-1em}
    \subfloat{\includegraphics[scale=.27, trim = 760 90 160 120,clip]{Final_scale}}
    \end{minipage}
     \caption{Comparison between optimal transport interpolations of the densities in \eqref{eq:bcscos} for different spaces on an uniform Cartesian mesh. }
     \label{fig:quadcomp}
\end{figure}

% \begin{figure}[htbp]
% \captionsetup[subfigure]{labelformat=empty}
%      \centering
%            \hfill     \subfloat{\includegraphics[scale=.33, trim = 750 670 500 0]{scale}}\\
%      \subfloat[$V_h = \mc{RT}_0$, $X_h^0$]{\includegraphics[scale=.25, trim = 230 100 200 100,clip]{RT_X0}}
%      \subfloat[$V_h = \mc{BDM}_1$, $X_h^0$]{\includegraphics[scale=.25,trim = 230 100 200 100,clip]{BDM_X0}}
%      \subfloat[$V_h = \mc{RT}_{[0]}$, $X_h^0$]{\includegraphics[scale=.25, trim = 230 100 200 100,clip]{RTCF_X0}}\\
%      \subfloat[$V_h = \mc{RT}_0$, $X_h^1$]{\includegraphics[scale=.25,trim = 230 100 200 100,clip]{RT_X1}}
%      \subfloat[$V_h = \mc{BDM}_1$, $X_h^1$]{\includegraphics[scale=.25, trim = 230 100 200 100,clip]{BDM_X1}}
%      \subfloat[$V_h = \mc{RT}_{[0]}$, $X_h^1$]{\includegraphics[scale=.25,trim = 230 100 200 100,clip]{Final_BDM0X0}}
%      \caption{Comparison between optimal transport interpolations of the density in \eqref{eq:bcscos} for different spaces. Note that in the case $V_h = \mc{BDM}_1$, $X_h^1$, the data exceeds the color map range.}
%      \label{fig:Gaussiancomp}
% \end{figure}
% %\pagebrea

\begin{figure}[h]
\captionsetup[subfigure]{labelformat=empty}
     \centering
     \subfloat[$\mc{RT}_0$(str.)]{\includegraphics[scale=.65,trim = 0 0 5 0,clip]{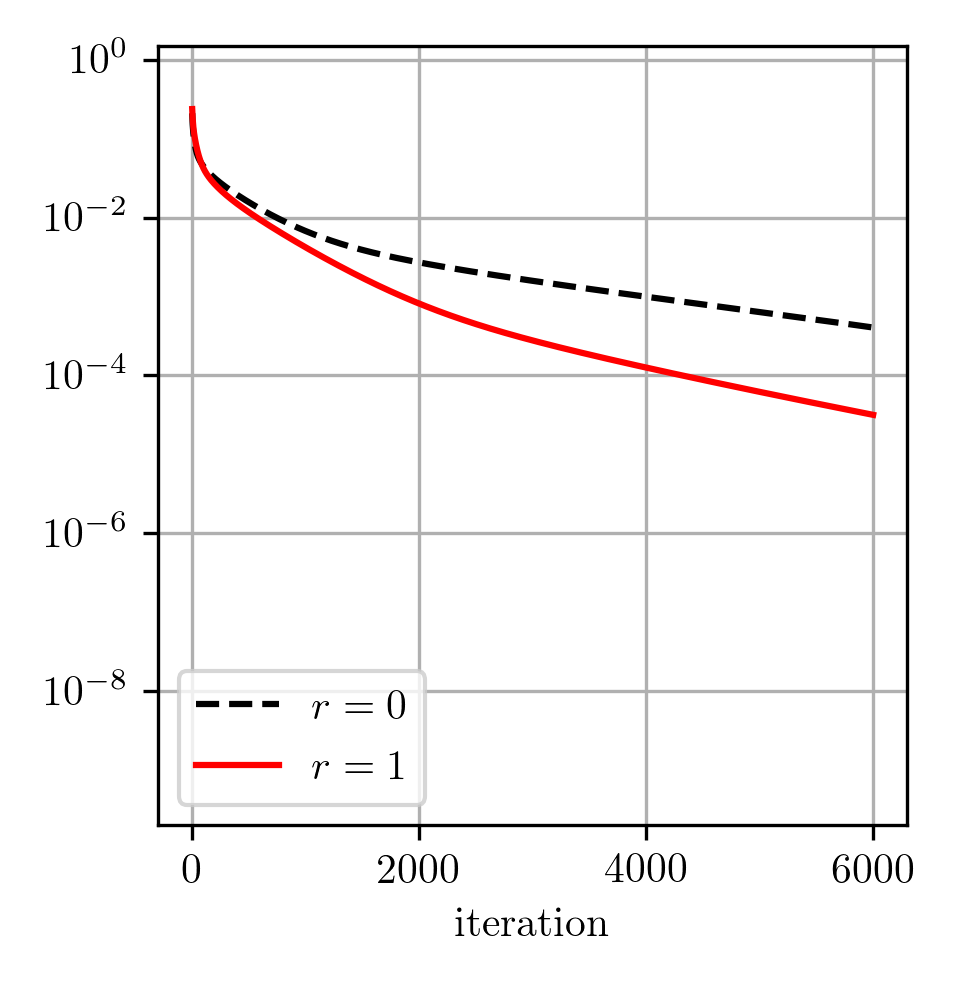}}
     \subfloat[$ \mc{BDM}_1$(str.)]{\includegraphics[scale=.65,trim = 32 0 5 0 ,clip]{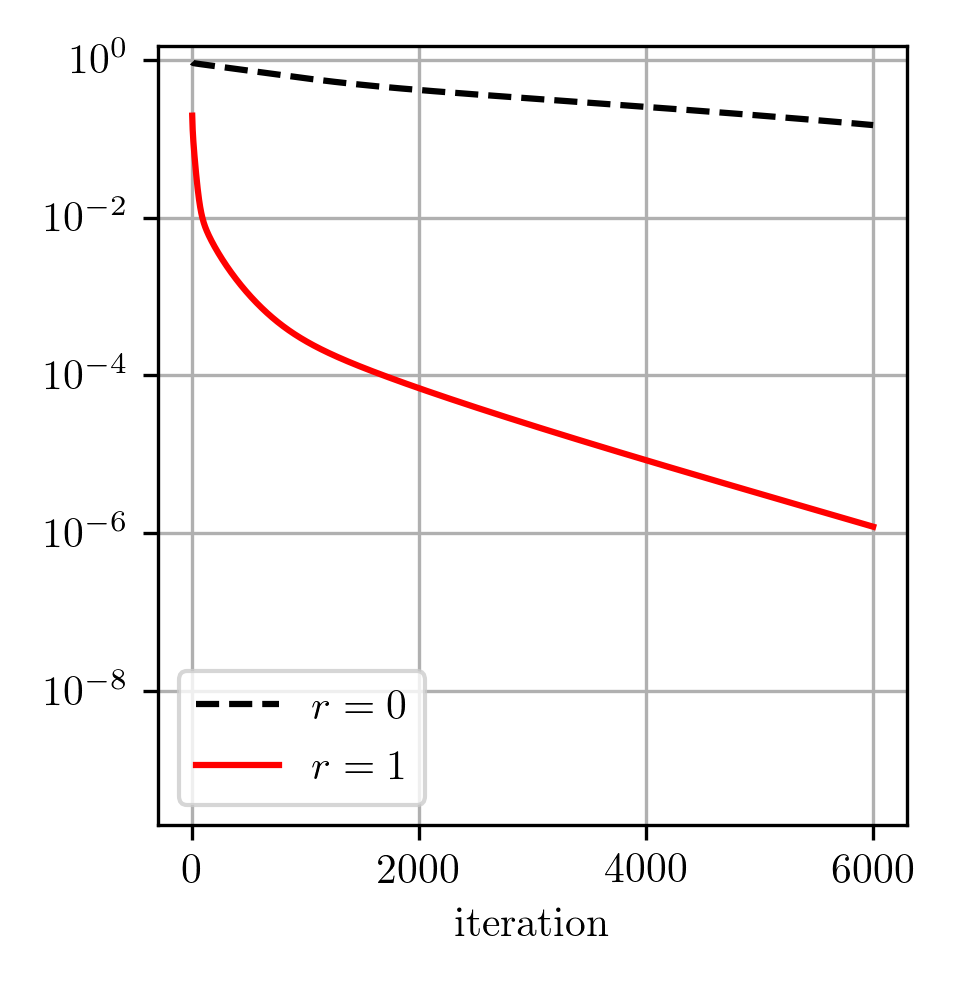}}
     \subfloat[$\mc{RT}_{[0]}$]{\includegraphics[scale=.65,trim = 32 0 5 0,clip]{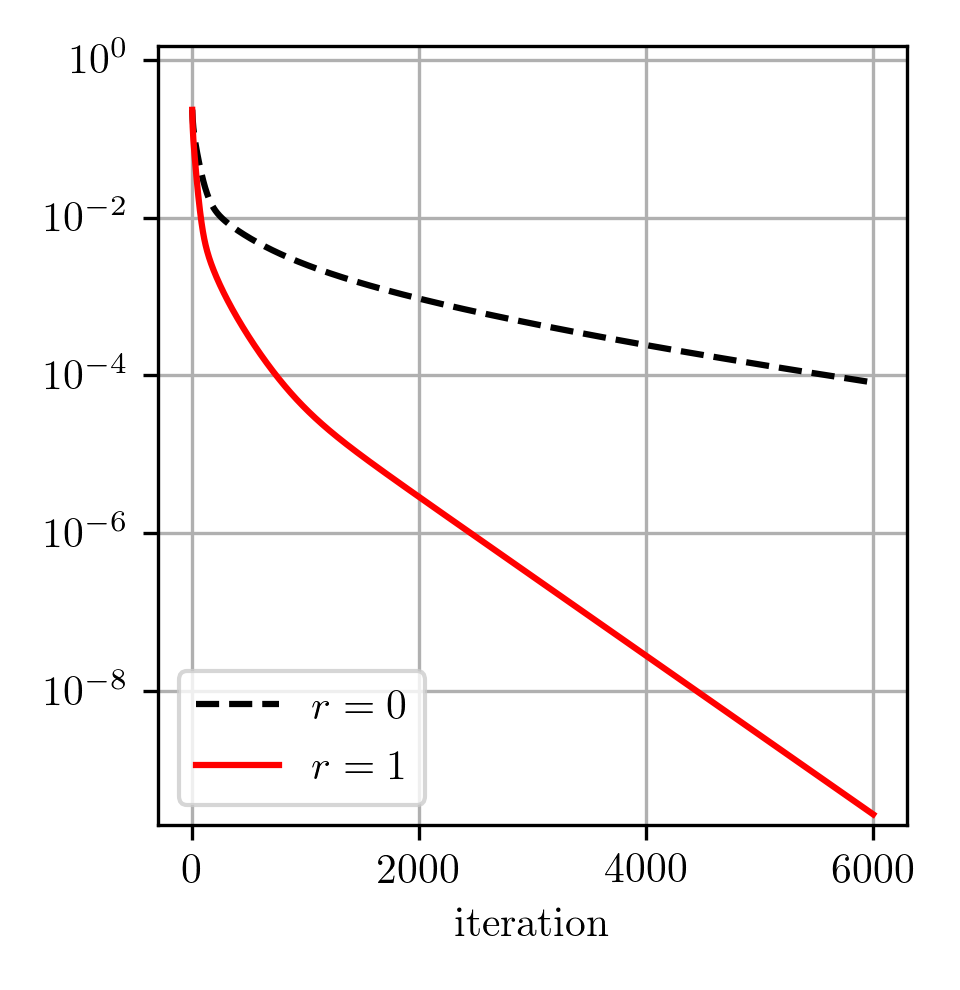}}\\
    \subfloat[$\mc{RT}_0$(unstr.)]{\includegraphics[scale=.65,trim = 0 0 5 0,clip]{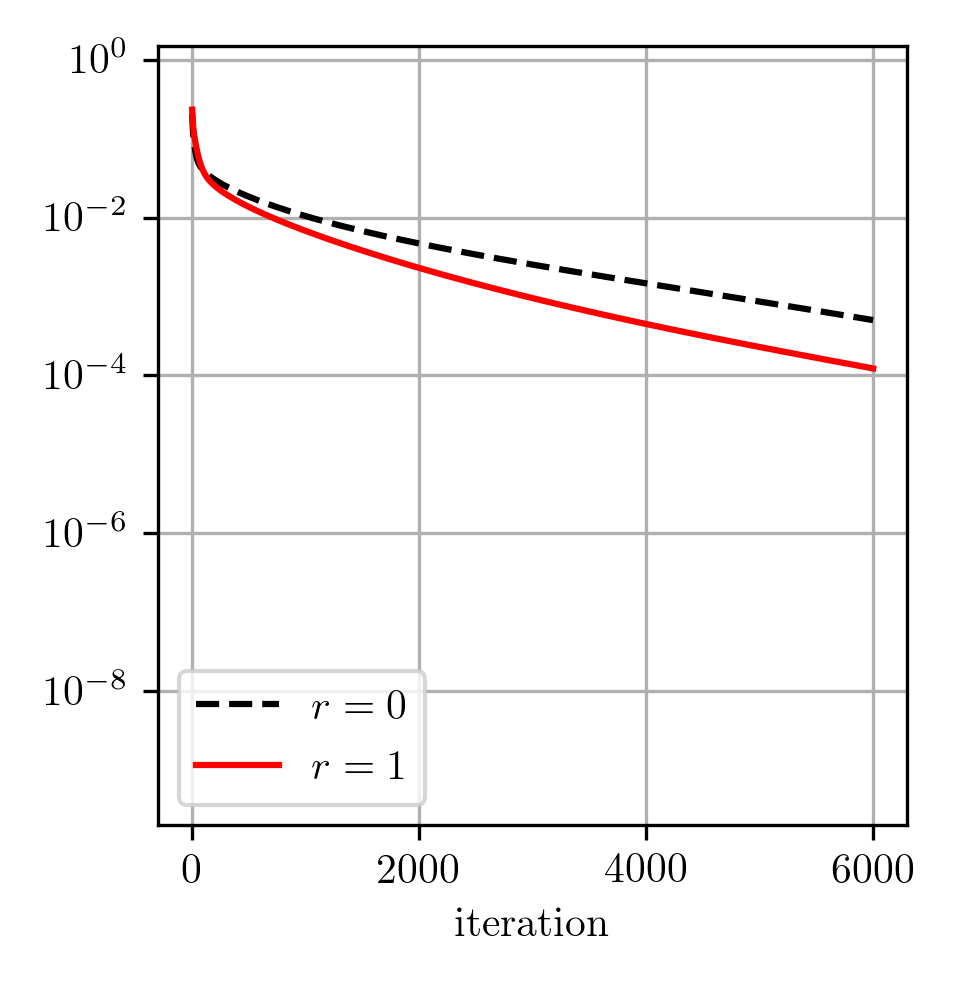}}
     \subfloat[$\mc{BDM}_1$(unstr.)]{\includegraphics[scale=.65,trim = 32 0 5 0 ,clip]{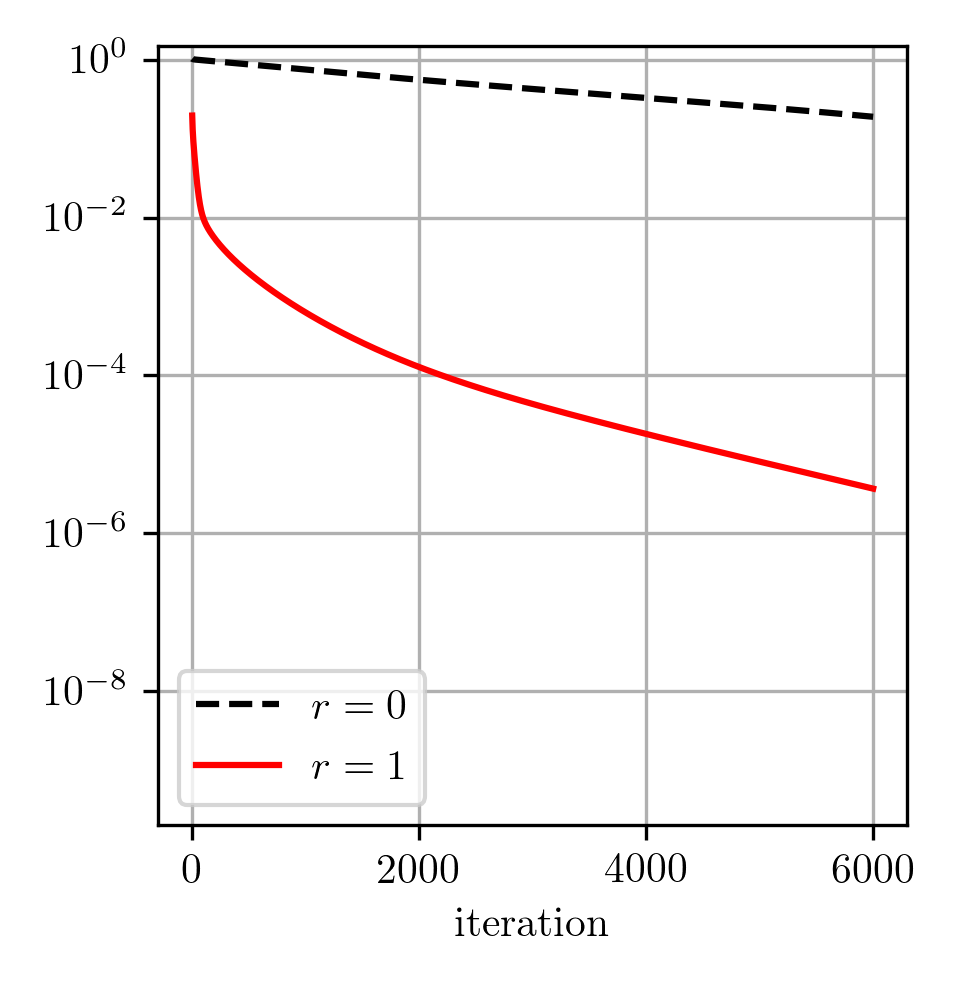}}
     \caption{Convergence of the proximal splitting algorithm measured by $\|\sigma_{n}-\sigma^*\|_{L^2(\Omega)}$  for different spaces $V_h$ and $X_h^r$ on the structured (str.) and unstructured (unstr.) triangular mesh, and on the Cartesian mesh.}
     \label{fig:Gaussianconv}
\end{figure}

\subsection{Non-convex domain}
We now consider a non-convex polygonal domain $D$, with the spatial mesh $\mc{T}_h$ represented in figure \ref{fig:Zmesh} and $\tau \coloneqq |t_{i+1} - t_i| = 1/30$. Note that even if the case of a non-convex domain is beyond the domain of applicability of the convergence results presented in this paper, our scheme is still well-defined for this case. The boundary conditions are given by 
\[
\rho_0(x) = \exp \left(-\frac{|x-x_0|^2}{2 s^2}\right) 
, \quad \rho_1(x) = \exp \left(-\frac{|x-x_1|^2}{2 s^2} \right) \,,
\]
with $s=0.1$, $x_0 = (0.5,0.1)$ and $x_1 = (0.5, 0.9)$. Such boundary conditions are illustrated in figure \ref{fig:Zmesh}. In this case, the exact density interpolation is not absolutely continuous, since mass concentrates on the segment connecting the two non-convex corners of the domain. Note that we have therefore refined the mesh along the diagonal where we expect the mass to concentrate.

In figure \ref{fig:Zevol}, \ref{fig:Zevolreg} and \ref{fig:ZevolregL} we show the density evolution up to time $t=0.5$ (the other half of the time evolution being symmetric in space given the boundary conditions and the domain shape) for the non-regularized case, the $H^1$ regularization and the $L^2$ regularization, respectively.  For both regularizations the density profile appears to be smoothened, but only the $H^1$ regularization avoids concentration at the corners.

The proximal operator of the projection on the continuity equation is more expensive computationally for the $H^1$ regularization than for the other two cases, since we have to solve a larger mixed system at each iteration. However, for both regularizations, the proximal splitting algorithm itself converges much faster than the non-regularized case, as it can be seen in figure \ref{fig:Zconv}.

%\newpage

\begin{figure}[h]
\captionsetup[subfigure]{labelformat=empty}
     \centering
     \subfloat[]{\includegraphics[scale=.29,trim = 350 110 312 120,clip]{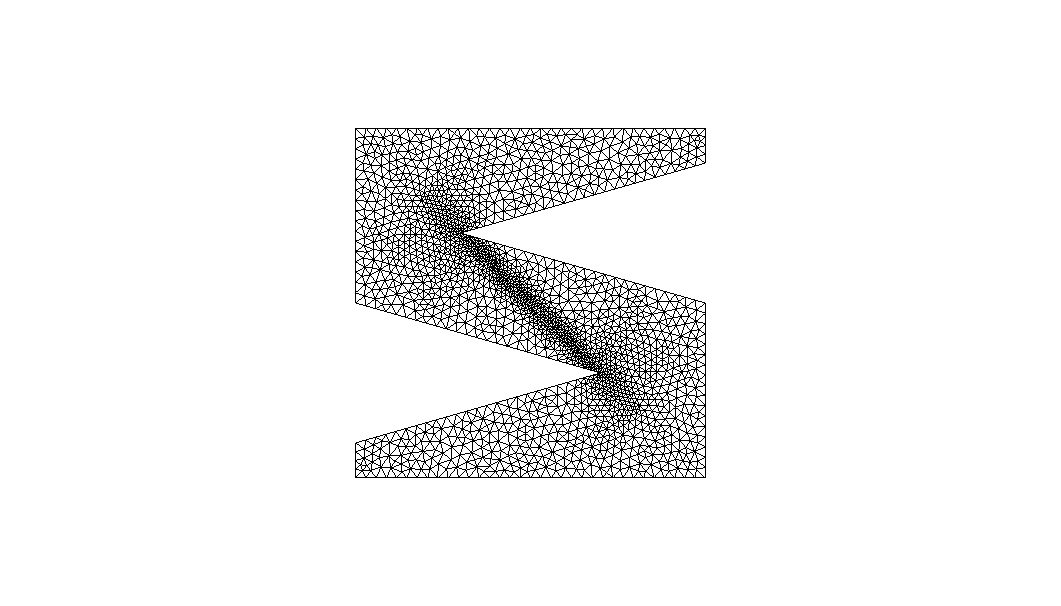}}
     \subfloat[$t = 0$]{\includegraphics[scale=.245,trim = 345 110 303 120,clip]{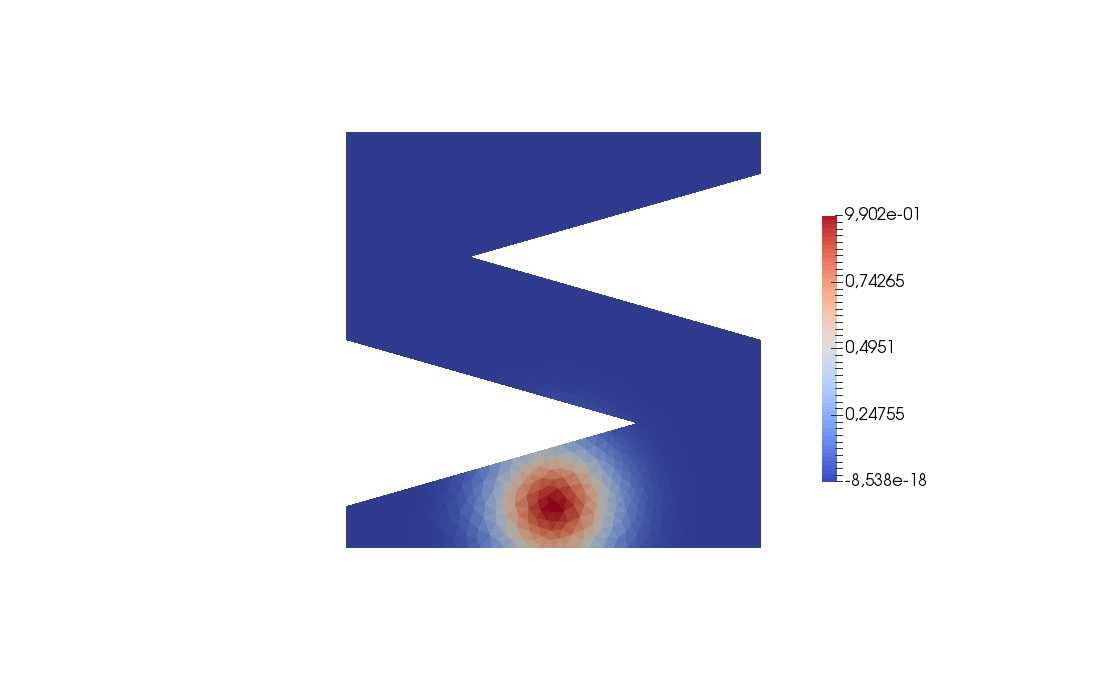}}
     \subfloat[$t = 1$]{\includegraphics[scale=.26,trim =  353 123 353 120,clip]{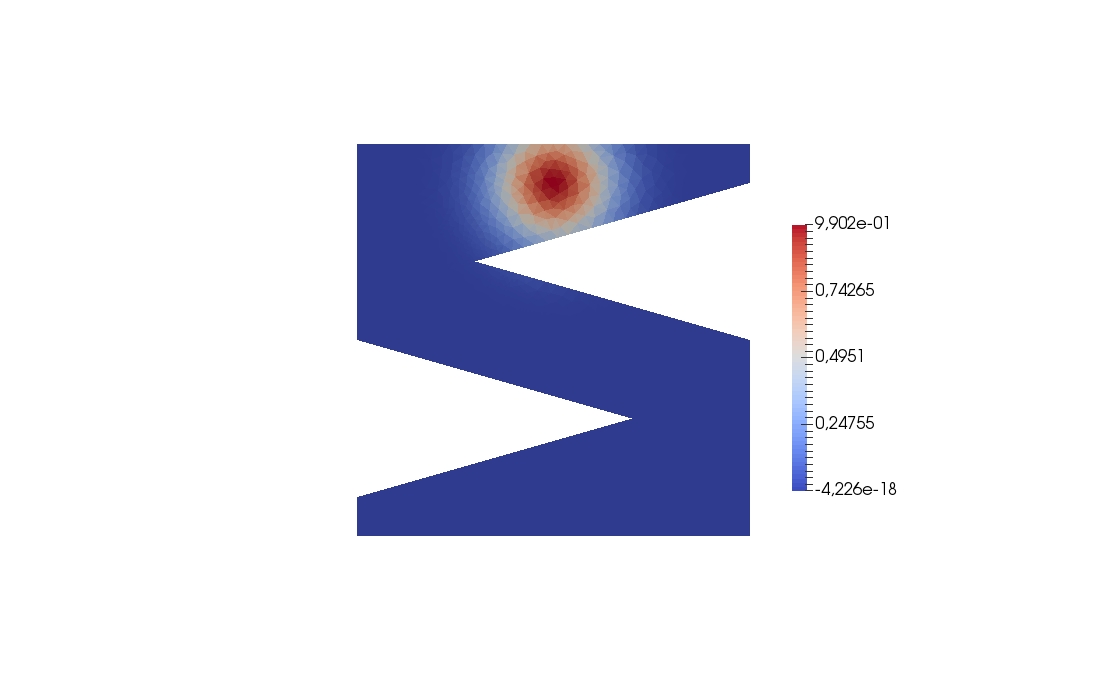}}     
     \subfloat[]{\includegraphics[scale=.3,trim =  760 130 200 120,clip]{Z10reg.jpg}}
          \caption{Mesh, and initial and final density for the non-convex domain test.}
     \label{fig:Zmesh}
\end{figure}

\begin{figure}[h]
\captionsetup[subfigure]{labelformat=empty}
     \centering
     \subfloat[$t = 0$]{\includegraphics[scale=.25,trim = 350 110 300 120,clip]{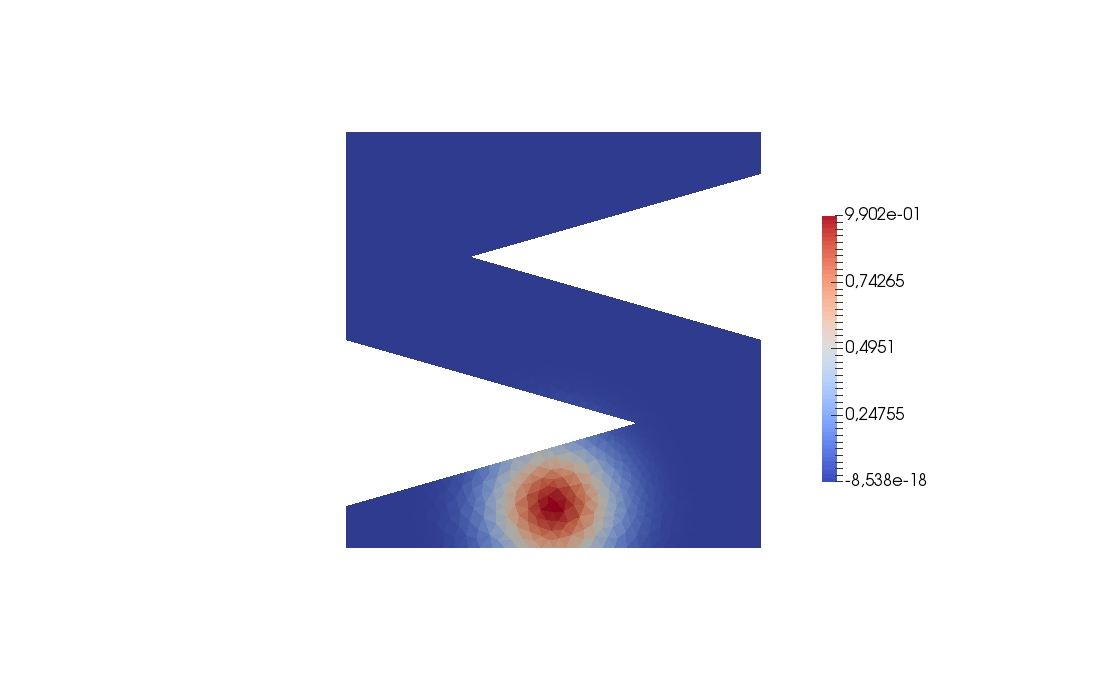}}
     \subfloat[$t = 0.1$]{\includegraphics[scale=.25,trim =  350 110 300 120,clip]{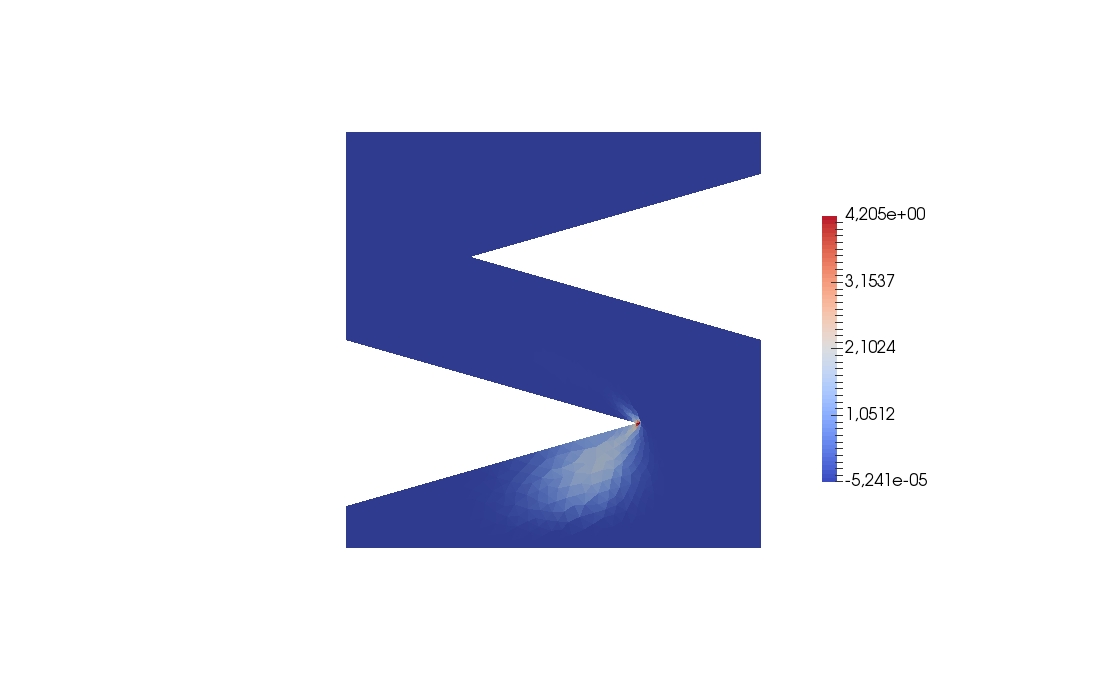}}
     \subfloat[$t = 0.2$]{\includegraphics[scale=.25,trim =  350 110 300 120,clip]{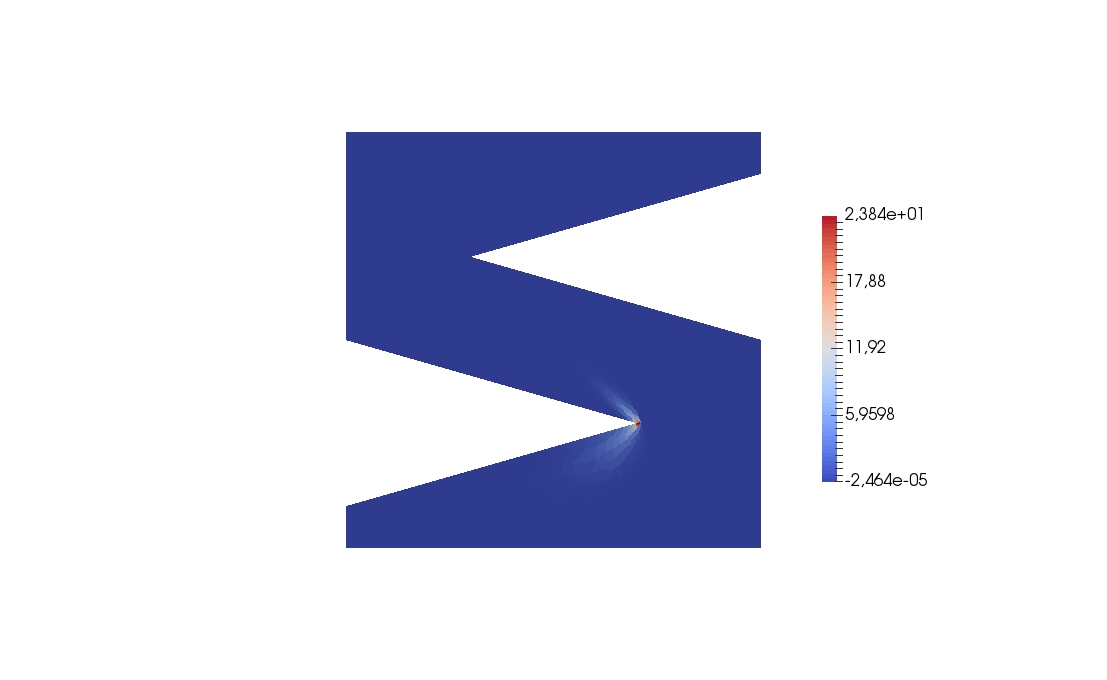}}\\
     \subfloat[$t = 0.3$]{\includegraphics[scale=.25,trim =  350 110 300 120,clip]{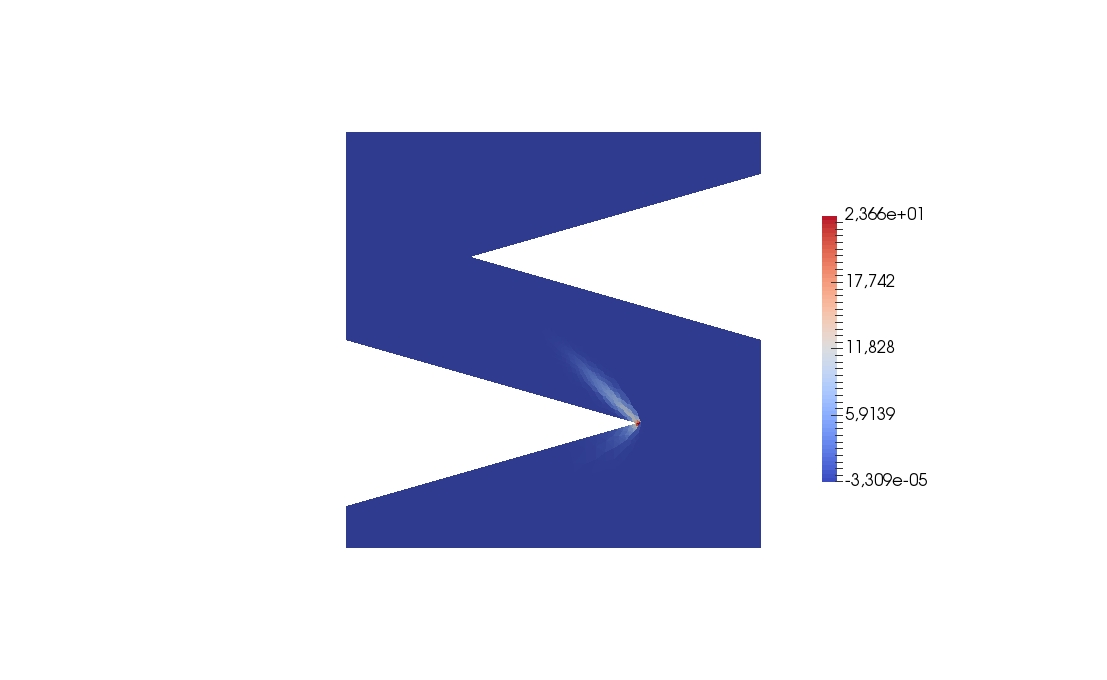}}
     \subfloat[$t = 0.4$]{\includegraphics[scale=.25,trim =  350 110 300 120,clip]{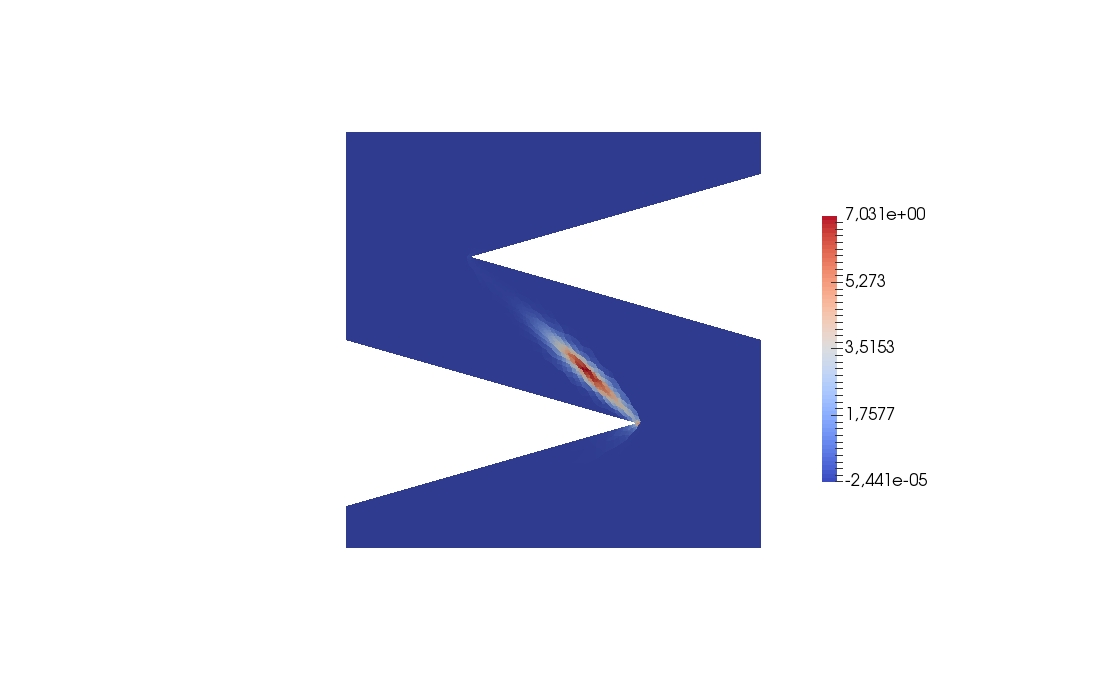}}
     \subfloat[$t = 0.5$]{\includegraphics[scale=.25,trim =  350 110 300 120,clip]{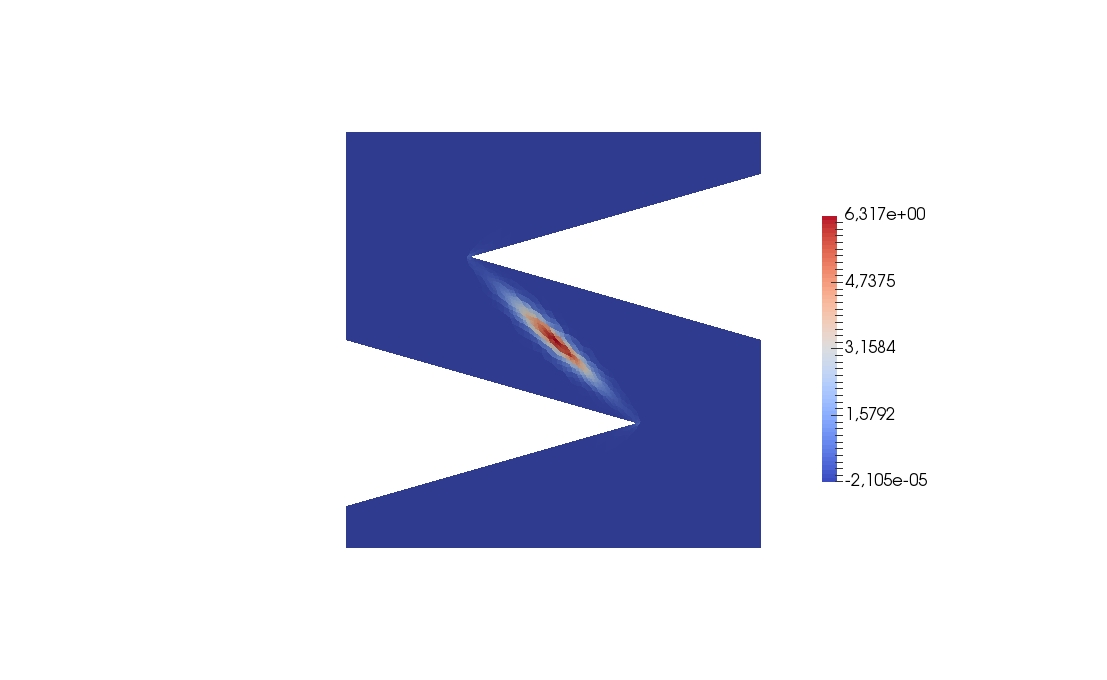}}
     \caption{Density evolution on the non-convex domain without regularization, $V_h = \mc{RT}_0$, $X_h^1$ (color scale is rescaled to fit data range).}
     \label{fig:Zevol}
\end{figure}

\begin{figure}[h]
\captionsetup[subfigure]{labelformat=empty}
     \centering
     \subfloat[$t = 0$]{\includegraphics[scale=.25,trim = 350 110 300 120,clip]{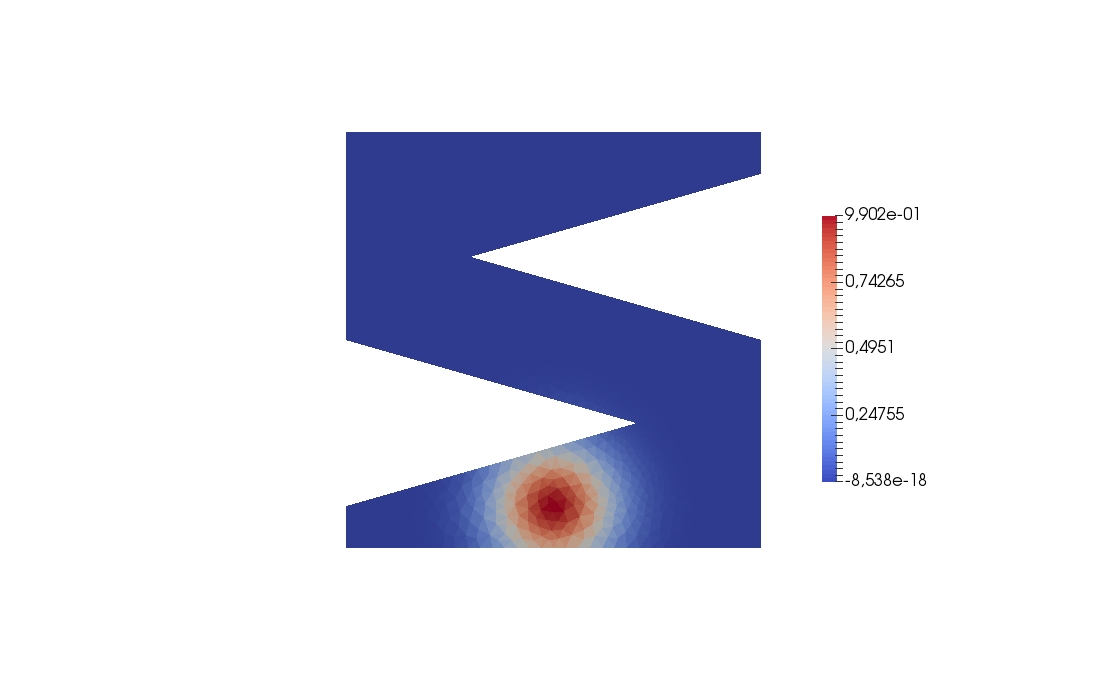}}
     \subfloat[$t = 0.1$]{\includegraphics[scale=.25,trim =  350 110 300 120,clip]{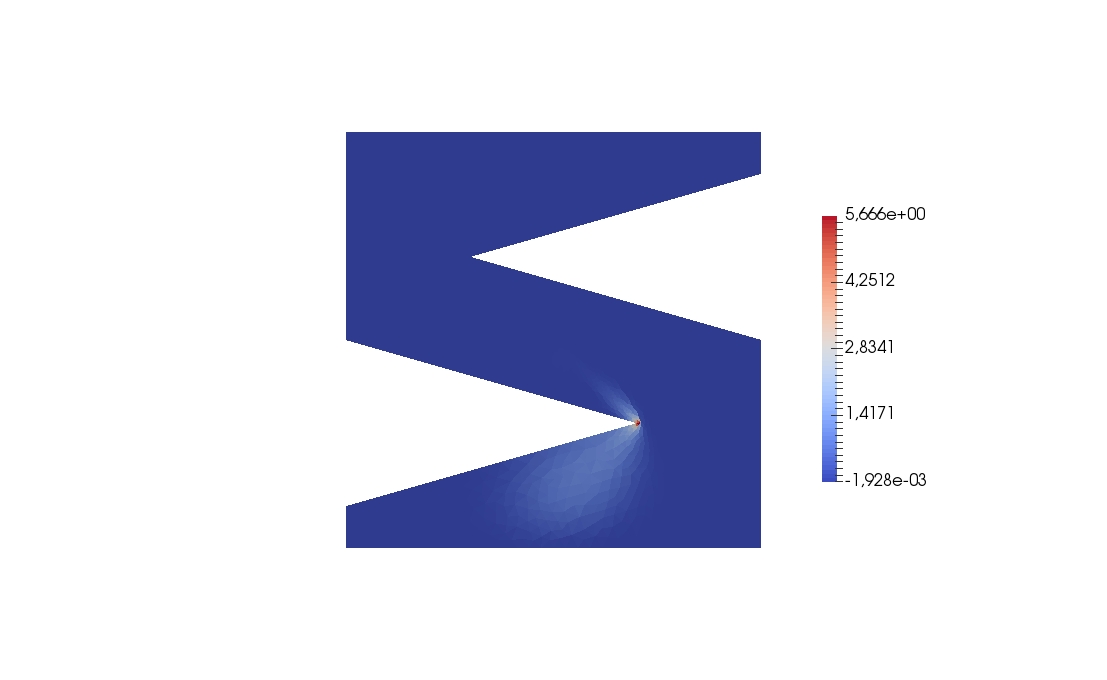}}
     \subfloat[$t = 0.2$]{\includegraphics[scale=.25,trim =  350 110 300 120,clip]{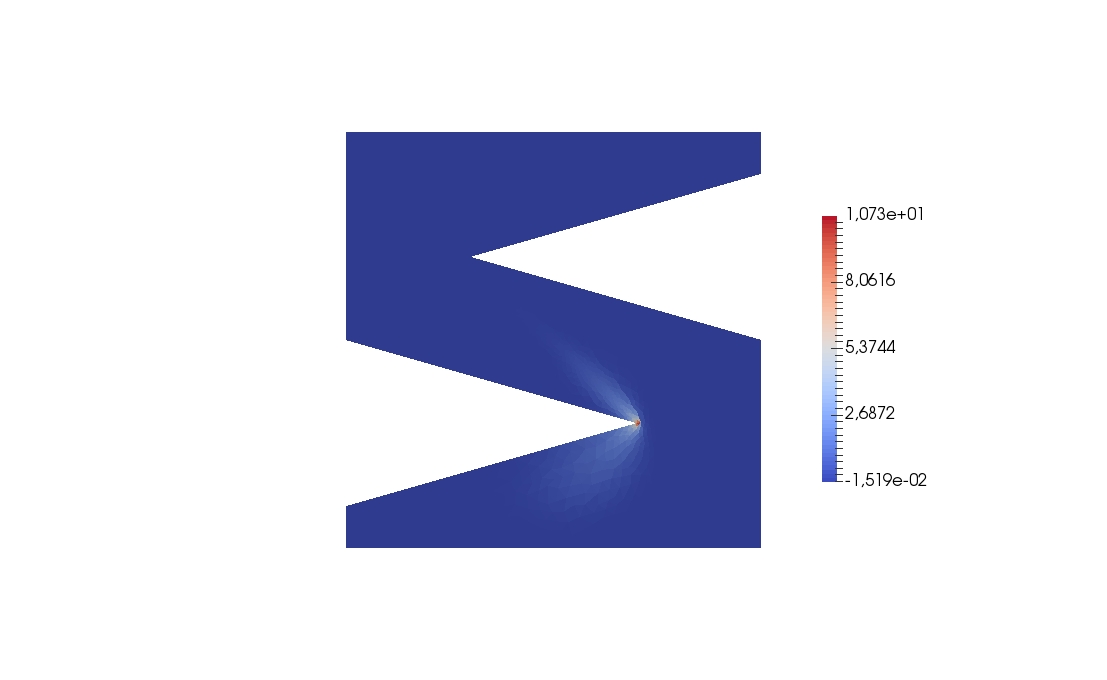}}\\
     \subfloat[$t = 0.3$]{\includegraphics[scale=.25,trim =  350 110 300 120,clip]{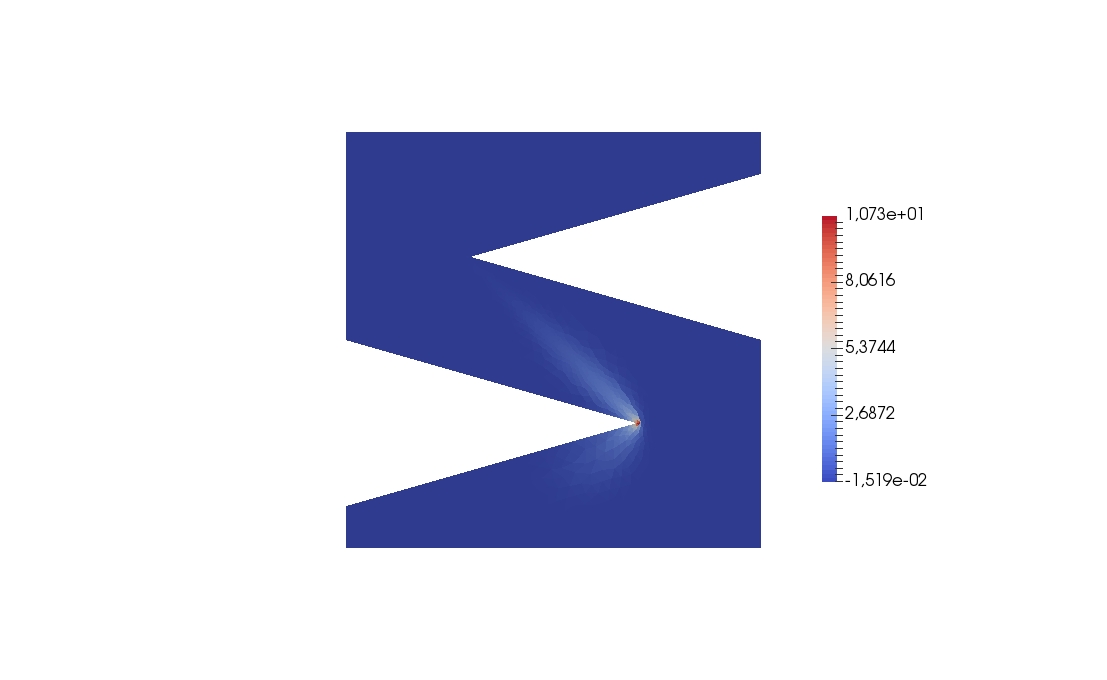}}
     \subfloat[$t = 0.4$]{\includegraphics[scale=.25,trim =  350 110 300 120,clip]{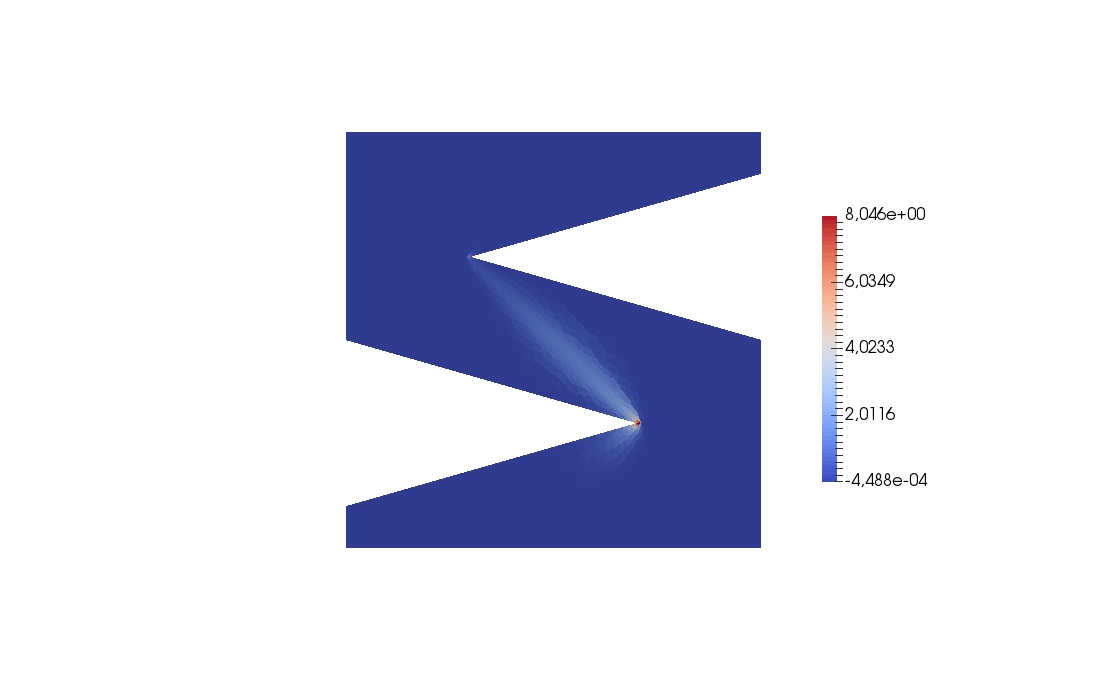}}
     \subfloat[$t = 0.5$]{\includegraphics[scale=.25,trim =  350 110 300 120,clip]{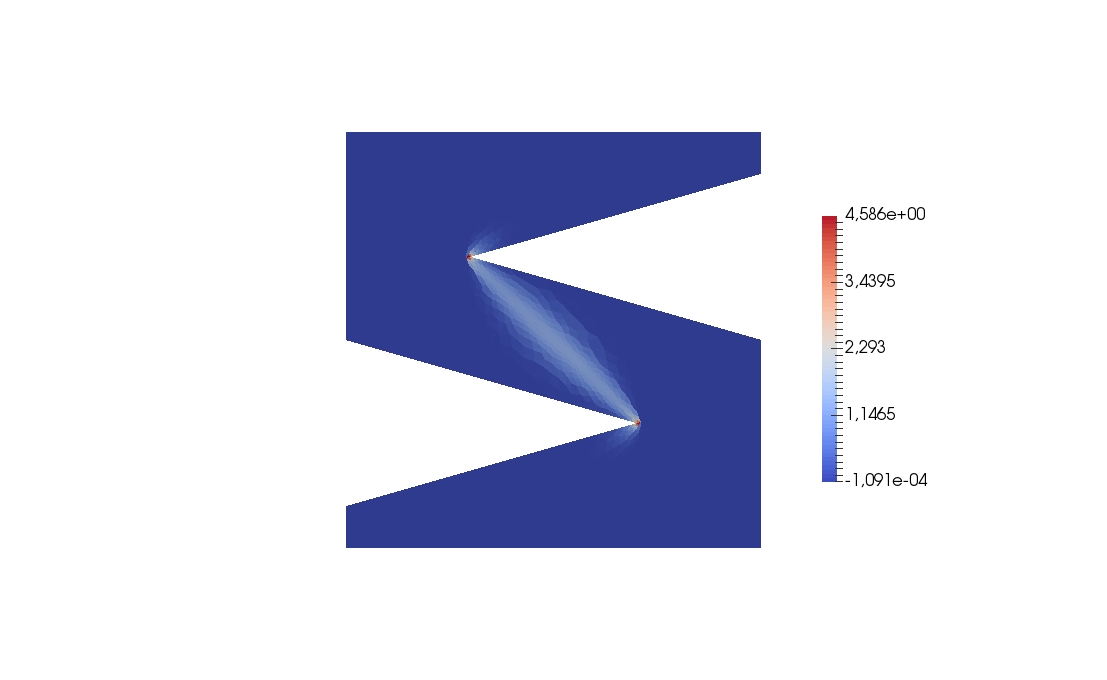}}
     \caption{Density evolution on the non-convex domain with $L^2$ regularization, $\alpha=0.002$, $V_h = \mc{RT}_0$, $X_h^1$ (color scale is rescaled to fit data range).}
     \label{fig:ZevolregL}
\end{figure}

\begin{figure}[h]
\captionsetup[subfigure]{labelformat=empty}
     \centering
     \subfloat[$t = 0$]{\includegraphics[scale=.25,trim = 350 110 300 120,clip]{Z00reg.jpg}}
     \subfloat[$t = 0.1$]{\includegraphics[scale=.25,trim =  350 110 300 120,clip]{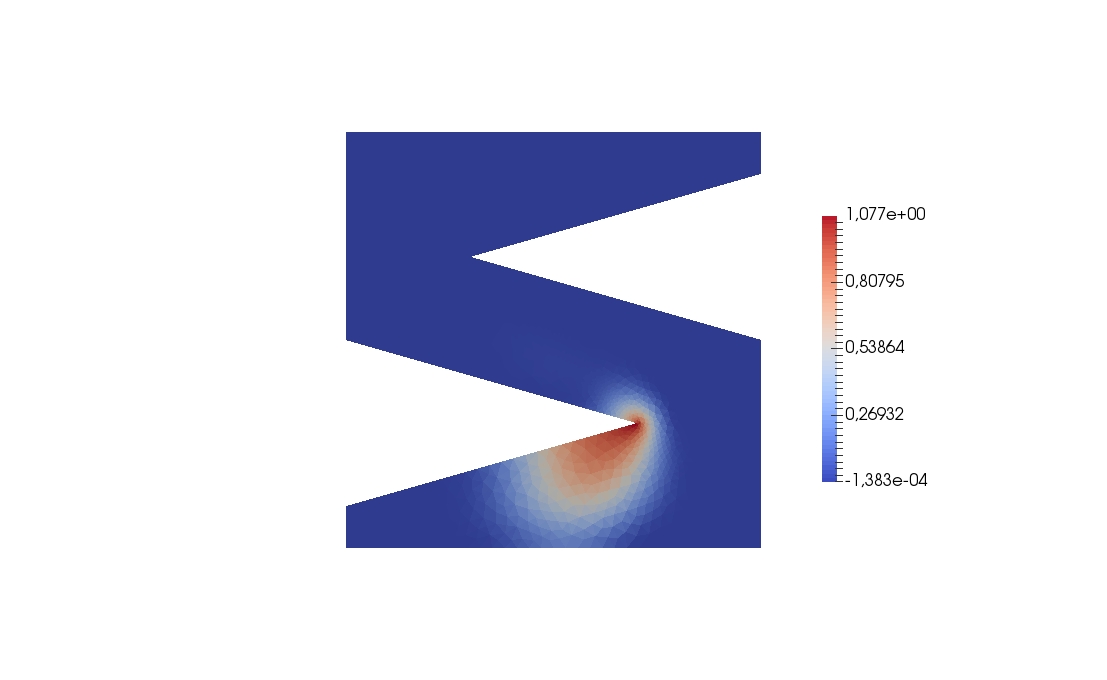}}
     \subfloat[$t = 0.2$]{\includegraphics[scale=.25,trim =  350 110 300 120,clip]{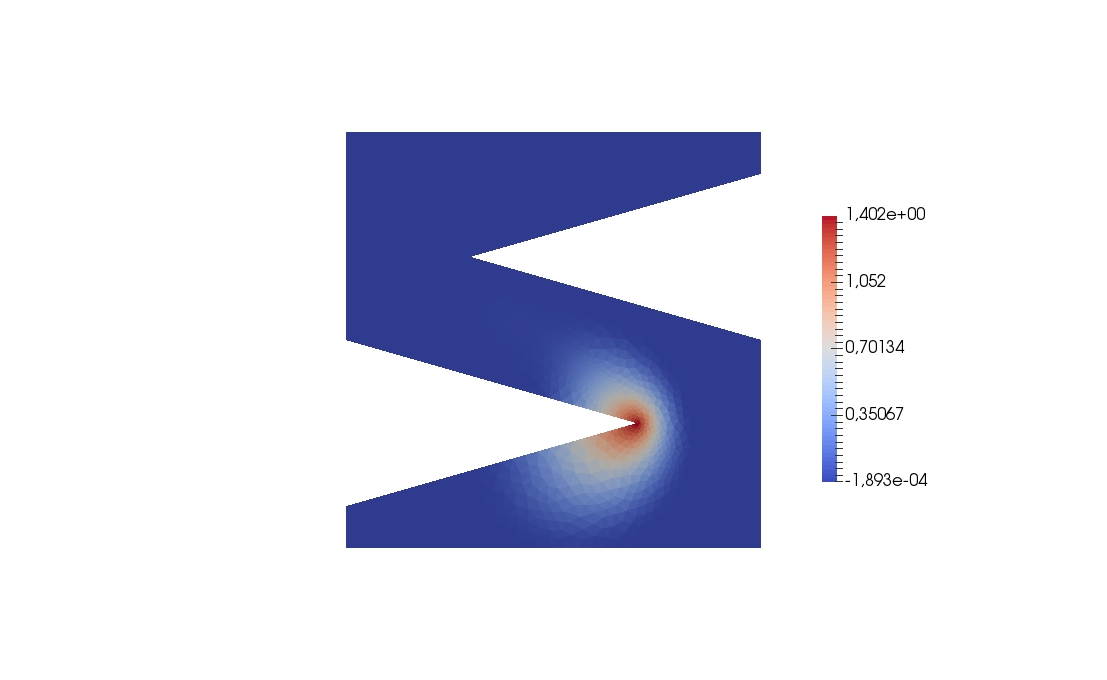}}\\
     \subfloat[$t = 0.3$]{\includegraphics[scale=.25,trim =  350 110 300 120,clip]{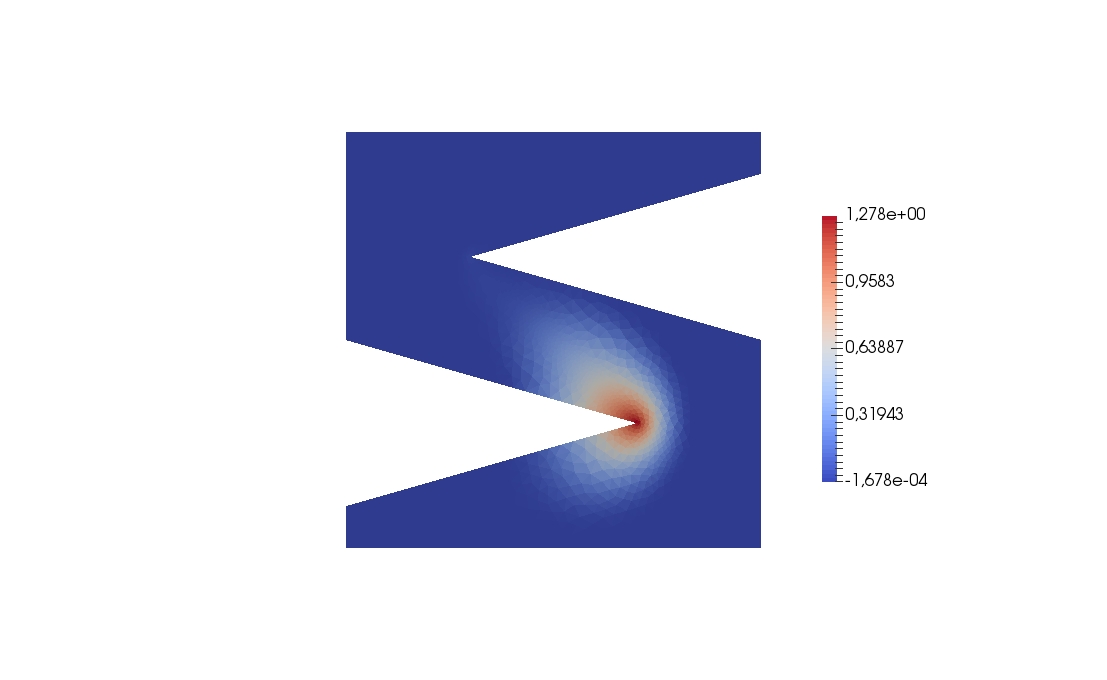}}
     \subfloat[$t = 0.4$]{\includegraphics[scale=.25,trim =  350 110 300 120,clip]{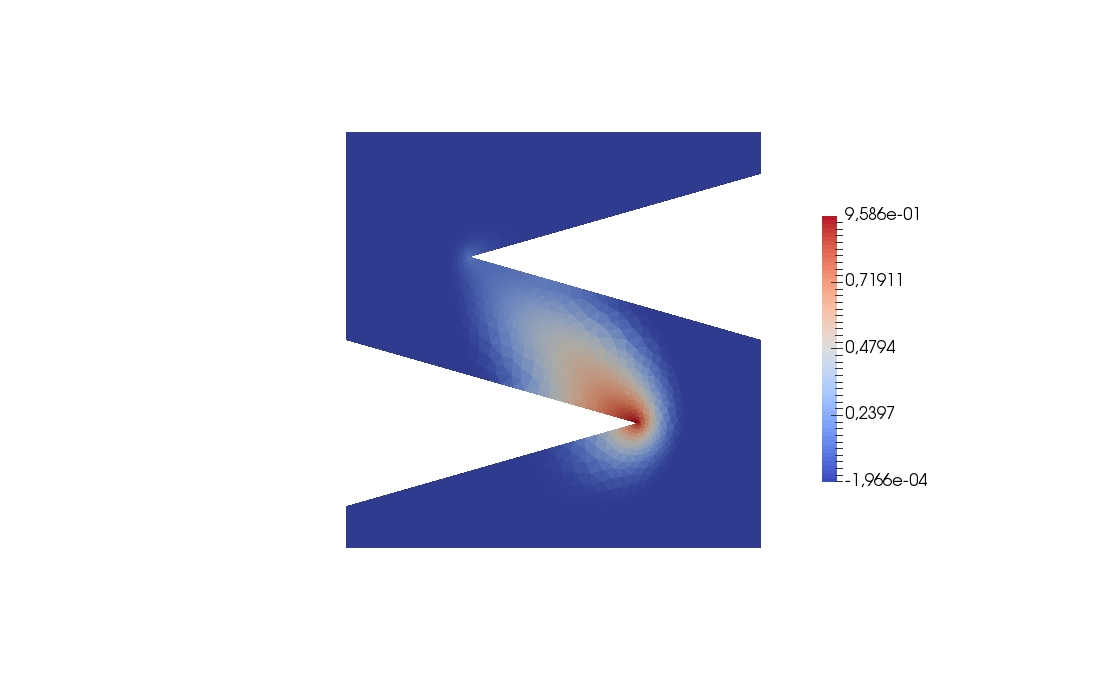}}
     \subfloat[$t = 0.5$]{\includegraphics[scale=.25,trim =  350 110 300 120,clip]{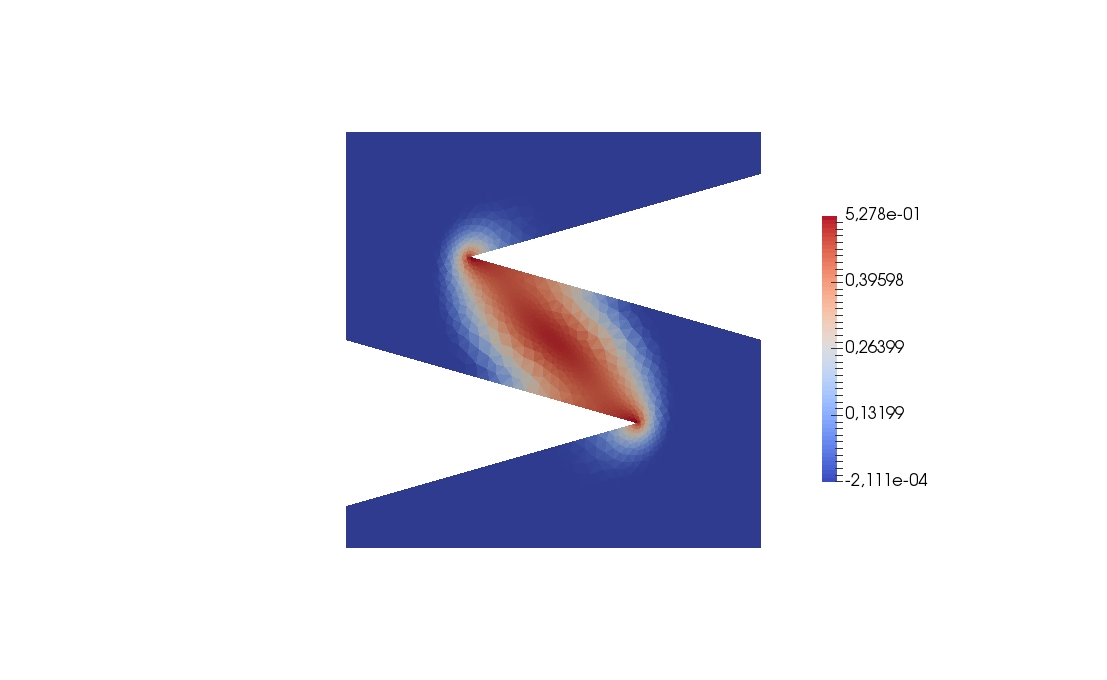}}
     \caption{Density evolution on the non-convex domain with $H^1$ regularization, $\alpha=0.002$, $V_h = \mc{RT}_0$, $X_h^1$ (color scale is rescaled to fit data range).}
     \label{fig:Zevolreg}
\end{figure}

\begin{figure}[h]
\centering
\includegraphics[scale=.65,trim = 0 0 5 0,clip]{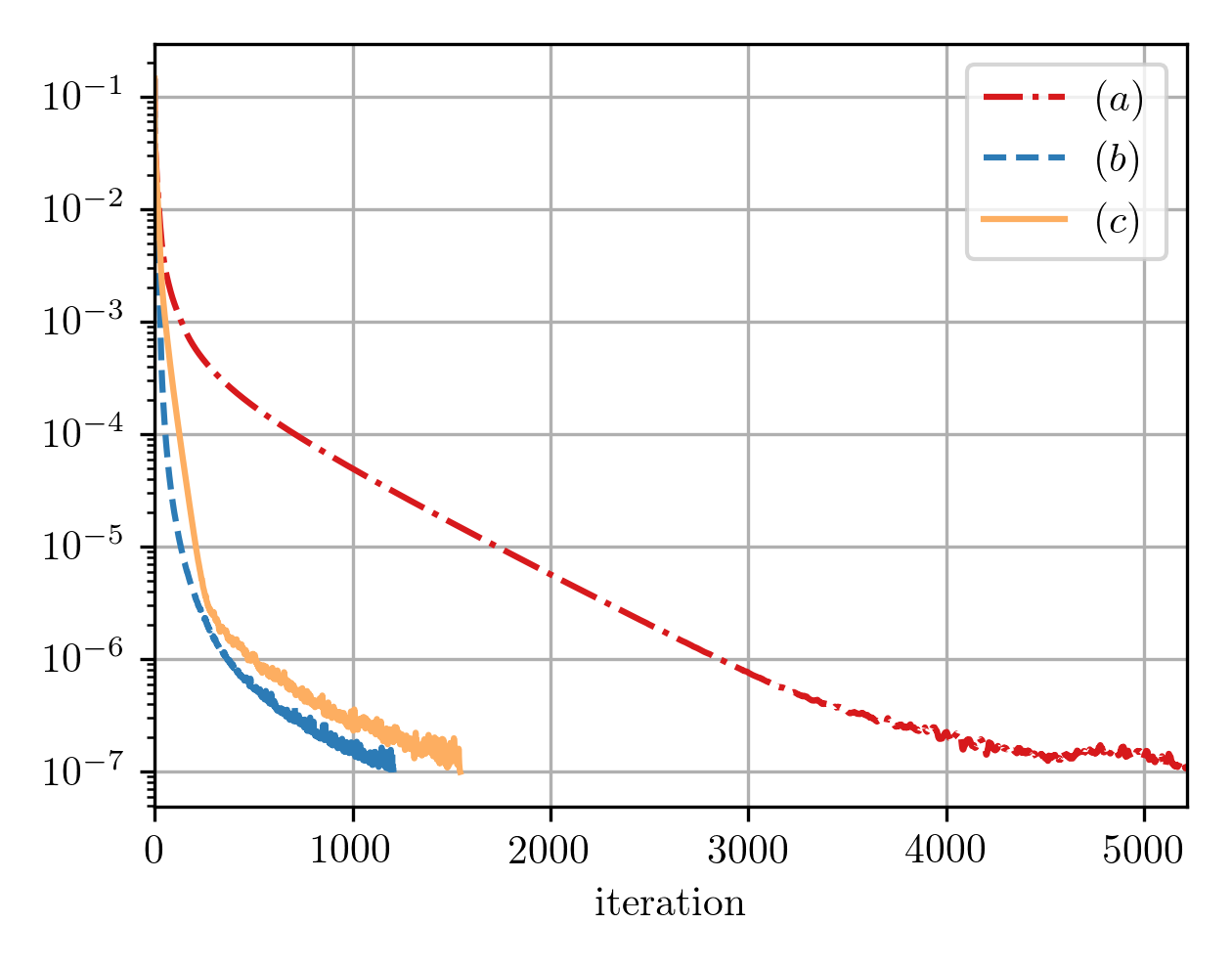}
\caption{Convergence of the proximal splitting algorithm measured by $\|\sigma_{n+1}-\sigma_n\|_{L^2(
\Omega)}$ for non-convex domain test without regularization (a); with the $H^1$ regularization and $\alpha = 0.002$ (b); with the $L^2$ regularization and $\alpha = 0.002$ (c).}
\label{fig:Zconv}
\end{figure}

\clearpage

\appendix

\section{Proof of Theorem \ref{th:convergence}}

Applying Theorem 2.16 in \cite{lavenant2019unconditional}, in order to prove Theorem \ref{th:convergence} it is sufficient to check that the conditions listed in definition 2.9 of \cite{lavenant2019unconditional} are verified. Such conditions translated to our finite element settings are listed in Proposition \ref{prop:properties} below. 

From now on, we assume $r=1$, and $X_h$ stands for $X_h^1$. We also denote by $\mc{I}$ is the standard nodal interpolant onto $X_h$, defined element by element.

First of all, we introduce some notation and list some technical results \cite{douglas1974stability}. Denote by $P_{X_h}$ and $P_{V_h}$ the $L^2$ projections  onto $X_h$ and $V_h$, respectively. Then,
\[
\|P_{X_h} \varphi \|_{L^p} \leq C \| \varphi \|_{L^p} \,,\quad \forall \varphi \in L^p \,,\,  1\leq p\leq \infty\,,
\]
and moreover $\forall\,  T\in \mc{T}_h$
\[
\|\varphi - P_{X_h} \varphi \|_{L^p(T)} \leq C h_T \|\nabla \varphi \|_{L^p(T)} \,,\quad \forall \varphi \in W^{1,p}(T)\,,\,  1\leq p\leq \infty\,,
\]
where, with an abuse of notation, we have used $P_{X_h}$ to denote the $L^2$ projection onto $X_h(T)$.
These imply the following lemma.
\begin{lemma}\label{lem:interp}
Given the regularity assumption in \eqref{eq:Cmesh} on $\mc{T}_h$, we have 
\[
\| \mc{I} |P_{V_h} b|^2 \|_{L^\infty} \leq C \| b\|^2_{L^\infty}\,,
\]
for any $b\in L^\infty(D)$, and
\[
\| \mc{I} |P_{V_h} b|^2 -|b|^2 \|_{L^\infty} \leq C h | b|_{W^{1,\infty}}\| b\|_{L^{\infty}}\,,
\]
for any $b\in W^{1,\infty}(D)$.
\end{lemma}
\begin{proof}
For the first inequality, using standard inverse inequalities, we have 
\[
\begin{aligned}
\| \mc{I} |P_{V_h} b| ^2 \|_{L^\infty}& \leq \||P_{V_h} b| ^2  \|_{L^\infty}\\
& \leq C h^{-d} \||P_{V_h} b| ^2  \|_{L^1}\\
& = C h^{-d} \|P_{V_h} b  \|^2_{L^2} \\
& \leq C h^{-d} \|P_{(X_h)^d} b  \|^2_{L^2} \\
& \leq C \|P_{(X_h)^d} b  \|^2_{L^\infty} \\
& \leq C \|b  \|^2_{L^\infty}\,.
\end{aligned}
\]
For the second inequality , we observe that
\[
\| \mc{I} |P_{V_h} b|^2 -|b|^2 \|_{L^\infty} \leq\| \mc{I} |P_{V_h} b|^2 - \mc{I} |b|^2 \|_{L^\infty}  +\| \mc{I} |b|^2 -|b|^2 \|_{L^\infty} \,.
\]
The second term of the right-hand side is easy to control. For the first term, we have
\[
\begin{aligned}
\| \mc{I} |P_{V_h} b|^2 - \mc{I} |b|^2 \|_{L^\infty} & \leq \| |P_{V_h} b|^2 - |b|^2 \|_{L^\infty} \\& \leq \| |P_{V_h} b|^2 - |P_{(X_h)^d} b|^2 \|_{L^\infty} +\| |b|^2 - |P_{(X_h)^d} b|^2 \|_{L^\infty} \,.
\end{aligned}
\]
Again, the second term is easy to control. For the first tem, using the same reasoning as above,
\[
\begin{aligned}
\| |P_{V_h} b|^2 - |P_{(X_h)^d} b|^2 \|_{L^\infty} & \leq C h^{-d} \| |P_{V_h} b|^2 - |P_{(X_h)^d} b|^2 \|_{L^1}\\
& \leq C h^{-d} \sum_{i=1}^d \| (P_{V_h} b)_i^2 - (P_{X_h} b_i)^2 \|_{L^1}\\
& \leq C h^{-d} \sum_{i=1}^d \| (P_{V_h} b)_i - P_{X_h} b_i \|_{L^1} \|b\|_{L^\infty}\\
& \leq C h^{-\frac{d}{2}} \| P_{V_h} b - P_{X_h} b \|_{L^2} \|b\|_{L^\infty} \\
& \leq C h \|\nabla P_{X_h} b \|_{L^\infty} \|b\|_{L^\infty} \leq C h \|\nabla b \|_{L^\infty} \|b\|_{L^\infty} \,.
\end{aligned}
\]
\end{proof}

\modif{As mentioned in Section \ref{sec:femd}, there exist projection operators $\Pi_{Q_h}:L^2(D)\rightarrow V_h$ and $\Pi_{V_h}:\mathcal{V}_D \rightarrow Q_h$ commuting with the divergence operator, where $\mathcal{V}_D$ is a dense subset of $H(\mathrm{div};D)$. We pick these to be the canonical projections introduced in Section 5.2 of \cite{arnold2006finite}, and in particular $\Pi_{Q_h}$ as in equation \eqref{eq:PiQh}. 
 %We will need also some approximation properties for the projection operators $\Pi_{Q_h}$ and $\Pi_{V_h}$. 
Such operators verify the following approximation properties (see Theorem 5.3 in \cite{arnold2006finite}): for any $\varphi\in H^1(D)$ and $\eta \in H^1(D)^d$
\begin{equation}\label{eq:approx}
\| \Pi_{Q_h} \varphi - \varphi\|_{L^2(D)} \leq C h \|\varphi\|_{H^1(D)} \,, \quad \| \Pi_{V_h} \eta - \eta\|_{L^2(D)^d} \leq C h \|\eta\|_{H^1(D)^d}\,.
\end{equation}
Notice in particular that given the mesh regularity assumption \eqref{eq:Cmesh}, equation \eqref{eq:approx} is a standard property for $\Pi_{Q_h}$ as defined in equation \eqref{eq:PiQh}.}%, whereas for $\Pi_{V_h}$ (since this has not been defined explicitly) we refer to Theorem 5.3 in \cite{arnold2006finite}, for example.

Proposition \ref{prop:properties} below contains the properties needed for convergence: it can be seen as a specific instance of Definition 2.9 of \cite{lavenant2019unconditional}. Note that a few of the properties listed therein are omitted here because they are either unnecessary or true by construction in our setting. Note also that the sampling operators used in \cite{lavenant2019unconditional} are replaced here with the canonical projections $\Pi_{Q_h}$ and $\Pi_{V_h}$, where $\Pi_{Q_h}$ can be naturally extended to $\mc{M}(D)$ (see equation \eqref{eq:PiQh}) and $\Pi_{V_h}$ is considered to be defined on a dense subset of $\mc{M}(D)^{d}$. Moreover the reconstruction operators are simply the injection operators from $Q_h$ and $V_h$ to $\mc{M}(D)$ and $\mc{M}(D)^d$, respectively. Finally, we define for any $(\rho,b)\in \mc{M}(D) \times C(D;\mathbb{R}^d)$
\[
A^*(\rho,b) \coloneqq \int_D \frac{|b|^2}{2} \rho\,,
\]
so that if $(\rho ,m) \in \mc{M}_+(D) \times \mc{M}(D)^d$ then 
\[
A(\rho,m) = \sup_{b \in  C(D;\mathbb{R}^d)} \langle m,b\rangle -A^*(\rho,b); 
\]
and for any $(\rho,b) \in Q_h \times V_h$,
 \[
A^*_h(\rho,b) \coloneqq \sup_{m\in V_h} \langle m,b \rangle -A_h(\rho,m)\,.
\]

\begin{proposition}\label{prop:properties}
The following properties hold:
\begin{enumerate}
\item For any $\rho \in \mc{M}_+(D)$, $\Pi_{Q_h} \rho \rightarrow \rho$ as $h\rightarrow 0$ weakly in $\mc{M}(D)$.
\item Let $B\subset (C^1(D))^d$ a bounded subset. Then there exists a constant $\varepsilon_h$ tending to $0$ as $h\rightarrow 0$ such that for any $b\in B$ and $\rho \in Q_h$
\[
A^*_h(\rho,P_{V_h} b) \leq A^*(\rho,b) +\epsilon_h \| \rho \|\,,
\]
where $P_{V_h}$ denotes the $L^2$ projection onto $V_h$. Moreover there exists a constant $C\geq 1$ such that for any $b\in C(D)^d$, there holds
\[
A^*_h(\rho,P_{V_h} b) \leq \frac{C}{2} \| \rho \| \| b \|^2_{L^\infty}\,.
\]
\item Let $B\subset C^0(D)\cap H^1(D)$ a bounded subset such that for all $\rho \in B$ there holds $\rho>C>0$, and let $B' \subset (C^0(D) \cap H^1(D))^d$ a bounded subset. There exists a constant $\varepsilon_h$ tending to $0$ as $h\rightarrow 0$ such that, given $(\rho,m)\in \mc{M}(D)^{d+1}$ such that $\rho$ has density in $B$ and $m$ in $B'$, then
\[
A_{h}(\Pi_{Q_h} \rho, \Pi_{V_h} m) \leq A(\rho,m) + \varepsilon_h.
\]
\item There exists $\varepsilon_h$ tending to $0$ as $h\rightarrow 0$ and a continuous function $\omega$ satisfying $\omega(0)=0$ such that: for any $x,y\in D$ and $h>0$ there exists $\rho \in Q_h^+$ and $m_1,m_2 \in V_h$ satisfying
\begin{equation}\label{eq:stabilitydiracs}
\left \{
\begin{array}{l}
\mr{div}\, m_1 = \rho - \Pi_{Q_h} (\delta_x) \\
\mr{div}\, m_2 = \rho - \Pi_{Q_h} (\delta_y)
\end{array} \right.
\quad \text{and} \quad A_h(\rho,m_i) \leq \omega(|x-y|)+\varepsilon_h\,,\forall i\in\{1,2\}.
\end{equation}  
\end{enumerate}
\end{proposition}

\begin{remark} 
In \cite{lavenant2019unconditional} point (3) of Proposition \ref{prop:properties} is stated with $B$ and $B'$ bounded subsets of $C^1(D)$ and $C^1(D)^d$, respectively. The condition we require here is stronger, but it is needed since we considered a convex polytope domain rather than a domain with a smooth boundary as in \cite{lavenant2019unconditional}. As a matter of fact, in \cite{lavenant2019unconditional} one applies the condition (3) on a regularized measure $(\tilde{\rho},\tilde{m})\in \mc{M}(D)^{d+1}$ obtained by convolution with the heat kernel and by solving an appropriate elliptic problem (see proposition 3.2 in \cite{lavenant2019unconditional}). For a convex polytope domain this procedure yields a couple $(\tilde{\rho},\tilde{m})$ with densities which are not $C^\infty$ given the singularities of the boundary. By classical elliptic regularity estimates on non-smooth domains (e.g., \cite{stampacchia1960problemi} and \cite{lieberman1988oblique}), the regularity we require in condition (3) is however sufficient for the proof in  \cite{lavenant2019unconditional} to apply without changes.
\end{remark}

\begin{proof}
The first point is immediate from the definition of $\Pi_{Q_h}$ in equation \eqref{eq:PiQh}. For (2), we observe that
\[
A_h(\rho,m) = \sup_{b\in X_h} \langle m,b\rangle -\frac{1}{2}\langle \rho, \mc{I} |b|^2  \rangle\,,
\]
where we recall that $\mc{I}$ is the standard element-wise nodal interpolant onto $X_h$. In fact, for any $b \in (X_h)^d$, we have $b^2 \leq \mc{I}|b|^2$, and therefore when $\rho \geq 0$ we can ``saturate" the constraint setting $a = -\mc{I}|b|^2/2$. On the other hand if $\rho<0$ on some element both sides of the equality are $+\infty$. For $(\rho,b,m) \in Q_h \times V_h\times V_h$ define
\[
{A}^{*}_{\mc{I},h} (\rho,b) \coloneqq  \frac{1}{2}\langle \rho, \mathcal{I} |b|^2  \rangle ,\quad \bar{A}_{\mc{I},h}(\rho,m) \coloneqq \sup_{b\in V_h} \langle m,b\rangle -{A}^{*}_{\mc{I},h} (\rho,b).
\]
Then, since when $\rho<0$ on some element $A^*_h(\rho,b)= -\infty$,
\[
A_h(\rho,m) \geq \bar{A}_{\mc{I},h}(\rho,m) ,\quad A^*_h(\rho,b) \leq {\bar{A}}^{*}_{\mc{I},h}(\rho,b) \leq {A}^{*}_{\mc{I},h}(\rho,b), 
\]
and we can prove (2) for $ A^*_{\mc{I},h}$. In particular, we have
\[
A^*_{\mc{I},h}(\rho,P_{V_h}b) \leq A^*(\rho,b)+ \frac{1}{2} \| \mc{I} |P_{V_h} b|^2 - |b|^2 \|_{L^\infty} \|\rho\|,
\]
and we obtain the result applying Lemma \ref{lem:interp}. Using again Lemma \ref{lem:interp}, we easily obtain the second bound as well.

For point (3), observe first that $A_h(\Pi_{Q_h}\rho, \Pi_{V_h} m)\leq A(\Pi_{Q_h}\rho, \Pi_{V_h} m)$ by definition. Then given the assumption on $\rho$ and $m$ we can simply write
\[
\begin{aligned}
A_h(\Pi_{Q_h}\rho, \Pi_{V_h} m) - A(\rho,m) & \leq \int_D \frac{|\Pi_{V_h} m|^2}{2 \Pi_{Q_h}\rho} - \frac{|m|^2}{2\rho} \\
& \leq \frac{1}{2} \int_D | \frac{|\Pi_{V_h} m|^2-|m|^2}{\Pi_{Q_h}\rho}| + | \frac{| m|^2}{ \Pi_{Q_h}\rho} - \frac{|m|^2}{\rho}| \\
& \leq C ( \| \Pi_{Q_h} \rho - \rho \|_{L^2} +  \| |\Pi_{V_h} m|^2-|m|^2 \|_{L^1})\,,  
\end{aligned}
\]
where the constant $C$ depends on the uniform lower bound on $\rho$ and on the $L^\infty$ norm of $|m|$. We conclude using Cauchy–Schwarz inequality on the second term and then equation \eqref{eq:approx}.

For the last point, we will establish a connection between our scheme and the one proposed by Gladbach, Kopfer and Maas \cite{gladbach2018scaling} and then use propoperty \eqref{eq:stabilitydiracs} for this scheme which was proved in \cite{lavenant2019unconditional}. \modif{We will consider only the case of a simplicial mesh and $V_h = \mc{RT}_0$ (which covers also the case of $V_h = \mc{BDM}_1$, since $\mc{RT}_0\subset \mc{BDM}_1$). The quadrilateral case with $V_h = \mc{RT}_{[0]}$ can be dealt with in a completely analogous  way.} %to the case of the space $\mc{RT}_0$ discussed below.% with $\mc{RT}_{[0]}$ in the discussion below. 

First, we introduce some notation. For each $T\in\mc{T}_h$, let $\mc{T}_{h,T}$ be the set of neighbouring elements $L\in \mc{T}_h$ such that $f_{T,L}\coloneqq \overline{T}\cap \overline{L} \neq \emptyset$, which we assume to be oriented. Define by $\mc{F}_h$ the set of $(d-1)$-dimensional facets in the triangulation. Let $T,L \in \mc{T}_h$ be neighbouring elements, we denote by $\varphi_{T,L}\in \mc{RT}_0$ the canonical basis function associated with the oriented facet $f_{T,L}$. Then, any $m\in \mc{RT}_0$ can be written as
\[
m = \sum_{f_{T,L} \in \mc{F}_h} m_{T,L} \varphi_{T,L}\,,
\]
where $m_{T,L}$ is the flux of $m$ on the oriented facet $f_{T,L}$. In other words we can identify functions in $(\rho,m) \in Q_h \times \mc{RT}_0$ with their finite volume representation $\{\rho_T, m_{T,L}\}_{T,L}$. Then, we can interpret the action for the finite volume scheme  \cite{gladbach2018scaling}, which we denote by $\hatt{A}_h^{FV}(\rho,m)$, as a function on $Q_h\times \mc{RT}_0$. This is given by the following expression
\[
\hatt{A}_h^{FV}(\rho,m) \coloneqq \sum_{f_{T,L}\in \mc{F}_h} \frac{m_{T,L}^2}{2\theta(\rho_T,\rho_{L})} |f_{T,L}||x_T-x_L|\,,
\]
where $\theta:\mathbb{R}^+\times \mathbb{R}^+ \rightarrow \mathbb{R}^+$ is an appropriate function (see \cite{gladbach2018scaling}) which we take to be the harmonic mean.

Now, in order to construct $\rho \in Q_h^+$ and $m_1,m_2 \in \mc{RT}_0$ satisfying \eqref{eq:stabilitydiracs}, we use the same construction as in \cite{lavenant2019unconditional} for the finite volume scheme, and interpolate these to the spaces $\mc{RT}_0$ and $Q_h^+$ to obtain $\rho$, $m_1$ and $m_2$ satisfying
\[
\left \{
\begin{array}{l}
\mr{div}\, m_1 = \rho - \Pi_{Q_h} (\delta_x) \,,\\
\mr{div}\, m_2 = \rho - \Pi_{Q_h} (\delta_y)\,.
\end{array} \right.
\] 
In particular the support of $\rho$, $m_1$ and $m_2$ is a chain of neighbouring elements $T_1,\ldots, T_N$.
To prove the bound on the action, we observe that $A_h(\rho,m_i)\leq A(\rho,m_i)$. Then, we only need to bound $A(\rho,m_i)$ by the action of the finite-volume scheme  $\hatt{A}_h^{FV}(\rho,m_i)$, since $A_h^{FV}$ satisfies the desired inequality thanks to the regularity assumption \eqref{eq:Cmesh} on the mesh  \cite{lavenant2019unconditional}. 
By the regularity assumption on the triangulation, we can assume
\[
\int_{T\cup L} |\varphi_{T,L}|^2 \, \ed x \leq C  |f_{T,L}||x_T-x_L|
\]
uniformly. Then, by explicit calculations we obtain $A(\rho,m_i) \leq  C\hatt{A}_h^{FV}(\rho,m_i)$ and we are done.

\end{proof}

\section*{Acknowledgements}
The work of A. Natale was supported by the European Research Council (ERC project NORIA).
G. Todeschi acknowledges that this project has received funding
from the European Union’s Horizon 2020 research and innovation
programme under the Marie Skłodowska-Curie grant agreement No 754362.
\begin{center}
	\includegraphics[width=0.15\textwidth]{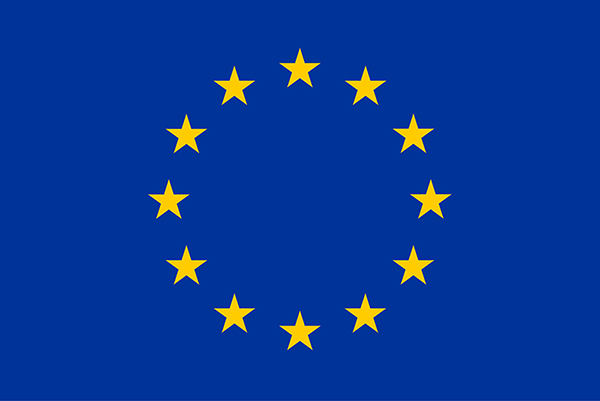}
\end{center}

\bibliographystyle{plain}
\bibliography{refs}   

\begin{thebibliography}{10}

\bibitem{Arnold14}
Douglas~N Arnold, Daniele Boffi, and Francesca Bonizzoni.
\newblock Finite element differential forms on curvilinear cubic meshes and
  their approximation properties.
\newblock {\em Numer. Math.}, 2014.
\newblock arXiv:1204.2595.

\bibitem{arnold2006finite}
Douglas~N Arnold, Richard~S Falk, and Ragnar Winther.
\newblock Finite element exterior calculus, homological techniques, and
  applications.
\newblock {\em Acta numerica}, 15:1--155, 2006.

\bibitem{petsc-user-ref}
Satish Balay, Shrirang Abhyankar, Mark~F. Adams, Jed Brown, Peter Brune, Kris
  Buschelman, Lisandro Dalcin, Victor Eijkhout, William~D. Gropp, Dmitry
  Karpeyev, Dinesh Kaushik, Matthew~G. Knepley, Dave~A. May, Lois~Curfman
  McInnes, Richard~Tran Mills, Todd Munson, Karl Rupp, Patrick Sanan, Barry~F.
  Smith, Stefano Zampini, Hong Zhang, and Hong Zhang.
\newblock {PETS}c users manual.
\newblock Technical Report ANL-95/11 - Revision 3.11, Argonne National
  Laboratory, 2019.

\bibitem{petsc-efficient}
Satish Balay, William~D. Gropp, Lois~Curfman McInnes, and Barry~F. Smith.
\newblock Efficient management of parallelism in object oriented numerical
  software libraries.
\newblock In E.~Arge, A.~M. Bruaset, and H.~P. Langtangen, editors, {\em Modern
  Software Tools in Scientific Computing}, pages 163--202. Birkh{\"{a}}user
  Press, 1997.

\bibitem{benamou2000computational}
Jean-David Benamou and Yann Brenier.
\newblock A computational fluid mechanics solution to the {M}onge-{K}antorovich
  mass transfer problem.
\newblock {\em Numerische Mathematik}, 84(3):375--393, 2000.

\bibitem{benamou2001mixed}
Jean-David Benamou and Yann Brenier.
\newblock {Mixed L2-Wasserstein optimal mapping between prescribed density
  functions}.
\newblock {\em Journal of Optimization Theory and Applications},
  111(2):255--271, 2001.

\bibitem{benamou2015augmented}
Jean-David Benamou and Guillaume Carlier.
\newblock Augmented lagrangian methods for transport optimization, mean field
  games and degenerate elliptic equations.
\newblock {\em Journal of Optimization Theory and Applications}, 167(1):1--26,
  2015.

\bibitem{benamou2016augmented}
Jean-David Benamou, Guillaume Carlier, and Maxime Laborde.
\newblock An augmented lagrangian approach to wasserstein gradient flows and
  applications.
\newblock {\em ESAIM: Proceedings and Surveys}, 54:1--17, 2016.

\bibitem{Bercea2016}
Gheorghe{-}Teodor Bercea, Andrew T.~T. McRae, David~A. Ham, Lawrence Mitchell,
  Florian Rathgeber, Luigi Nardi, Fabio Luporini, and Paul H.~J. Kelly.
\newblock A structure-exploiting numbering algorithm for finite elements on
  extruded meshes, and its performance evaluation in firedrake.
\newblock {\em Geoscientific Model Development}, 9(10):3803--3815, 2016.

\bibitem{boffi2013mixed}
Daniele Boffi, Franco Brezzi, Michel Fortin, et~al.
\newblock {\em Mixed finite element methods and applications}, volume~44.
\newblock Springer, 2013.

\bibitem{carrillo2019primal}
Jose~A Carrillo, Katy Craig, Li~Wang, and Chaozhen Wei.
\newblock Primal dual methods for wasserstein gradient flows.
\newblock {\em arXiv preprint arXiv:1901.08081}, 2019.

\bibitem{chambolle2016introduction}
Antonin Chambolle and Thomas Pock.
\newblock An introduction to continuous optimization for imaging.
\newblock {\em Acta Numerica}, 25:161--319, 2016.

\bibitem{douglas1974stability}
Jim Douglas, Todd Dupont, and Lars Wahlbin.
\newblock The stability in {$L^q$} of the {$L^2$}-projection into finite
  element function spaces.
\newblock {\em Numerische Mathematik}, 23(3):193--197, 1974.

\bibitem{erbar2020computation}
Matthias Erbar, Martin Rumpf, Bernhard Schmitzer, and Stefan Simon.
\newblock Computation of optimal transport on discrete metric measure spaces.
\newblock {\em Numerische Mathematik}, 144(1):157--200, 2020.

\bibitem{gallouet2019unbalanced}
Thomas Gallou{\"e}t, Maxime Laborde, and Leonard Monsaingeon.
\newblock An unbalanced optimal transport splitting scheme for general
  advection-reaction-diffusion problems.
\newblock {\em ESAIM: Control, Optimisation and Calculus of Variations}, 25:8,
  2019.

\bibitem{gladbach2018scaling}
Peter Gladbach, Eva Kopfer, and Jan Maas.
\newblock Scaling limits of discrete optimal transport.
\newblock {\em arXiv preprint arXiv:1809.01092}, 2018.

\bibitem{guittet2003time}
Kevin Guittet.
\newblock On the time-continuous mass transport problem and its approximation
  by augmented lagrangian techniques.
\newblock {\em SIAM Journal on Numerical Analysis}, 41(1):382--399, 2003.

\bibitem{henry2019primal}
Morgane Henry, Emmanuel Maitre, and Val{\'e}rie Perrier.
\newblock Primal-dual formulation of the dynamic optimal transport using
  helmholtz-hodge decomposition.
\newblock 2019.

\bibitem{hug2015multi}
Romain Hug, Emmanuel Maitre, and Nicolas Papadakis.
\newblock Multi-physics optimal transportation and image interpolation.
\newblock {\em ESAIM: Mathematical Modelling and Numerical Analysis},
  49(6):1671--1692, 2015.

\bibitem{hug2017convergence}
Romain Hug, Emmanuel Maitre, and Nicolas Papadakis.
\newblock On the convergence of augmented lagrangian method for optimal
  transport between nonnegative densities.
\newblock 2017.

\bibitem{igbida2017augmented}
Noureddine Igbida and Van~Thanh Nguyen.
\newblock {Augmented Lagrangian Method for Optimal Partial Transportation}.
\newblock {\em IMA Journal of Numerical Analysis}, 38(1):156--183, 03 2017.

\bibitem{lavenant2019unconditional}
Hugo Lavenant.
\newblock Unconditional convergence for discretizations of dynamical optimal
  transport.
\newblock {\em arXiv preprint arXiv:1909.08790}, 2019.

\bibitem{lavenant2018dynamical}
Hugo Lavenant, Sebastian Claici, Edward Chien, and Justin Solomon.
\newblock Dynamical optimal transport on discrete surfaces.
\newblock {\em ACM Transactions on Graphics (TOG)}, 37(6):1--16, 2018.

\bibitem{li2018computations}
Wuchen Li, Penghang Yin, and Stanley Osher.
\newblock Computations of optimal transport distance with fisher information
  regularization.
\newblock {\em Journal of Scientific Computing}, 75(3):1581--1595, 2018.

\bibitem{lieberman1988oblique}
Gary~M Lieberman.
\newblock {Oblique derivative problems in Lipschitz domains. II. Discontinuous
  boundary data}.
\newblock {\em J. reine angew. Math}, 389:1--21, 1988.

\bibitem{McRae2016}
Andrew T.~T. McRae, Gheorghe-Teodor Bercea, Lawrence Mitchell, David~A. Ham,
  and Colin~J. Cotter.
\newblock Automated generation and symbolic manipulation of tensor product
  finite elements.
\newblock {\em SIAM Journal on Scientific Computing}, 38(5):S25--S47, 2016.

\bibitem{papadakis2014optimal}
Nicolas Papadakis, Gabriel Peyr{\'e}, and Edouard Oudet.
\newblock Optimal transport with proximal splitting.
\newblock {\em SIAM Journal on Imaging Sciences}, 7(1):212--238, 2014.

\bibitem{pock2009algorithm}
Thomas Pock, Daniel Cremers, Horst Bischof, and Antonin Chambolle.
\newblock An algorithm for minimizing the mumford-shah functional.
\newblock In {\em 2009 IEEE 12th International Conference on Computer Vision},
  pages 1133--1140. IEEE, 2009.

\bibitem{Rathgeber2016}
Florian Rathgeber, David~A. Ham, Lawrence Mitchell, Michael Lange, Fabio
  Luporini, Andrew T.~T. McRae, Gheorghe-Teodor Bercea, Graham~R. Markall, and
  Paul H.~J. Kelly.
\newblock Firedrake: automating the finite element method by composing
  abstractions.
\newblock {\em ACM Trans. Math. Softw.}, 43(3):24:1--24:27, 2016.

\bibitem{santambrogio2015optimal}
Filippo Santambrogio.
\newblock Optimal transport for applied mathematicians.
\newblock {\em Birk{\"a}user, NY}, pages 99--102, 2015.

\bibitem{stampacchia1960problemi}
Guido Stampacchia.
\newblock Problemi al contorno ellittici, con dati discontinui, dotati di
  soluzioni h{\"o}lderiane.
\newblock {\em Annali di Matematica pura ed applicata}, 51(1):1--37, 1960.

\end{thebibliography}

\end{document}